\documentclass[11pt,a4paper,leqno]{amsart}
\usepackage{mathrsfs}
\usepackage{url}

\usepackage[margin=3cm]{geometry}

\usepackage{xspace}
\usepackage{graphics}
\usepackage{amsfonts,amssymb}
\usepackage{amsthm}
\usepackage[all]{xy}
\usepackage{stmaryrd}
\usepackage{color}
\usepackage{titletoc}

\newcommand{\be}{\begin{otherlanguage}{english}}
\newcommand{\ee}{\end{otherlanguage}}

\theoremstyle{definition}
\newtheorem{defn}{Definition}[section]
\newtheorem{rem}[defn]{Remark}
\theoremstyle{plain}
\newtheorem{lem}[defn]{Lemma}
\newtheorem{prop}[defn]{Proposition}
\newtheorem{thm}[defn]{Theorem}
\newtheorem{cor}[defn]{Corollary}

\newtheorem*{thm*}{Theorem}
\newtheorem*{prop*}{Proposition}
\newtheorem*{clm*}{Claim}
\theoremstyle{remark}
\newtheorem{exem}[defn]{\textbf{Example}}
\numberwithin{equation}{section}

\newcommand{\beq}{\begin{equation}}
\newcommand{\eeq}{\end{equation}}

\newcommand*{\rom}[1]{\uppercase\expandafter{\romannumeral #1\relax}}

\begin{document}
\title[]{Some results of Hamiltonian homeomorphisms on closed aspherical surfaces}

\author{Jian Wang}
\address{Max Planck Institute for Mathematics in
the Sciences, Inselstra{\ss}e 22, D-04103 Leipzig,
Germany\newline\indent Chern Institute of Mathematics and Key Lab of
Pure Mathematics and Combinatorics
 of Ministry of Education, Nankai University, Tianjin 300071,
P.R.China}\email{jianwang@mis.mpg.de;\indent wangjian@nankai.edu.cn}




\date{Oct. 21, 2016}
\maketitle
\begin{abstract}On closed symplectically aspherical manifolds, Schwarz proved a classical
result that the action function of a nontrivial Hamiltonian
diffeomorphism is not constant by using Floer homology. In this
article, we generalize Schwarz's theorem to the $C^0$-case on closed
aspherical surfaces. Our methods involve the theory of transverse
foliations for dynamical systems of surfaces inspired by Le Calvez
and its recent progresses. As an application, we prove that the
contractible fixed points set (and consequently the fixed points
set) of a nontrivial Hamiltonian homeomorphism is not connected.
Furthermore, we obtain that the growth of the action width of a
Hamiltonian homeomorphism increases at least linearly, and
 that the group of Hamiltonian homeomorphisms of $\mathbb{T}^2$
and the group of area preserving homeomorphisms isotopic to the
identity of $\Sigma_g$ ($g>1$) are torsion free, where $\Sigma_g$ is
a closed orientated  surface with genus $g$. Finally, we will show
how the $C^1$-Zimmer's conjecture on surfaces deduces from
$C^0$-Schwarz's theorem.

\vspace{3mm}
\begin{flushleft}
\end{flushleft}

\vspace{3mm}
\begin{flushleft}
\end{flushleft}

\end{abstract}
\bigskip

\section{Introduction}
\subsection{Background}The famous Gromov-Eliashberg Theorem, that the group of
symplectic diffeomorphisms is $C^0$-closed in the full group of
diffeomorphisms, makes us interested in defining a symplectic
homeomorphism as a homeomorphism which is a $C^0$-limit of
symplectic diffeomorphisms. This becomes a central theme of what is
now called ``$C^0$-symplectic topology". There is a family of
problems in symplectic topology that are interesting to be extended
to the continuous analogs of classical smooth objects of the
symplectic world (see, e.g.,
\cite{BS,BHS,Hum,HLS,HLecS,P1,P3,OM,Se,V06}). In the theme of
$C^0$-symplectic topology, there are many questions still open,
e.g., the $C^0$-flux conjecture (see \cite{B14,LMP98,On06}) and the
simplicity of the group of Hamiltonian homeomorphisms of surfaces
(see \cite{Fathi,OM}).

Another noteworthy rigidity phenomenon is the Zimmer program which
attracted many mathematicians to work (see, e.g.,
\cite{BFH,Fish,F2,FS,G,P1,P}). A central conjecture of Zimmer
program \cite{Zim} predicts that lattices in simple Lie groups of
rank $n$ do not act volume-preserving faithfully on compact
manifolds of dimension less than $n$.
\smallskip

Suppose that $(M,\omega)$ is a symplectic manifold. Let
$I=(F_t)_{t\in \mathbb{R}}$ be a Hamiltonian flow on $M$ with
$F_0=\mathrm{Id}_M$ and $F_1=F$. When $M$ is compact, among the
properties of $F$, one may notice that it preserves the volume form
$\omega^n=\omega\wedge\cdots\wedge\omega$ and that the ``rotation
vector'' $\rho_{M,I}(\mu)$ (see Section \ref{subsec:rotation
vector}) of the finite measure $\mu$ induced by $\omega^n$ vanishes.
Let $M$ be a closed oriented surface with genus $g\geq 1$. In this
case, $M$ is a closed aspherical surface with the property
$\pi_2(M)=0$. Let $I=(F_t)_{t\in [0,1]}$ be an identity isotopy on
$M$, that is, $I$ is a continuous path in $\mathrm{Homeo}(M)$ with
$F_0=\mathrm{Id}_M$. We suppose that its time-one map $F$ preserves
the measure $\mu$ induced by $\omega$. It is well known that the
condition  $\rho_{M,I}(\mu)=0$ is equivalent to the fact that the
homeomorphism $F$ is in the $C^0$-closure of
$\mathrm{Ham}(M,\omega)$. In this sense, we call such $I$ a
\emph{Hamiltonian isotopy} and such $F$ a \emph{Hamiltonian
homeomorphism}. In this article, we carry out some foundational
studies of Hamiltonian homeomorphisms (and a more general notion) on
closed aspherical surfaces. We also show the link between Zimmer's
conjecture on surfaces and our work which is of independent
interest.

\smallskip

Let $(M,\omega)$ be a symplectic manifold with $\pi_2(M)=0$. Suppose
that $H : \mathbb{R}\times M \rightarrow \mathbb{R}$, one-periodic
in time, is the Hamiltonian function generating the flow $I$. Denote
by $\mathrm{Fix}_{\mathrm{Cont},I}(F)$ the set of contractible fixed
points of $F$, that is, $x\in\mathrm{Fix}_{\mathrm{Cont},I}(F)$ if
and only if $x$ is a fixed point of $F$ and the oriented loop $I(x):
t\mapsto F_t(x)$ defined on $[0,1]$ is contractible on $M$. The
classical action function is defined, up to an additive constant, on
$\mathrm{Fix}_{\mathrm{Cont},I}(F)$  as follows
\begin{equation}\label{eq:action functional}
\mathcal{A}_{H}(x)=\int_{D_x}\omega-\int_0^1
H(t,F_t(x))\,\mathrm{d}t,
\end{equation}
where $x\in \mathrm{Fix}_{\mathrm{Cont},I}(F)$ and $D_x\subset M$ is
any $2$-simplex with $\partial D_x=I(x)$.  The following deep result
\cite{SM} was proved  by using Floer homology with the real
filtration induced by the action function.

\begin{thm}[Schwarz] Let $(M,\omega)$ be a closed symplectic manifold with
$\pi_2(M)=0$. Let $I=(F_t)_{t\in \mathbb{R}}$ be a Hamiltonian flow
on $M$ with $F_0=\mathrm{Id}_M$ and $F_1=F$ generated by a
Hamiltonian function $H$. Assume that $F\neq \mathrm{Id}_M$. Then
there are $x,y\in\mathrm{Fix}_{\mathrm{Cont},I}(F)$ such that
$\mathcal{A}_{H}(x)\neq\mathcal{A}_{H}(y)$.\end{thm}

Let $M$ be a closed oriented surface with genus $g\geq 1$ and $F$ be
the time-one map of an identity isotopy $I$ on $M$. We denote by
$\mathrm{Homeo}(M)$ (resp. $\mathrm{Diff}(M)$, $\mathrm{Diff}^1(M)$)
the set of homeomorphisms (resp. diffeomorphisms,
$C^1$-diffeomorphisms) of $M$. Denote by $\mathcal {M}(F)$ the set
of Borel finite measures on $M$ that are invariant by $F$ and have
no atoms on $\mathrm{Fix}_{\mathrm{Cont},I}(F)$.  Through the
WB-property \cite{W2} (see Definition \ref{def:wb property and b
property} below),
 the classical action of Hamiltonian
diffeomorphism has been generalized
 to the case of Hamiltonian homeomorphism (and to more general
cases) \cite[page 86]{W2} (or see \cite{W3}):

\begin{thm}\label{thm:PW}
Let $F\in\mathrm{Homeo}(M)$ be the time-one map of an identity
isotopy $I$ on $M$. Suppose that $\mu\in\mathcal {M}(F)$ and
$\rho_{M,I}(\mu)=0$. In each of the following cases:
\begin{itemize}
  \item $F\in\mathrm{Diff}(M)$ (not necessarily $C^1$);
  \item $I$ satisfies the WB-property and the measure $\mu$ has full
  support;
  \item $I$ satisfies the WB-property and the measure $\mu$ is ergodic,
\end{itemize} an action function $L_\mu$ can be defined,
which generalizes the classical one given in Eq. \ref{eq:action
functional}.\end{thm}

The contributions of this paper can be summarized as following:

\begin{itemize}
  \item In the classical case, one can prove that the action
function is a constant on a connected set of contractible fixed
points by Sard's theorem. In each of the generalized cases given in
Theorem \ref{thm:PW}, we prove that this property still holds (see
Proposition \ref{prop:action is constant on the path connected
components}). Our method is purely topological.

  \item Given the generalized action function, one may ask whether
Schwarz's theorem is still true. We show in this article that it is
true in the second case of Theorem \ref{thm:PW} but no longer true
when the measure $\mu$ has no full support even if
$F\in\mathrm{Diff}(M)$ (see Theorem \ref{prop:F is not constant if
the contractible fixed points is finite}).
  \item As applications of Proposition \ref{prop:action is constant on the path connected
components} and Theorem \ref{prop:F is not constant if the
contractible fixed points is finite}, we obtain that the
contractible fixed points set (and consequently the fixed points
set) of a nontrivial Hamiltonian homeomorphism is not connected. We
 emphasize that this is merely a 2-dimension phenomenon.
Indeed, Buhovsky et al. 
\cite{BHS} have recently constructed a Hamiltonian homeomorphism
with a single fixed point on any closed symplectic manifold of
dimension at least four. However, in the classical $C^1$-case this
property always holds when $(M^{2n},\omega)$ ($n\geq1$) is a closed
symplectically aspherical  manifold by Schwarz's theorem. It seems
to us that one can not obtain the result in the $C^0$-case on
dimension two through $C^0$-approximation by Hamiltonian
diffeomorphisms.
  \item  We generalize Polterovich's result \cite{P}
on the growth of the action width to the $C^0$-case, based on which
we obtain that the groups $\mathrm{Hameo}(\mathbb{T}^2,\mu)$ and
$\mathrm{Homeo}_*(\Sigma_g,\mu)$ ($g>1$) (see below for the
notations) are torsion free, where $\mu$ is a measure with full
support.
\item We give an
alternative proof of the $C^1$-Zimmer's conjecture on surfaces when
the measure is a Borel finite measure with full support from
$C^0$-Schwarz's theorem.
\end{itemize}

This paper extends our unpublished manuscripts \cite{W2,W3} with
substantial additional contents. A preceding work of this paper
\cite{W3} is under review of a journal. Please refer to \cite{W2} or
\cite{W3} for further details and related results.

\subsection{Statement of results}\quad

Before stating our main results, let us recall the WB-property and
B-property.

Let $M$ be a surface homeomorphic to the complex plane $\mathbb{C}$
and let $I=(F_t)_{t\in[0,1]}$ be an identity isotopy on $M$. For
every two different fixed points $z$ and $z'$ of $F_1$, the
\emph{linking number} $i_I(z,z')\in \mathbb{Z}$ is the degree of the
map $\xi: \mathbf{S}^1\rightarrow \mathbf{S}^1$ defined by
\[\xi(e^{2i\pi t})=\frac{h\circ F_{t}(z')-h\circ F_{t}(z)}
{|h\circ F_{t}(z')-h\circ F_{t}(z)|}\,,\] where
$h:M\rightarrow\mathbb{C}$ is a homeomorphism. The linking number is
independent of $h$.\smallskip

Let $F$ be the time-one map of an identity isotopy
$I=(F_t)_{t\in[0,1]}$ on a closed oriented  surface $M$ of genus
$g\geq 1$, and $\widetilde{F}$ be the time-one map of the lifted
identity isotopy $\widetilde{I}=(\widetilde{F}_t)_{t\in[0,1]}$ of
$I$ on the universal cover $\widetilde{M}$ of $M$. Assume that $\pi:
\widetilde{M}\rightarrow M$ is the covering map. Denote by $\Delta$
(resp. $\widetilde{\Delta}$) the diagonal of
$\mathrm{Fix}_{\mathrm{Cont},I}(F)\times
\mathrm{Fix}_{\mathrm{Cont},I}(F)$ (resp.
$\mathrm{Fix}(\widetilde{F})\times \mathrm{Fix}(\widetilde{F})$).

We write by $\mathrm{Homeo}_*(M)$ the identity component of the
topological space of $\mathrm{Homeo}(M)$ for the compact-open
topology. When $g>1$, it is well known that the fundamental group
$\pi_1(\mathrm{Homeo}_*(M))$ is trivial \cite{H2}. It implies that
any two identity isotopies $I ,I' \subset \mathrm{Homeo}_*(M)$ with
fixed endpoints are homotopic. Hence, $I$ is unique up to homotopy,
which implies that $\widetilde{F}$ is uniquely defined and
independent of the choice of the isotopy from $\mathrm{Id}_M$ to
$F$. When $g=1$, $\widetilde{F}$ depends on the isotopy $I$ since
$\pi_1(\mathrm{Homeo}_*(M))\simeq\mathbb{Z}^2$ \cite{H1}.

Note that the universal cover $\widetilde{M}$ is homeomorphic to
$\mathbb{C}$. We define the \emph{linking number}
$i(\widetilde{F};\widetilde{z},\widetilde{z}\,')$ for each pair
$(\widetilde{z},\widetilde{z}\,')\in(\mathrm{Fix}(\widetilde{F})\times
\mathrm{Fix}(\widetilde{F}))\setminus\widetilde{\Delta}$ as
\begin{equation}\label{eq:linking number for two fixed points}
i(\widetilde{F};\widetilde{z},\widetilde{z}\,')=i_{\widetilde{I}}
(\widetilde{z},\widetilde{z}\,').
\end{equation}

\begin{defn}\label{def:wb property and b property} We say that
$I$ satisfies \emph{the weak boundedness property at
$\widetilde{a}\in\mathrm{Fix}(\widetilde{F})$}, written WB-property
at $\widetilde{a}$, if there exists a positive number
$N_{\widetilde{a}}$ such that
$|i(\widetilde{F};\widetilde{a},\widetilde{b})|\leq
N_{\widetilde{a}}$ for all $\widetilde{b}\in
\mathrm{Fix}(\widetilde{F})\setminus\{\widetilde{a}\}$. We say that
$I$ satisfies \emph{the weak boundedness property}, denoted
WB-property, if it satisfies the weak boundedness property at every
$\widetilde{a}\in \mathrm{Fix}(\widetilde{F})$. Let
$\widetilde{X}\subseteq\mathrm{Fix}(\widetilde{F})$. We say that $I$
satisfies \emph{the boundedness property on $\widetilde{X}$},
written B-property on $\widetilde{X}$, if there exists a positive
number $N_{\widetilde{X}}$ such that
$|i(\widetilde{F};\widetilde{a},\widetilde{b})|\leq
N_{\widetilde{X}}$ for all
$(\widetilde{a},\widetilde{b})\in\widetilde{X}\times
\mathrm{Fix}(\widetilde{F})$ with $\widetilde{a}\neq\widetilde{b}$.
We say that $I$ satisfies \emph{the boundedness property}, denoted
B-property, if $\widetilde{X}=\mathrm{Fix}(\widetilde{F})$.
\end{defn}

Obviously, the B-property implies the WB-property. It has been
proved that the WB-property is satisfied if $F\in\mathrm{Diff}(M)$
and that the B-property is satisfied if $F\in\mathrm{Diff}^1(M)$
\cite{W2}. Moreover, the set of all WB-property points of $I$, that
is, the set $$\{\widetilde{a}\in\mathrm{Fix}(\widetilde{F})\mid I
\mbox{ satisfies the WB-property at } \widetilde{a}\}$$ is shown
dense in $\mathrm{Fix}(\widetilde{F})$ \cite{Ler14}. In Lemma
\ref{lem:fps is connected and B property} below, we prove that $I$
satisfies the B-property if the number of the connected components
of $\mathrm{Fix}_{\mathrm{Cont},I}(F)$ is finite.
\smallskip

We say that a homeomorphism $F$ is \emph{$\mu$-symplectic} if
$\mu\in \mathcal {M}(F)$ has full support. An identity isotopy $I$
is \emph{$\mu$-Hamiltonian} if the time-one map $F$ is
$\mu$-symplectic and $\rho_{M,I}(\mu)=0$. A homeomorphism $F$ is
\emph{$\mu$-Hamiltonian} if there exists a $\mu$-Hamiltonian isotopy
$I$ such that the time-one map of $I$ is $F$. The main results
 of this article are summarized as follows.

\begin{prop}\label{prop:action is constant on the path connected
components} Under the hypotheses of Theorem \ref{thm:PW}, the action
function defined in Theorem \ref{thm:PW} is a constant on each
connected component of $\mathrm{Fix}_{\mathrm{Cont},I}(F)$.
\end{prop}

\begin{thm}\label{prop:F is not
constant if the contractible fixed points is finite} Let $F$ be the
time-one map of a $\mu$-Hamiltonian isotopy $I$. If $I$ satisfies
the WB-property and $F\neq\mathrm{Id}_{M}$, the action function
defined in Theorem \ref{thm:PW} is not constant.
\end{thm}
Theorem \ref{prop:F is not constant if the contractible fixed points
is finite} is a generalization of Schwarz's theorem  on closed
oriented surfaces. The main tools we use in its proof are the theory
of transverse foliations for dynamical systems of surfaces inspired
by Le Calvez \cite{P1,P3} and its recent progress
\cite{J}.\smallskip


Recall the classical version of Arnold conjecture for surface
homeomorphisms  due to Matsumoto \cite{Ma} (see also \cite{P1}): any
Hamiltonian homeomorphism has at least three contractible fixed
points (see Theorem \ref{thm:hamiltonian map has at least 3
contricible fixed points} below).

As a consequence of Proposition \ref{prop:action is constant on the
path connected components} and Theorem \ref{prop:F is not constant
if the contractible fixed points is finite}, we have the following
theorem:

\begin{thm}\label{thm:connected and Id}
Let $F$ be the time-one map of a $\mu$-Hamiltonian isotopy $I$. If
the set $\mathrm{Fix}_{\mathrm{Cont},I}(F)$ is connected, then $F$
must be $\mathrm{Id}_{M}$. In particular, if  $\mathrm{Fix}(F)$ is
connected, then $F$ must be $\mathrm{Id}_{M}$.
\end{thm}

\begin{proof}By Theorem \ref{thm:hamiltonian map has at
least 3 contricible fixed points},
$\mathrm{Fix}_{\mathrm{Cont},I}(F)\neq\emptyset$. Moreover, if the
set $\mathrm{Fix}_{\mathrm{Cont},I}(F)$ is connected, the isotopy
must satisfy the WB-property according to Lemma \ref{lem:fps is
connected and B property}. Therefore, the action function is well
defined by Theorem \ref{thm:PW}. The conclusion follows from
Proposition \ref{prop:action is constant on the path connected
components} and Theorem \ref{prop:F is not constant if the
contractible fixed points is finite}. Note that the connectedness of
$\mathrm{Fix}(F)$ implies that
$\mathrm{Fix}_{\mathrm{Cont},I}(F)=\mathrm{Fix}(F)$ because
$\mathrm{Fix}_{\mathrm{Cont},I}(F)$ is an open and closed subset of
$\mathrm{Fix}(F)$.
\end{proof}

If $F\neq\mathrm{Id}_{M}$, Theorem \ref{thm:connected and Id}
implies that the number of connected components of the set
$\mathrm{Fix}_{\mathrm{Cont},I}(F)$ is at least 2, which is optimal
by the following example.
\begin{exem}\label{ex:connected components of fixed point set of Hamilton homeomorphisms}
Let $\mu$ be the measure induced by the area form $\omega$ and $D$
be a topological closed disk on $M$. Up to a diffeomorphism, we may
suppose that $D$ is the closed unit Euclidean disk and that
$\omega|_{D}=\mathrm{d}x\wedge \mathrm{d}y$. Let us consider the
polar coordinate for $D$ with the center $z_0=(0,0)$. Consider the
following isotopy $(F_t)_{t\in[0,1]}$ on $M$ defined as follows
\begin{eqnarray*}
  F_t: D &\rightarrow& D \\
  (r,\theta)&\mapsto& (r,\theta+2\pi rt),
\end{eqnarray*}
and $F_t|_{M\setminus D}=\mathrm{Id}_{M\setminus D}$ for all
$t\in[0,1]$. Obviously, $\rho_{M,I}(\mu)=0$ and
$\mathrm{Fix}_{\mathrm{Cont},I}(F)$ has exactly two
  connected components: $\{z_0\}$ and $M\setminus
\mathrm{Int}(D)$, where $\mathrm{Int}(D)$ is the interior of $D$.
\end{exem}
 By Theorem \ref{thm:connected and
Id} and Theorem \ref{thm:hamiltonian map has at least 3 contricible
fixed points}, if $\mathrm{Fix}_{\mathrm{Cont},I}(F)$ has exactly
two connected components, its cardinality must be
infinite.\smallskip

Remark that a measure with full support is essential for Theorem
\ref{prop:F is not constant if the contractible fixed points is
finite} and Theorem \ref{thm:connected and Id}. Without such
condition, Theorem \ref{prop:F is not constant if the contractible
fixed points is finite} can not hold as illustrated by Example
\ref{exem:T2 and supp(u)notM} and \ref{exem:M and supp(u)notM}
(Section \ref{Appendix}). In the case where $M=\mathbb{T}^2$,
Example \ref{exem:T2 and supp(u)notM} is also a counter-example of
Theorem \ref{thm:connected and Id} if the measure is without full
support. When the genus of $M$ is more than two, one can choose an
identity isotopy on $M$ with exactly one contractible fixed point
$z$ (such isotopy exists by Lefschetz-Nielsen's formula) and the
Dirac measure $\delta_z$. \smallskip

Denote by $\mathscr{IS}_*(M)$ the group of all identity isotopies
$I=(F_t)_{t\in[0,1]}$ on $M$, where the composition is given by
Equation \ref{the product operate of the symplectic group of
isotopy} (Section \ref{sec:Identity isotopies})\footnote{\,Usually,
there is another definition of the composition as follows: for any
two identity isotopies $I=(F_t)_{t\in[0,1]},I'=(F_t')_{t\in[0,1]}\in
\mathscr{I}\mathscr{S}_*(M)$, we define $I\circ I'=(F_t\circ
F_t')_{t\in[0,1]}$. However, the two new identity isotopies by such
two definitions are homotopic with fixed extremities.}. We say that
two identity isotopies $(F_t)_{t\in[0,1]},(G_t)_{t\in[0,1]}\in
\mathscr{I}\mathscr{S}_*(M)$ are homotopic with fixed extremities if
$F_1=G_1$
  and there exists a continuous map $[0,1]^2\rightarrow\mathrm{Homeo}(M)$, $(t,s)\mapsto
H_{t,s}$ such that $H_{0,s}=\mathrm{Id}_{M}$, $H_{1,s}=F_1=G_1$,
$H_{t,0}=F_t$ and $H_{t,1}=G_t$.

Under the same hypotheses as Theorem \ref{thm:PW}, we define the
 \emph{action spectrum of $I$} (up to an additive
constant):
\begin{equation*}
\sigma(I)=\{L_\mu(z)\mid z\in
\mathrm{Fix}_{\mathrm{Cont},I}(F)\}\subset\mathbb{R},
\end{equation*}
and the following \emph{action width of $I$}:
\begin{equation*}
\mathrm{width}(I)=\sup_{x,y\in\sigma(I)}|x-y|.
\end{equation*}
It turns out that the action spectrum $\sigma(I)$ (and hence
$\mathrm{width}(I)$) is invariant by conjugation in
$\mathrm{Homeo}^+(M,\mu)$, where $\mathrm{Homeo}^+(M,\mu)$ is the
subgroup of $\mathrm{Homeo}(M)$ whose elements preserve the measure
$\mu$ and the orientation (see \cite[Corollary 4.6.14]{W2}).
Moreover, the action function $L_\mu$ only depends on the homotopic
class with fixed endpoints of $I$ (see Proposition
\ref{prop:invariant in homotopic sense} below), so do $\sigma(I)$
and $\mathrm{width}(I)$. Observing that
$\pi_1(\mathrm{Homeo}_*(\Sigma_g))\simeq\{0\}$ ($g>1$), and that
$\pi_1(\mathrm{Homeo}_*(\mathbb{T}^2))\simeq\mathbb{Z}^2$ and
$\rho_{\mathbb{T}^2,I}(\mu)=0$, we can simply write $\sigma(F)$
(resp. $\mathrm{width}(F)$) instead of $\sigma(I)$ (resp.
$\mathrm{width}(I)$).

For any $q\geq1$, we define an identity isotopy $I^q$ on $M$:
$I^{q}(z)=\prod_{k=0}^{q-1}I({F^k(z)})$ for $z\in M$ (see Equation
\ref{the product operate of the symplectic group of isotopy} in
Section \ref{sec:Identity isotopies} for details). Under the
hypotheses of Theorem \ref{thm:PW}, for every two distinct
contractible fixed points $a$ and $b$ of $F$, the following
iteration formula \cite[Corollary 4.7.3]{W2} holds:
$I_{\mu}(I^q;a,b)=qI_{\mu}(I;a,b)$ for all $q\geq1$, where
$I_{\mu}(I;a,b)=L_{\mu}(I;b)-L_{\mu}(I;a)$.\smallskip

We follow the conventions of Polterovich \cite{P}: given two
positive sequences $\{a_n\}$ and $\{b_n\}$, we write $a_n\succeq
b_n$ if there is $c>0$ such that $a_n\geq cb_n$ for all
$n\in\mathbb{N}$, and $a_n\sim b_n$ if $a_n\succeq b_n$ and
$a_n\preceq b_n$. Based on Theorem \ref{prop:F is not constant if
the contractible fixed points is finite} and the iteration formula
above, we have the following conclusion which is a generalization of
Proposition 2.6.\,A in \cite{P}.

\begin{prop}\label{prop:ggeq1 no torsion}Let $F$ be the
time-one map of a $\mu$-Hamiltonian isotopy $I$. If $I$ satisfies
the WB-property and $F\neq\mathrm{Id}_{M}$, then
$\mathrm{width}(F^n)\succeq n.$
\end{prop}

In fact, we define the action functions not only for the
``Hamiltonian case: $\rho_{M,I}(\mu)=0$'' (the double quotation
marks means that the isotopy $I$ is not truely Hamiltonian if the
measure $\mu$ has no full measure by the definition), but also for
the ``non-Hamiltonian case: $\rho_{M,I}(\mu)\neq0$'' (see Corollary
\ref{clm:I is coboundary} below for details). Recall that
$\widetilde{F}$ is the time-one map of the lifted identity isotopy
of $I$ to $\widetilde{M}$. In the non-Hamiltonian case, we define
the action function $l_\mu$ on the set $\mathrm{Fix}(\widetilde{F})$
(see Corollary \ref{clm:I is coboundary}). When $M$ is hyperbolic,
we have the following theorem which is a generalization of Theorem
2.1.\,C in \cite{P}.

\begin{thm}\label{cor:the symplectic action when M with genus bigger
1}Let $F\in \mathrm{Homeo}_*(M)\setminus\{\,\mathrm{Id}_{M}\}$ be
the time-one map of an identity isotopy $I$ on a closed oriented
surface $M$ with $g>1$. If $F$ is $\mu$-symplectic and $I$ satisfies
the WB-property, then the action function $l_\mu$ is not constant.
\end{thm}

Similar to $\sigma(I)$ and $\mathrm{width}(I)$, we can define the
following \emph{lifted action spectrum of $I$} (up to an additive
constant):
\begin{equation*}
\widetilde{\sigma}(I)=\{l_\mu(z)\mid z\in
\mathrm{Fix}(\widetilde{F})\}\subset\mathbb{R},
\end{equation*}
and \emph{lifted action width of $I$}:
\begin{equation*}
\widetilde{\mathrm{width}}(I)=\sup_{x,y\in\widetilde{\sigma}(I)}|x-y|.
\end{equation*}
As before, $\widetilde{\sigma}(I)$ and
$\widetilde{\mathrm{width}}(I)$ are invariant by conjugation in
$\mathrm{Homeo}^+(M,\mu)$ and merely depend on the homotopic class
with fixed endpoints of $I$. When $M=\Sigma_g$ ($g>1$), we can
simply write $\widetilde{\sigma}(F)$ (resp.
$\widetilde{\mathrm{width}}(F)$) instead of $\widetilde{\sigma}(I)$
(resp. $\widetilde{\mathrm{width}}(I)$) since
$\pi_1(\mathrm{Homeo}_*(\Sigma_g))$ is trivial.\smallskip

For each two distinct fixed points $\widetilde{a}$ and
$\widetilde{b}$ of $\widetilde{F}$, a similar iteration formula
holds (see \cite[Corollary 4.6.12]{W2}):
$i_{\mu}(I^q;\widetilde{a},\widetilde{b})=qi_{\mu}(I;\widetilde{a},\widetilde{b})$
for all $q\geq1$, where
$i_{\mu}(I;\widetilde{a},\widetilde{b})=l_{\mu}(I;\widetilde{b})-l_{\mu}(I;\widetilde{a})$.
By Theorem \ref{cor:the symplectic action when M with genus bigger
1} and this iteration formula, we also have

\begin{prop}\label{prop:g>1 no torsion}
Let $F\in \mathrm{Homeo}_*(M)\setminus\{\,\mathrm{Id}_{M}\}$ be the
time-one map of an identity isotopy $I$ on a closed oriented surface
$M$ with $g>1$. If $F$ is $\mu$-symplectic and $I$ satisfies the
WB-property, then $\widetilde{\mathrm{width}}(F^n)\succeq n.$
\end{prop}
\smallskip

We fix a Borel finite measure $\mu$ which has a full support and has
no atoms on $M$ (e.g., the measure $\mu$ induced by the area form
$\omega$). Obviously, the sets $\mathrm{Homeo}(M)$ and
$\mathrm{Homeo}_*(M)$ form groups (the operation is the composition
of the maps). Denote by $\mathrm{Homeo}_*(M,\mu)$ the subgroup of
$\mathrm{Homeo}_*(M)$ whose elements preserve the measure $\mu$.
Denote by $\mathrm{Hameo}(M,\mu)$ the subset of
$\mathrm{Homeo}_*(M,\mu)$ whose elements are $\mu$-Hamiltonian. It
has been proved that $\mathrm{Hameo}(M,\mu)$ forms a group
\cite{F1}.


Denote by $\mathscr{I}\mathscr{S}_*(M,\mu)$ the subgroup of
$\mathscr{I}\mathscr{S}_*(M)$ whose element
$(F_t)_{t\in[0,1]}\in\mathscr{I}\mathscr{S}_*(M)$ satisfies
$(F_1)_*\mu=\mu$. The homotopic relation is an equivalence relation
on $\mathscr{I}\mathscr{S}_*(M)$ (resp.
$\mathscr{I}\mathscr{S}_*(M,\mu)$). Denote the set of equivalence
classes by $\mathscr{H}_*(M)$ (resp. $\mathscr{H}_*(M,\mu)$). It
turns out that $\mathscr{H}_*(M)$ and $\mathscr{H}_*(M,\mu)$ are
groups. Indeed, $\mathscr{H}_*(M)$ (resp. $\mathscr{H}_*(M,\mu)$) is
the universal covering space of $\mathrm{Homeo}_*(M)$ (resp.
$\mathrm{Homeo}_*(M,\mu)$) \cite[Section 5]{Fathi}.

Given $I\in\mathscr{H}_*(M,\mu)$, if the isotopy $I$ does not
satisfy the WB-property (Definition \ref{def:wb property and b
property}), then there must exist three fixed points
$\widetilde{a}$,
 $\widetilde{b}$ and $\widetilde{c}$ of $\widetilde{F}$ such that
$i(\widetilde{F};\widetilde{a},\widetilde{c})-i(\widetilde{F};\widetilde{b},\widetilde{c})\neq
0$ which is independent of the choice of the isotopy $I$ from
$\mathrm{Id}_M$ to $F$ (see \cite[Section 4.2.4.2]{W2}). By Equation
\ref{eq:linking number for two fixed points}, we have
$i(\widetilde{F}^n;\widetilde{a},\widetilde{b})=n\cdot
i(\widetilde{F};\widetilde{a},\widetilde{b})$ for every
$n\in\mathbb{N}$, where
$i(\widetilde{F}^n;\widetilde{z},\widetilde{z}\,')=i_{\widetilde{I}^n}
(\widetilde{z},\widetilde{z}\,')$ and $\widetilde{I}^n$ is the
lifted identity isotopy of $I^n$ to $\widetilde{M}$. Hence we obtain
that
$|i(\widetilde{F}^n;\widetilde{a},\widetilde{c})-i(\widetilde{F}^n;\widetilde{b},\widetilde{c})|\succeq
n$. Note that the value of $\rho_{M,I}(\mu)$ only depend on the
homotopic class with fixed endpoints of $I$. Therefore, if
$\rho_{M,I}(\mu)\neq0$, by the morphism property of
$\rho_{M,\cdot}(\mu): \mathscr{H}_*(M,\mu)\rightarrow
H_1(M,\mathbb{R})$:
$\rho_{M,II'}(\mu)=\rho_{M,I}(\mu)+\rho_{M,I'}(\mu)$ (see \cite{F1}
for details), we have
$\|\rho_{M,I^n}(\mu)\|_{H_1(M,\mathbb{R})}\succeq n$, where
$\|\cdot\|_{H_1(M,\mathbb{R})}$ is a norm on the space
$H_1(M,\mathbb{R})$.

Applying Proposition \ref{prop:ggeq1 no torsion}, Proposition
\ref{prop:g>1 no torsion}, and the arguments above, we immediately
obtain the following result:
\begin{cor}\label{cor:no torsion}
The groups $\mathrm{Hameo}(\mathbb{T}^2,\mu)$ and
$\mathrm{Homeo}_*(\Sigma_g,\mu)$ ($g>1$) are torsion free.
\end{cor}

Note that it is easy to find a homeomorphism $F\in
\mathrm{Homeo}_*(\mathbb{T}^2,\mu)\setminus
\mathrm{Hameo}(\mathbb{T}^2,\mu)$ such that
$F^n=\mathrm{Id}_{\mathbb{T}^2}$ for $n>1$, such as any rigidity
rotation on $\mathbb{T}^2$ with rotation $\alpha\in
\mathbb{Q}^2\setminus\mathbb{Z}^2$. 
\bigskip

Let us finish this section by introduce the Zimmer's conjecture.
Denote by $\mathrm{Ham}^1(M,\mu)$ the group
$\mathrm{Hameo}(M,\mu)\cap \mathrm{Diff}^1(M)$ and by
$\mathrm{Diff}^1_*(M,\mu)$ the group $\mathrm{Homeo}_*(M,\mu)\cap
\mathrm{Diff}^1(M)$. Let us now recall the definition of distortion
(see \cite{P}). If $\mathscr{G}$ is a finitely generated group with
generators $\{g_1,\ldots,g_s\}$,
$f\in\mathscr{G}$ is a \emph{distortion element} 
of $\mathscr{G}$ provided that $f$ has infinite order and
$$\liminf_{n\rightarrow+\infty}\frac{\|f^n\|_{\mathscr{G}}}{n}=0,$$
where $\|f^n\|_{\mathscr{G}}$ is the word length of $f^n$ in the
generators $\{g_1,\ldots,g_s\}$. If $\mathscr{G}$ is not finitely
generated, then we say that $f\in\mathscr{G}$ is distorted in
$\mathscr{G}$ if it is distorted in some finitely generated subgroup
of $\mathscr{G}$.
\begin{thm}\label{thm:Fn thicksim n}
Assume that
$F\in\mathrm{Diff}^1_*(\Sigma_g,\mu)\setminus\{\mathrm{Id}_{\Sigma_g}\}$
($g>1$) (resp.
$F\in\mathrm{Ham}^1(\mathbb{T}^2,\mu)\setminus\{\mathrm{Id}_{\mathbb{T}^2}\}$),
and $\mathscr{G}\subset\mathrm{Diff}^1_*(\Sigma_g,\mu)$ ($g>1$)
(resp. $\mathscr{G}\subset\mathrm{Ham}^1(\mathbb{T}^2,\mu)$) is a
finitely generated subgroup containing $F$, then
$$\|F^n\|_{\mathscr{G}}\thicksim n.$$
As a consequence, the groups $\mathrm{Diff}^1_*(\Sigma_g,\mu)$
($g>1$) and $\mathrm{Ham}^1(\mathbb{T}^2,\mu)$ have no
distortion.\end{thm}

\begin{thm}\label{thm:zimmer2}Every homomorphism from
$\mathrm{SL}(n,\mathbb{Z})$ ($n\geq3$) to
$\mathrm{Ham}^1(\mathbb{T}^2,\mu)$ or
$\mathrm{Diff}^1_*(\Sigma_g,\mu)$ ($g>1$) is trivial. As a
consequence, every  homomorphism from $\mathrm{SL}(n,\mathbb{Z})$
($n\geq3$) to $\mathrm{Diff}^1(\Sigma_g,\mu)$ ($g>1$) has only
finite images.
\end{thm}

Theorem \ref{thm:zimmer2} is a more general conjecture of Zimmer
\cite{Zim} in the special surfaces case. Remark that Polterovich
\cite{P} showed us a Hamiltonian version of this theorem by using
symplectic filling function, and that Franks and Handel \cite{F4}
obtained this theorem  by the Thurston theory of normal forms for
surface homeomorphisms. Our strategy of proof is similar to the
proof of Polterovich \cite{P}. Hence our proof is totally different
from Franks and Handel's. However, the technology of our proof is
different from Polterovich's so that we can generalize his results
to the group $\mathrm{Ham}^1(\mathbb{T}^2,\mu)$ and the group
$\mathrm{Diff}^1(\Sigma_g,\mu)$ ($g>1$), where $\mu$ is a usual
Borel finite measure with full support. We note that the group
$\mathrm{Ham}^1(\mathbb{T}^2,\mu)$ is defined on the homology level
(comparing to the definition of the classical Hamiltonian
diffeomorphism which is defined on the co-homology level). The
reader can find more information about Zimmer's conjecture in
Section \ref{sec:distortion of group}.

\smallskip

The article is organized as follows. In Section 2, we first
introduce some notations, recall some results about identity
isotopies, and study the WB-property on a connected subset of
$\mathrm{Fix}(\widetilde{F})$. In Section 3, we explain the approach
to defining the generalized action function and study the continuity
of this action function. Our main results Proposition
\ref{prop:action is constant on the path connected components},
Theorem \ref{prop:F is not constant if the contractible fixed points
is finite} and Theorem \ref{cor:the symplectic action when M with
genus bigger 1} will be proved in Section 4, Section 5 and Section
6, respectively. In Section 7, we will provide an alternative proof
of the $C^1$-version of the Zimmer's conjecture on surfaces when the
measure is a Borel finite measure with full support from our method.
In Appendix, we provide the proofs of the lemmas which are not given
in the main sections and also construct Example \ref{exem:T2 and
supp(u)notM} and Example \ref{exem:M and supp(u)notM} to complete
Theorem \ref{prop:F is not constant if the contractible fixed points
is finite} and Theorem \ref{thm:connected and Id}.\bigskip

\noindent\textbf{Acknowledgements.} I would like to thank Patrice Le
Calvez for many helpful discussions and suggestions. I also thank
Fr\'{e}d\'{e}ric Le Roux and Olivier Jaulent for explaining their
results to me. I am grateful to Yiming Long, Matthias Schwarz,
Lucien Guillou and Yinshan Chang for reading the manuscript and many
useful remarks. I would like to express my deep sorrow over the
passing away of Lucien Guillou in 2015.

\section{Notations}

We denote by $|\cdot|$ the usual Euclidean metric on $\mathbb{R}^k$
or $\mathbb{C}^k$ and by $\mathbf{S}^{k-1}=\{x\in\mathbb{R}^k\mid
|x|=1\}$ the unit sphere.

If $A$ is a set, we write $\sharp A$ for the cardinality of $A$. 
If $(S, \sigma,\mu)$ is a measure space and $V$ is any finite
dimensional linear space, we denote by $L^1(S,V,\mu)$ the set of
$\mu$-integrable functions from $S$ to $V$. If $X$ is a topological
space and $A$ is a subset of $X$, we denote by $\mathrm{Int}_X(A)$
and $\mathrm{Cl}_X(A)$ respectively the interior and the closure of
$A$. We will omit the subscript $X$ if no confusion arises. If $M$
is a manifold and $N$ is a submanifold of $M$, we denote by
$\partial N$ the boundary of $N$ on $M$.

\subsection{Identity isotopies}\label{sec:Identity isotopies}An
\emph{identity isotopy} $I=(F_t)_{t\in[0,1]}$ on $M$ is a continuous
path
\begin{eqnarray*}
  [0,1] &\rightarrow& \mathrm{Homeo}(M) \\
  t&\mapsto& F_t
\end{eqnarray*}
such that $F_0=\mathrm{Id}_{M}$, where the last set is endowed with
the compact-open topology. We naturally extend this map to
$\mathbb{R}$ by writing $F_{t+1}=F_t\circ F_1$. We can also define
the inverse isotopy of $I$ as
$I^{-1}=(F_{-t})_{t\in[0,1]}=(F_{1-t}\circ F_1^{-1})_{t\in[0,1]}$.
We denote by $\mathrm{Homeo}_*(M)$ the set of all homeomorphisms of
$M$ that are isotopic to the identity.\smallskip

A \emph{path} on a manifold $M$ is a continuous map $\gamma:
J\rightarrow M$ defined on a nontrivial interval $J$ (up to an
increasing reparametrization). We can talk of a proper path (i.e.
$\gamma^{-1}(K)$ is compact for any compact set $K$) or a compact
path (i.e. $J$ is compact). When $\gamma$ is a compact path,
$\gamma(\inf J)$ and $\gamma(\sup J)$ are the \emph{ends} of
$\gamma$. We say that a compact path $\gamma$ is a \emph{loop} if
the two ends of $\gamma$ coincide. The inverse of the path $\gamma$
is defined by $\gamma^{-1}:t\mapsto\gamma(-t),\,t\in -J$. If
$\gamma_1: J_1\rightarrow M$ and $\gamma_2: J_2\rightarrow M$ are
two paths such that $$b_1=\sup J_1\in J_1,\quad a_2=\inf J_2\in
J_2,\quad\mathrm{and}\quad\gamma_1(b_1)=\gamma_2(a_2),$$ then the
\emph{concatenation of $\gamma_1$ and $\gamma_2$} is defined on
$J=J_1\cup(J_2+(b_1-a_2))$  in the classical way, where
$(J_2+(b_1-a_2))$ represents the translation of
 $J_2$ by $(b_1-a_2)$:
\begin{equation*}\gamma_1\gamma_2(t)=
\begin{cases}\gamma_1(t)& \textrm{if} \quad t\in J_1;
\\\gamma_2(t+a_2-b_1)& \textrm{if} \quad
t\in J_2+(b_1-a_2).\end{cases}
\end{equation*}

Let $\mathcal {I}$ be an interval (maybe infinite) of $\mathbb{Z}$.
If $\{\gamma_i:J_i\rightarrow M\}_{i\in\mathcal {I}}$ is a family of
compact paths satisfying that
$\gamma_i(\sup(J_i))=\gamma_{i+1}(\inf(J_{i+1}))$ for every
$i\in\mathcal {I}$, then we can define their concatenation as
$\prod_{i\in\mathcal {I}}\gamma_i$.

If $\{\gamma_i\}_{i\in \mathcal {I}}$ is a family of compact paths
where $\mathcal {I}=\bigsqcup_{j\in\mathcal {J}}\mathcal {I}_j$ and
$\mathcal {I}_j$ is an interval of $\mathbb{Z}$ such that
$\prod_{i\in\mathcal {I}_j}\gamma_i$ is well defined (in the
concatenation sense) for each $j\in\mathcal {J}$, we define their
\emph{product} by abusing notations:
$$\prod\limits_{i\in\mathcal
{I}}\gamma_i=\prod_{j\in\mathcal {J}}\prod\limits_{i\in\mathcal
{I}_j}\gamma_i.$$

The \emph{trajectory} of a point $z$ for the isotopy
$I=(F_t)_{t\in[0,1]}$ is the oriented path $I(z): t\mapsto F_t(z)$
defined on $[0,1]$. Suppose that $\{I_k\}_{1\leq k\leq k_0}$ is a
family of identity isotopies on $M$. Write
$I_k=(F_{k,t})_{t\in[0,1]}$. We can define a new identity isotopy
$I_{k_0}\cdots I_{2}I_1=(F_t)_{t\in[0,1]}$ by concatenation as
follows
\begin{equation}\label{the product operate of the symplectic group of isotopy}
    F_t(z)=F_{k,\,k_0t-(k-1)}(F_{k-1,1}\circ F_{k-2,1}\circ\cdots\circ F_{1,1}(z))\quad\mathrm{if}
\quad \frac{k-1}{k_0}\leq t\leq\frac{k}{k_0}.
\end{equation}
In particular, $I^{k_0}(z)=\prod_{k=0}^{k_0-1}I({F^k(z)})$ when
$I_k=I$ for all $1\leq k\leq k_0$.\smallskip

We write $\mathrm{Fix}(F)$ for the set of fixed points of $F$. A
fixed point $z$ of $F=F_1$ is \emph{contractible} if $I(z)$ is
homotopic to zero. We write $\mathrm{Fix}_{\mathrm{Cont},I}(F)$ for
the set of contractible fixed points of $F$, which obviously depends
on $I$.

\subsection{The algebraic intersection number}\label{sec:the
algebraic intersection number} Choosing an orientation on $M$
permits us to define the algebraic intersection number
$\Gamma\wedge\Gamma'$ between two loops. We keep the same notation
$\Gamma\wedge\gamma$ for the algebraic intersection number between a
loop $\Gamma$ and a path $\gamma$ when it is defined, e.g., when
$\gamma$ is proper or when $\gamma$ is a compact path whose
extremities are not in $\Gamma$. Similarly, we write
$\gamma\wedge\gamma'$ for the algebraic intersection number of two
paths $\gamma$ and $\gamma'$ when it is defined, e.g., when $\gamma$
and $\gamma'$ are compact paths and the ends of $\gamma$ (resp.
$\gamma'$) are not on $\gamma'$ (resp. $\gamma$). If $\Gamma$ is a
loop on a smooth manifold $M$, we write $[\Gamma]\in
H_1(M,\mathbb{Z})$ for the homology class of $\Gamma$. It is clear
that the value $\Gamma\wedge\gamma$ does not depend on the choice of
the path $\gamma$ whose endpoints are fixed when $[\Gamma]=0$.
\subsection{Rotation vector}\label{subsec:rotation vector}
Let us introduce the classical notion of rotation vector which was
originally defined in \cite{S}. Suppose that $F$ is the time-one map
of an identity isotopy $I=(F_t)_{t\in[0,1]}$. Let
$\mathrm{Rec}^+(F)$ be the set of positively recurrent points of
$F$. If $z\in \mathrm{Rec}^+(F)$, we fix an open disk $U\subset M$
containing $z$, and write $\{F^{n_k}(z)\}_{k\geq 1}$ for the
subsequence of the positive orbit of $z$ obtained by keeping the
points in $U$. For any $k\geq 0$, choose a simple path
$\gamma_{F^{n_k}(z),z}$ in $U$ joining $F^{n_k}(z)$ to $z$. The
homology class $[\Gamma_k]\in H_1(M,\mathbb{Z})$ of the loop
$\Gamma_k= I^{n_k}(z)\gamma_{F^{n_k}(z),z}$ is independent of the
choice of $\gamma_{F^{n_k}(z),z}$. We say that $z$ has a
\emph{rotation vector} $\rho_{M,I}(z)\in H_1(M,\mathbb{R})$ if
\[\lim_{l\rightarrow
+\infty}\frac{1}{n_{k_l}}[\Gamma_{k_l}]=\rho_{M,I}(z)\] for any
subsequence $\{F^{n_{k_l}}(z)\}_{l\geq 1}$ which converges to $z$.
Neither the existence nor the value of the rotation vector depends
on the choice of $U$.

Suppose that $M$ is compact and that $F$ is the time-one map of an
identity isotopy $I=(F_t)_{t\in[0,1]}$ on $M$. Recall that $\mathcal
{M}(F)$ is the set of Borel finite measures on $M$ whose elements
are invariant by $F$. If $\mu\in\mathcal {M}(F)$, we can define the
rotation vector $\rho_{M,I}(z)$ for $\mu$-almost every positively
recurrent points. Moreover, we can prove that the rotation vector is
uniformly bounded if it exists (see \cite[page 52]{W2}). Therefore,
we define the \emph{rotation vector of the measure} to be
$$\rho_{M,I}(\mu)=\int_M\rho_{M,I}\, \mathrm{d}\mu\in H_1(M,\mathbb{R}).$$

\subsection{Some results about identity isotopies}

\begin{rem}\label{rem: contractible fixed point and isotopy}
Suppose that $M$ is an oriented compact surface and that $F$ is the
time-one map of an identity isotopy $I=(F_t)_{t\in[0,1]}$ on $M$.
When $z\in \mathrm{Fix}_{\mathrm{Cont},I}(F)$, there is another
identity isotopy $I'=(F'_t)_{t\in[0,1]}$ homotopic to $I$ with fixed
endpoints such that $I'$ fixes $z$ (see, e.g., \cite[Proposition
2.15]{J}). That is, there is a continuous map $H:
[0,1]\times[0,1]\times M\rightarrow M$ such that
\begin{itemize}
\item $H(0,t,z)=F_t(z)$ and $H(1,t,z)=F'_t(z)$ for all $t\in[0,1]$;
\item $H(s,0,z)=\mathrm{Id}_{M}(z)$ and $H(s,1,z)=F(z)$ for all $s\in[0,1]$;
\item $F'_t(z)=z$ for all $t\in[0,1]$.
\end{itemize}

\end{rem}

\begin{lem}[\cite{W2}, page 54]\label{rem:identity isotopies fix three points on sphere}
Let $\mathbf{S}^2$ be the 2-sphere and $I=(F_t)_{t\in[0,1]}$ be an
identity isotopy on $\mathbf{S}^2$. For every three different fixed
points $z_i$ ($i=1,2,3$) of $F_1$, there exists another identity
isotopy $I'=(F'_t)_{t\in[0,1]}$ from $\mathrm{Id}_{\mathbf{S}^2}$ to
$F_1$ such that $I'$ fixes $z_i$ $(i=1,2,3)$.
\end{lem}

As a consequence, we have the following corollary.

\begin{cor}\label{cor:identity isotopy fixes two points}
Let $I=(F_t)_{t\in[0,1]}$ be an identity isotopy on $\mathbb{C}$.
For any two different fixed points $z_1$ and $z_2$ of $F_1$, there
exists another identity isotopy $I'$ from $\mathrm{Id}_{\mathbb{C}}$
to $F_1$ such that $I'$ fixes $z_1$ and $z_2$.
\end{cor}

\begin{rem}\label{rem:some result of of sphere delete three points}
Let $z_i\in \mathbf{S}^2$ ($i=1,2,3$) and
$\mathrm{Homeo}_*(\mathbf{S}^2,{z_1,z_2,z_3})$ be the identity
component of the space of all homeomorphisms of $\mathbf{S}^2$
leaving $z_i$ ($i=1,2,3$) pointwise fixed (for the compact-open
topology). It is well known that
$\pi_1(\mathrm{Homeo}_*(\mathbf{S}^2,{z_1,z_2,z_3}))=0$
\cite{H2,Han}. It implies that any two identity isotopies $I ,I'
\subset \mathrm{Homeo}_*(\mathbf{S}^2,{z_1,z_2,z_3})$ with fixed
endpoints are homotopic. As a consequence, let
$\mathrm{Homeo}_*(\mathbb{C},{z_1,z_2})$ be the identity component
of the space of all homeomorphisms of $\mathbb{C}$ leaving two
different points $z_1$ and $z_2$ of $\mathbb{C}$ pointwise fixed, we
have $\pi_1(\mathrm{Homeo}_*(\mathbb{C},{z_1,z_2}))=0$.
\end{rem}

\subsection{The linking number on a connected subset of
the fixed points set}\label{subsec:boundedness}

Let $X$ be a connected component of
$\mathrm{Fix}_{\mathrm{Cont},I}(F)$. Either $X$ is
\emph{contractible}, which means it is included in an open disk. In
this case, the preimage of $X$ in the universal covering space is a
disjoint union of sets $\widetilde{X}$ such that the projection
induces a homeomorphim from $\widetilde{X}$ to $X$. Or $X$ is not
contractible, and in this case every connected component of the
preimage of $X$ is unbounded.\smallskip

Recall the linking number
$i(\widetilde{F};\widetilde{z},\widetilde{z}\,')$ for
$(\widetilde{z},\widetilde{z}\,')\in(\mathrm{Fix}(\widetilde{F})\times
\mathrm{Fix}(\widetilde{F}))\setminus\widetilde{\Delta}$ defined in
Formula \ref{eq:linking number for two fixed points}. To prove our
main results, we need the following three lemmas whose proofs are
provided in Appendix.

\begin{lem}\label{rem:linking number on connected set}
If $\widetilde{X}$ is a connected subset of
$\mathrm{Fix}(\widetilde{F})$ and
$\widetilde{z}\in\mathrm{Fix}(\widetilde{F})$, then
$i(\widetilde{F};\widetilde{z},\widetilde{z}\,')$
($\widetilde{z}\,'$ as variable, $\widetilde{z}\,'\neq
\widetilde{z}$) is a constant on $\widetilde{X}$. Furthermore, if
$\widetilde{X}$ is not reduced to a singleton,
$i(\widetilde{F};\cdot,\cdot)$ is a constant on
$(\widetilde{X}\times\widetilde{X})\setminus\widetilde{\Delta}$.\end{lem}

\begin{lem}\label{lem:the linking number on a unbounded set}
If $\widetilde{X}$ is a connected unbounded subset of
$\mathrm{Fix}(\widetilde{F})$, then
$i(\widetilde{F};\widetilde{z},\widetilde{z}\,')=0$ for all
$(\widetilde{z},\widetilde{z}')\in\mathrm{Fix}(\widetilde{F})\times\widetilde{X}$
with $\widetilde{z}\neq \widetilde{z}\,'$. Consequently, if $X$ is a
connected component of $\mathrm{Fix}_{\mathrm{Cont},I}(F)$ and $X$
is not contractible,
$i(\widetilde{F};\widetilde{z},\widetilde{z}\,')=0$ for all
$(\widetilde{z},\widetilde{z}\,')\in\mathrm{Fix}(\widetilde{F})\times\pi^{-1}(X)$
with $\widetilde{z}\neq \widetilde{z}\,'$.\end{lem}

\begin{lem}\label{lem:fps is connected and B property}
We have the following properties:
\begin{enumerate}
  \item If $X$ is a connected subset of\,
$\mathrm{Fix}_{\mathrm{Cont},I}(F)$ and $X$ is not reduced to a
singleton, $I$ satisfies the B-property on $\pi^{-1}(X)$.
  \item $I$ satisfies the B-property if the number of the
connected components of $\mathrm{Fix}_{\mathrm{Cont},I}(F)$ is
finite. In particular, $I$ satisfies the B-property if the set
$\mathrm{Fix}_{\mathrm{Cont},I}(F)$ is connected.
\end{enumerate}

\end{lem}

\section{The generalized action function revisited}\label{sec:definition of
action function} In this section, we recall the approach to defining
the generalized action function and show the continuity of this
function.\vspace{-2mm}
\subsection{The linking number of positively
recurrent points} \label{subsec:the definition of a new linking
number}\quad

Recall that $F$ is the time-one map of an identity isotopy
$I=(F_t)_{t\in[0,1]}$ on a closed oriented surface $M$ of genus
$g\geq 1$ and that $\widetilde{F}$ is the time-one map of the lifted
identity isotopy $\widetilde{I}=(\widetilde{F}_t)_{t\in[0,1]}$ on
the universal cover $\widetilde{M}$ of $M$. We can compactify
$\widetilde{M}$ into a sphere by adding a point $\infty$ at infinity
and then the lift $\widetilde{F}$ can be extended by fixing this
point.

For every distinct fixed points $\widetilde{a}$ and $\widetilde{b}$
of $\widetilde{F}$, by Lemma \ref{rem:identity isotopies fix three
points on sphere}, we can choose an isotopy $\widetilde{I}_1$ from
$\mathrm{Id}_{\widetilde{M}}$ to $\widetilde{F}$ that fixes
$\widetilde{a}$ and $\widetilde{b}$. Recall that $\pi:
\widetilde{M}\rightarrow M$ is the covering map.

Fix $z\in
\mathrm{Rec}^+(F)\setminus\pi(\{\widetilde{a},\widetilde{b}\})$ and
consider an open disk $U\subset
M\setminus\pi(\{\widetilde{a},\widetilde{b}\})$ containing $z$. We
define the first return map $\Phi: \mathrm{Rec}^+(F)\cap
U\rightarrow \mathrm{Rec}^+(F)\cap U$ and write
$\Phi(z)=F^{\tau(z)}(z)$, where $\tau(z)$ is the first return time,
that is, the least number $n\geq1$ such that $F^n(z)\in U$. 
For every pair $(z',z'')\in U^2$, we choose an oriented simple path
$\gamma_{z',z''}$ in $U$ from $z'$ to $z''$. Denote by
$\widetilde{\Phi}$ the lift
 of the first return map $\Phi$:
\begin{eqnarray*}
  \widetilde{\Phi}: \pi^{-1}(\mathrm{Rec}^+(F))\cap\pi^{-1}(U)&\rightarrow&
 \pi^{-1}(\mathrm{Rec}^+(F))\cap\pi^{-1}(U) \\
\widetilde{z} &\mapsto&
  \widetilde{F}^{\tau(\pi(\widetilde{z}))}(\widetilde{z}).
\end{eqnarray*}

For any $\widetilde{z}\in\pi^{-1}(U)$, we write $U_{\widetilde{z}}$
the connected component of $\pi^{-1}(U)$ that contains
$\widetilde{z}$. For each $j\geq1$, let
$\tau_j(z)=\sum\limits_{i=0}^{j-1}\tau(\Phi^i(z))$. For every
$n\geq1$, consider the following curves in $\widetilde{M}$:
$$
\widetilde{\Gamma}_{\widetilde{I}_1,\widetilde{z}}^n=\widetilde{I}_1^{\,\tau_n(z)}(\widetilde{z})
\widetilde{\gamma}_{\widetilde{\Phi}^n(\widetilde{z}),\widetilde{z}_n}\,,$$
where $\widetilde{z}_n\in \pi^{-1}(\{z\})\cap
\widetilde{U}_{\widetilde{\Phi}^n(\widetilde{z})}$ and
$\widetilde{\gamma}_{\widetilde{\Phi}^n(\widetilde{z}),\widetilde{z}_n}$
is the lift of $\gamma_{\Phi^n(z),z}$ that is contained in
$\widetilde{U}_{\widetilde{\Phi}^n(\widetilde{z})}$. We define the
following infinite product (see Section \ref{sec:Identity
isotopies}):
$$\widetilde{\Gamma}_{\widetilde{I}_1,z}^n=\prod_{\pi(\widetilde{z})=z}\widetilde
{\Gamma}_{\widetilde{I}_1,\widetilde{z}}^n\,.$$In particular, when
$z\in\mathrm{Fix}(F)$,
$\widetilde{\Gamma}_{\widetilde{I}_1,z}^1=\prod\limits_{\pi(\widetilde{z})=z}\widetilde{I}_1(\widetilde
z)$.

When
$\widetilde{U}_{\widetilde{\Phi}^n(\widetilde{z})}=\widetilde{U}_{\widetilde{z}}$,
the curve $\widetilde{\Gamma}_{\widetilde{I}_1,\widetilde{z}}^n$ is
a loop and hence $\widetilde{\Gamma}_{\widetilde{I}_1,z}^n$ is an
infinite family of loops, which will be called a \emph{multi-loop}.
When
$\widetilde{U}_{\widetilde{\Phi}^n(\widetilde{z})}\neq\widetilde{U}_{\widetilde{z}}$,
the curve $\widetilde{\Gamma}_{\widetilde{I}_1,\widetilde{z}}^n$ is
a compact path and $\widetilde{\Gamma}_{\widetilde{I}_1,z}^n$
 is therefore an infinite family of
 paths (it can be seen as a family of proper paths, which both ends of these paths go to
$\infty$), which will be called a \emph{multi-path}.

In the both cases, for every neighborhood $\widetilde{V}$ of
$\infty$, there must be finitely many loops or paths
$\widetilde{\Gamma}_{\widetilde{I}_1,\widetilde{z}}^n$ that are not
included in $\widetilde{V}$. By adding the point $\infty$ at
infinity, we get a multi-loop on the sphere
$\mathbf{S}=\widetilde{M}\sqcup\{\infty\}$.

Hence $\widetilde{\Gamma}_{\widetilde{I}_1,z}^n$ can be seen as a
multi-loop in the annulus
$A_{\widetilde{a},\widetilde{b}}=\mathbf{S}\setminus\{\widetilde{a},\widetilde{b}\}$
with a finite homology. As a consequence, if $\widetilde{\gamma}$ is
a path from $\widetilde{a}$ to $\widetilde{b}$, then the
intersection number
$\widetilde{\gamma}\wedge\widetilde{\Gamma}_{\widetilde{I}_1,z}^{n}$
is well defined and does not depend on $\widetilde{\gamma}$. By
Remark \ref{rem:some result of of sphere delete three points} and
the properties of intersection number, the intersection number
depends on $U$ but not on the choice of the identity isotopy
$\widetilde{I}_1$. Moreover, by observing that the path
$(\prod_{i=0}^{n-1}\gamma_{\Phi^{n-i}(z)
 ,\Phi^{n-i-1}(z)})(\gamma_{\Phi^n(z),z})^{-1}$ is a loop in $U$, we
 have
\begin{equation}\label{eq:Birkhoff sum}
\widetilde{\gamma}\wedge\widetilde{\Gamma}_{\widetilde{I}_1,z}^{n}=
\widetilde{\gamma}\wedge\prod_{j=0}^{n-1}\widetilde{\Gamma}_{\widetilde{I}_1,\Phi^j(z)}^{1}
=\sum_{j=0}^{n-1}\widetilde{\gamma}\wedge\widetilde{\Gamma}_{\widetilde{I}_1,\Phi^j(z)}^{1}.
\end{equation}
\smallskip

For $n\geq1$, we can define the function $$L_{n}:
((\mathrm{Fix}(\widetilde{F})\times
\mathrm{Fix}(\widetilde{F}))\setminus\widetilde{\Delta})\times
(\mathrm{Rec}^+(F)\cap U)\rightarrow \mathbb{Z},$$
\begin{equation}\label{eq: Ln Birkhoff sum}
    L_{n}(\widetilde{F};\widetilde{a},\widetilde{b},z)=\widetilde{\gamma}\wedge
\widetilde{\Gamma}^n_{\widetilde{I}_1,z}=\sum\limits_{j=0}^{n-1}
L_1(\widetilde{F};\widetilde{a},\widetilde{b},\Phi^j(z)),
\end{equation}
where $U\subset M\setminus\pi(\{\widetilde{a},\widetilde{b}\})$. The
last equality of Equation \ref{eq: Ln Birkhoff sum} follows from
Equation \ref{eq:Birkhoff sum}. And the function $L_n$ depends on
$U$ but not on the choice of $\gamma_{\Phi^n(z),z}$.

\begin{defn}\label{def:Intersection number density} Fix $z
\in \mathrm{Rec}^+(F)\setminus\pi(\{\widetilde{a},\widetilde{b}\})$.
We say that the linking number
$i(\widetilde{F};\widetilde{a},\widetilde{b},z)\in \mathbb{R}$ is
defined, if
$$\lim_{k\rightarrow +\infty} \frac{L_{n_k}(\widetilde{F};\widetilde{a},\widetilde{b},z)}{\tau_{n_{k}}(z)}
=i(\widetilde{F};\widetilde{a},\widetilde{b},z)$$ for any
subsequence $\{\Phi^{n_{k}}(z)\}_{k\geq 1}$ of
$\{\Phi^{n}(z)\}_{n\geq 1}$ which converges to $z$.
\end{defn}
Note here that the linking number
$i(\widetilde{F};\widetilde{a},\widetilde{b},z)$ does not depend on
$U$ since if $U$ and $U'$ are open disks containing $z$, there
exists a disk containing $z$ that is contained in $U\cap U'$. In
particular, when $z\in\mathrm{Fix}(F)\setminus\pi(\{\widetilde{a}
,\widetilde{b}\})$, the linking number
$i(\widetilde{F};\widetilde{a},\widetilde{b},z)$ always exists and
is equal to $L_1(\widetilde{F};\widetilde{a},\widetilde{b},z)$.
Moreover, if $z\in\mathrm{Fix}_{\mathrm{Cont},I}(F)$, we have
\cite[page 57]{W2}
\begin{equation}\label{linking number of fixed points}
    i(\widetilde{F};\widetilde{a},\widetilde{b},z)=\sum_{\pi(\widetilde{z})=z}
\left(i(\widetilde{F};\widetilde{a},\widetilde{z})-i(\widetilde{F};\widetilde{b},\widetilde{z})\right).
\end{equation}

\subsection{Some elementary properties of the linking number \cite[Section 4.5.2]{W2}}\quad

\begin{prop}\label{prop:alpha- one exists then
 other exists}
Let $G$ be the covering transformation group of $\pi:
\widetilde{M}\rightarrow M$. For every $\alpha\in G$, all distinct
fixed points $\widetilde{a}$ and $\widetilde{b}$ of $\widetilde{F}$,
and every $z\in
\mathrm{Rec}^+(F)\setminus\pi(\{\widetilde{a},\widetilde{b}\})$, we
have
$L_n(\widetilde{F};\alpha(\widetilde{a}),\alpha(\widetilde{b}),z)=
L_n(\widetilde{F};\widetilde{a},\widetilde{b},z)$ for each $n$. If
$i(\widetilde{F};\widetilde{a},\widetilde{b},z)$ exists, then
$i(\widetilde{F};\alpha(\widetilde{a}),\alpha(\widetilde{b}),z)$
also exists and
$i(\widetilde{F};\alpha(\widetilde{a}),\alpha(\widetilde{b}),z)=
i(\widetilde{F};\widetilde{a},\widetilde{b},z)$.
\end{prop}

\begin{prop}\label{lem:i is 3coboundary for point}
For all distinct fixed points $\widetilde{a}$, $\widetilde{b}$ and
$\widetilde{c}$ of $\widetilde{F}$, and every $z\in
\mathrm{Rec}^+(F)\setminus\pi(\{\widetilde{a},\widetilde{b},\widetilde{c}\})$,
we have
$L_n(\widetilde{F};\widetilde{a},\widetilde{b},z)+L_n(\widetilde{F};\widetilde{b},
\widetilde{c},z)+L_n(\widetilde{F};\widetilde{c},\widetilde{a},z)=0$
for each $n$. Moreover, if two among the three linking numbers
$i(\widetilde{F};\widetilde{a}, \widetilde{b},z)$,
$i(\widetilde{F};\widetilde{b},\widetilde{c},z)$ and
$i(\widetilde{F};\widetilde{c},\widetilde{a},z)$ exist, then the
third one also exists and we have
$$i(\widetilde{F};\widetilde{a},\widetilde{b},z)+i(\widetilde{F};\widetilde{b},
\widetilde{c},z)+i(\widetilde{F};\widetilde{c},\widetilde{a},z)=0.$$
\end{prop}

The following lemma gives the continuity property of the function
$L_k$.
\begin{lem}\label{lem:L1anaz}Suppose that
$\widetilde{a}\in\mathrm{Fix}(\widetilde{F})$ and
$\{\widetilde{a}_n\}_{n\geq1}\subset
\mathrm{Fix}(\widetilde{F})\setminus\{\widetilde{a}\}$ satisfying
$\widetilde{a}_n\rightarrow\widetilde{a}$ as $n\rightarrow+\infty$.
Then we have
\begin{itemize}
  \item $\lim_{n\rightarrow+\infty}i(\widetilde{F};\widetilde{a}_n,\widetilde{a},z)
=0$ for $z\in \mathrm{Fix}(F)\setminus\{\pi(\widetilde{a})\}$;
  \item $\lim_{n\rightarrow+\infty}L_k(\widetilde{F};\widetilde{a}_n,\widetilde{a},z)=0$
  for every $k\geq1$ and $z\in \mathrm{Rec}^+(F)\cap U$, where $U$ is an open disk of $M\setminus\{\pi(\widetilde{a})\}$.
\end{itemize}

\end{lem}
\begin{proof}Let
$C_z^k=\pi^{-1}(\{z,F(z),\cdots,F^{\tau_k(z)-1}(z)\})$, where
$\tau_k(z)=\sum_{i=0}^{k-1}\tau(\Phi^i(z))$.

For each $n$, let $\widetilde{I}_n$ be the isotopy that fixes
$\widetilde{a}$, $\widetilde{a}_n$ and $\infty$, as constructed in
Lemma \ref{rem:identity isotopies fix three points on sphere}. Up to
conjugacy by a homeomorphism $h:\widetilde{M}\rightarrow
\mathbb{C}$, we can identify $\widetilde{M}$ with the complex plane
$\mathbb{C}$ (refer to Remark \ref{rem:some result of of sphere
delete three points} and Section \ref{subsec:the definition of a new
linking number} for the reasons). Through a simple computation (see
the proof of Lemma \ref{rem:identity isotopies fix three points on
sphere} \cite[page 54]{W2}), we can get the formula of
$\widetilde{I}_n$ as follows
\begin{equation}\label{formula:isotopy}
   \widetilde{I}_n(\widetilde{z})(t)=\frac{\widetilde{a}_n-\widetilde{a}}{\widetilde{F}_t
   (\widetilde{a}_n)-\widetilde{F}_t(\widetilde{a})}\cdot(\widetilde{F}_t(\widetilde{z})-\widetilde{F}_t(\widetilde{a}))
   +\widetilde{a}.
\end{equation}
Let $\widetilde{V}_n$ be a disk whose center is $\widetilde{a}$ and
whose radius is $2|\widetilde{a}_n-\widetilde{a}|$. As the functions
$L_k(\widetilde{F};\widetilde{a}_n,\widetilde{a},z)$ do not depend
on the path from $\widetilde{a}$ to $\widetilde{a}_n$ (see Section
\ref{subsec:the definition of a new linking number}), we can suppose
that the path $\widetilde{\gamma}$ from $\widetilde{a}_n$ to
$\widetilde{a}$ is always in $\widetilde{V}_n$.  Since $z\neq
\pi(\widetilde{a})$, the value
\begin{equation}\label{formula:c}
    c=\liminf_{n\geq1}\min_{t\in[0,1],\widetilde{z}\in
C_z^k}|\widetilde{F}_t(\widetilde{z})-\widetilde{F}_t(\widetilde{a}_n)|
\end{equation}
is positive and merely depends on $z$ and $k$. For the constant $c$,
we can find $N>0$ large enough such that
$\max_{t\in[0,1]}|\widetilde{F}_t(\widetilde{a}_n)-\widetilde{F}_t(\widetilde{a})|<c/3$
when $n\geq N$. This implies that, for every $\widetilde{z}\in
C_z^k$, $t\in[0,1]$, and $n\geq N$,
$$|\widetilde{I}_n(\widetilde{z})(t)-\widetilde{a}|>2|\widetilde{a}_n-\widetilde{a}|.$$
As a consequence, we have
$$\lim_{n\rightarrow+\infty}i(\widetilde{F};\widetilde{a}_n,\widetilde{a},z)
=0, \quad \text{for } z\in
\mathrm{Fix}(F)\setminus\{\pi(\widetilde{a})\}$$ and
$$\lim_{n\rightarrow+\infty}L_k(\widetilde{F};\widetilde{a}_n,\widetilde{a},z)=0, \quad \text{for } z\in \mathrm{Rec}^+(F)\cap U.$$
\end{proof}

\subsection{Definition of the generalized action
function}\quad\smallskip

Recall that $F$ is the time-one map of an identity isotopy
$I=(F_t)_{t\in[0,1]}$ on $M$. In \cite[page 85]{W2}, we have proved
that the function $i(\widetilde{F};\widetilde{a},\widetilde{b},z)$
is $\mu$-integrable in each of the following cases:
\begin{enumerate}
  \item the map $F$ and its inverse $F^{-1}$ are
differentiable at $\pi(\widetilde{a})$ and $\pi(\widetilde{b})$;
  \item the isotopy $I$ satisfies the WB-property at $\widetilde{a}$
and $\widetilde{b}$ and the measure $\mu$ has full support;
  \item the isotopy $I$ satisfies the WB-property at $\widetilde{a}$
and $\widetilde{b}$ and the measure $\mu$ is ergodic.
\end{enumerate}

Suppose now the function
$i(\widetilde{F};\widetilde{a},\widetilde{b},z)$ is
$\mu$-integrable. We define the \emph{action difference of
$\widetilde{a}$ and $\widetilde{b}$} as follows
\begin{equation}\label{eq:imu}
    i_{\mu}(\widetilde{F};\widetilde{a},\widetilde{b})= \int_{M\setminus\pi(\{\widetilde{a},\widetilde{b}\})}
i(\widetilde{F};\widetilde{a},\widetilde{b},z)\,\mathrm{d}\mu.
\end{equation}

As an immediate consequence of Proposition \ref{prop:alpha- one
exists then
 other exists}, we have:

\begin{cor}\label{cor:imu(F,alphaa,alphab)=imu(F,a,b)}
$i_{\mu}(\widetilde{F};\alpha(\widetilde{a}),\alpha(\widetilde{b}))=i_{\mu}(\widetilde{F};\widetilde{a},\widetilde{b})$
for any $\alpha\in G$.
\end{cor}


 We suppose now that the action
difference $i_{\mu}(\widetilde{F};\widetilde{a},\widetilde{b})$ is
well defined for arbitrary two distinct fixed points $\widetilde{a}$
and $\widetilde{b}$ of $\widetilde{F}$. We define the \emph{action
difference} as follows:
\begin{eqnarray*}
  i_{\mu}: (\mathrm{Fix}(\widetilde{F})\times
\mathrm{Fix}(\widetilde{F}))\setminus\widetilde{\Delta}
&\rightarrow&\mathbb{R} \\
  (\widetilde{a},\widetilde{b}) &\mapsto&
  i_{\mu}(\widetilde{F};\widetilde{a},\widetilde{b}).
\end{eqnarray*}

Note that for each of the following cases, the action difference can
be defined  \cite[page 86]{W2} for every pair
$(\widetilde{a},\widetilde{b})\in(\mathrm{Fix}(\widetilde{F})\times
\mathrm{Fix}(\widetilde{F}))\setminus\widetilde{\Delta}$~:
\begin{itemize}
\item $F\in\mathrm{Diff}(M)$;
\item $I$ satisfies the
WB-property and $\mu$ has full support;
\item $I$ satisfies the
WB-property and $\mu$ is ergodic.
\end{itemize}

The following corollary is an immediate conclusion of Proposition
\ref{lem:i is 3coboundary for point}:
\begin{cor}\label{clm:I is coboundary}
For any distinct fixed points $\widetilde{a}$,
 $\widetilde{b}$ and $\widetilde{c}$ of $\widetilde{F}$, we have
$$i_{\mu}(\widetilde{F};\widetilde{a},\widetilde{b})+i_{\mu}(\widetilde{F};\widetilde{b},\widetilde{c})+i_{\mu}(\widetilde{F};\widetilde{c},\widetilde{a})=0.$$
That is, $i_{\mu}$ is a coboundary on $\mathrm{Fix}(\widetilde{F})$.
So there is a function $l_{\mu}:
\mathrm{Fix}(\widetilde{F})\rightarrow\mathbb{R}$, defined up to an
additive constant, such that
\begin{equation*}\label{eq:imu and lmu}
    i_{\mu}(\widetilde{F};\widetilde{a},\widetilde{b})=l_{\mu}(\widetilde{F};\widetilde{b})-l_{\mu}(\widetilde{F};\widetilde{a}).
\end{equation*}
\end{cor}
We call the function $l_\mu$ the \emph{action function} on
$\mathrm{Fix}(\widetilde{F})$ deduced by the measure $\mu$.
\smallskip

From Corollary \ref{cor:imu(F,alphaa,alphab)=imu(F,a,b)} and
Corollary \ref{clm:I is coboundary}, we have the following
proposition:

\begin{prop}[\cite{W2}, page 87]\label{clm:L is well defined}
If $\rho_{M,I}(\mu)=0$, then
$i_{\mu}(\widetilde{F};\widetilde{a},\alpha(\widetilde{a})) = 0$ for
each $\widetilde{a}\in \mathrm{Fix}(\widetilde{F})$ and each
$\alpha\in G\setminus\{e\}$, where $e$ is the unit element of $G$.
As a consequence, there exists a function $L_{\mu}$ defined on
$\mathrm{Fix}_{\mathrm{Cont},I}(F)$ such that for every two distinct
fixed points $\widetilde{a}$ and $\widetilde{b}$ of $\widetilde{F}$,
we have
$$i_{\mu}(\widetilde{F};\widetilde{a},\widetilde{b})=
L_{\mu}(\widetilde{F};\pi(\widetilde{b}))-L_{\mu}(\widetilde{F};\pi(\widetilde{a}))
.$$
\end{prop}

We call the function $L_\mu$ the \emph{action function} on
$\mathrm{Fix}_{\mathrm{Cont},I}(F)$ defined by the measure $\mu$. We
proved that the function $L_\mu$ is a generalization of the
classical case (Theorem \ref{thm:PW}, \cite[Theorem 4.3.2]{W2}). By
the construction of $i_\mu$, the following property holds:
\begin{prop}\label{prop:invariant in homotopic sense} The action difference $i_{\mu}$
 (hence the action  $l_{\mu}$) only depends on the
homotopic class with fixed endpoints of $I$. Moreover, $i_{\mu}$
only depends on the time-one map $F$ when $g>1$ and $i_{\mu}$
depends on the homotopic class of $I$ when $g=1$. The same property
holds for $I_{\mu}$ (hence $L_{\mu}$) which defines in the case
where $\rho_{M,I}(\mu)=0$.
\end{prop}\smallskip


\subsection{The continuity of the generalized action
function}\label{sec:The continuity of the generalized action
function}\quad\smallskip

We have the following continuity property of the generalized action
function whose proof details will be also used in the proof of
Theorem \ref{prop:F is not constant if the contractible fixed points
is finite}.

\begin{prop}\label{prop:the continuity of lmu} Suppose that $F$ is the time-one map of
an identity isotopy $I$ on $M$ and that $\mu\in\mathcal{M}(F)$. Let
$\widetilde{X}\subseteq \mathrm{Fix}(\widetilde{F})$. If $I$
satisfies the B-property on $\widetilde{X}$ and $F$ is
$\mu$-symplectic, then we have
$$\lim_{n\rightarrow +\infty}i_\mu(\widetilde{F};\widetilde{a}_n,\widetilde{a})=0$$
for any $\widetilde{a}\in\widetilde{X}$ and
$\{\widetilde{a}_n\}_{n\geq1}\subset
\widetilde{X}\setminus\{\widetilde{a}\}$ satisfying
$\widetilde{a}_n\rightarrow\widetilde{a}$ as $n\rightarrow+\infty$
\footnote{\,In fact, it is also true for the following cases (refer
the proof of Proposition 6.8 in ArXiv:1106.1104):
\begin{enumerate}
  \item $I$
satisfies the B-property on $\widetilde{X}$ and
$F\in\mathrm{Diff}(M)$; 
  \item $I$
satisfies the B-property on $\widetilde{X}$ and $\mu$ is ergodic.
\end{enumerate}}. As a
conclusion, if $I$ satisfies the B-property on $\widetilde{X}$ and
the WB-property (on $\mathrm{Fix}(\widetilde{F})$), the action
$l_\mu$ is continuous on $\widetilde{X}$. Moreover, if $I$ is
$\mu$-Hamiltonian, the action function $L_\mu$ is continuous on
$\pi(\widetilde{X})$.

\end{prop}
\begin{proof}
There exists a triangulation $\{U_i\}_{i=1}^{+\infty}$ of
$M\setminus\mathrm{Fix}(F)$ such that, for each $i$, the interior of
$U_i$ is an open free disk for $F$ (i.e., $F(\mathrm{Int}(U_i))\cap
\mathrm{Int}(U_i)=\emptyset$) and satisfies $\mu(\partial U_i)=0$.
By a slight abuse of notations we will also write $U_i$ for its
interior.
\smallskip

%
According to Lemma \ref{lem:L1anaz}, we have that
$\lim\limits_{n\rightarrow+\infty}i(\widetilde{F};\widetilde{a}_n,\widetilde{a},z)=0$
for $z\in \mathrm{Fix}(F)\setminus\{\pi(\widetilde{a})\}$, and that
$\lim\limits_{n\rightarrow+\infty}L_1(\widetilde{F};\widetilde{a}_n,\widetilde{a},z)=0$
for $z\in \mathrm{Rec}^+(F)\cap U_i$, for every
$i\in\mathbb{N}$.\smallskip

Choose a compact set $\widetilde{P}\subset\widetilde{M}$ such that
$\widetilde{a}\in\mathrm{Int}(\widetilde{P})$ and
$\{\widetilde{a}_n\}_{n\geq1}\subset\widetilde{P}$. As before, when
$\widetilde{a}\,'$ and $\widetilde{b}\,'$ are two distinct fixed
points of $\widetilde{F}$ in $\widetilde{P}$, we can always suppose
that the path $\widetilde{\gamma}$ that joins $\widetilde{a}\,'$ and
$\widetilde{b}\,'$ is in $\widetilde{P}$. By the definition of
B-property, we may suppose that there exists a number $N>0$ such
that
$$N>\mathrm{ess\,}\sup_{n\geq1}\left\{\,\left|i(\widetilde{F};\widetilde{a}_n,\widetilde{a},z)\right|\,\right\},$$where
``$\mathrm{ess}\,\sup$'' is the essential supremum (see Proposition
4.6.11 in \cite[page 85]{W2}). \smallskip

By Lebesgue's dominating convergence theorem (the dominated function
is $N$), we get
\begin{equation*}\label{eq:the estimation of the fixed point case}
    \lim_{n\rightarrow+\infty}\int_{\mathrm{Fix}(F)}
  \left|(i(\widetilde{F};\widetilde{a}_n,\widetilde{a},z)\right|\,\mathrm{d}\mu=0.
\end{equation*}

It is then sufficient to prove that
\begin{equation*}\label{eq:the estimation of non fixed point case}
    \lim_{n\rightarrow+\infty}\int_{M\setminus\mathrm{Fix}(F)}
  \left|(i(\widetilde{F};\widetilde{a}_n,\widetilde{a},z)\right|\,\mathrm{d}\mu=0.
\end{equation*}

Fix $\epsilon>0$. Since
$\mu(\bigcup_{i=1}^{+\infty}U_i)=\mu(M\setminus\mathrm{Fix}(F))<+\infty$,
there exists a positive integer $N'$ such that
$$\mu(\bigcup_{N'+1}^{+\infty}U_i)<\frac{\epsilon}{2N}.$$

For every pair $(\widetilde{a},\widetilde{b})\in
(\mathrm{Fix}(\widetilde{F})\times\mathrm{Fix}(\widetilde{F}))
\setminus\widetilde{\Delta}$, each $i$, and $\mu$-a.e. $z\in U_i$,
we have the following facts:
\begin{itemize}
  \item $\int_{U_i}\tau\, \mathrm{d}\mu=\mu(\bigcup_{k\geq
0}F^k(U_i))$ (by Kac Lemma \cite{kac}, see \cite[page 55]{W2});
  \item $i(\widetilde{F};\widetilde{a},\widetilde{b},z)$ is the action of
$F$, i.e.,
$i(\widetilde{F};\widetilde{a},\widetilde{b},F(z))=i(\widetilde{F};\widetilde{a},\widetilde{b},z)$.
Indeed, one can consider the point $F(z)\in \mathrm{Rec}^+(F)$ and
the open disk $F(U_i)$. Then this fact follows from Definition
\ref{def:Intersection number density}.
\end{itemize}
Therefore,
\begin{equation}\label{eq: FkUi}
    \int_{\bigcup_{k\geq0}F^k(U_i)}\left|i(\widetilde{F};\widetilde{a},\widetilde{b},z)\right|\,\mathrm{d}\mu \\
 =\int_{U_i}\tau(z)\left|i(\widetilde{F};\widetilde{a},\widetilde{b},z)\right|\,\mathrm{d}\mu.
\end{equation}

As $L_1(\widetilde{F};\widetilde{a},\widetilde{b},z)\in
  L^1(U_i,\mathbb{R},\mu)$ (refer to
Proposition 4.6.10 in \cite[page 84]{W2} for the proof), the
following limit exists
$$L^{*}(\widetilde{F};\widetilde{a},\widetilde{b},z)=\lim_{m\rightarrow+\infty}
\frac{L_m(\widetilde{F};\widetilde{a},\widetilde{b},z)}{m}=\lim_{m\rightarrow+\infty}
\frac{1}{m}\sum_{j=1}^{m-1}L_1(\widetilde{F};\widetilde{a},\widetilde{b},\Phi^j(z)).$$

Moreover, we have the following inequality (modulo subsets of
measure zero of $U_i$)
\begin{eqnarray}\label{ineq:two birkhoff ineqs}
\left|L^{*}(\widetilde{F};\widetilde{a},\widetilde{b},z)\right| &=&
\lim_{m\rightarrow +\infty} \frac{1}{m}\left|\sum_{j=0}^{m-1}
   (L_1(\widetilde{F};\widetilde{a},\widetilde{b},\Phi^j(z))\right|\\
  &\leq& \lim_{m\rightarrow +\infty} \frac{1}{m}\sum_{j=0}^{m-1}\left
  |L_1(\widetilde{F};\widetilde{a},\widetilde{b},\Phi^j(z))\right|\nonumber \\
   &=& \left|L_1(\widetilde{F};\widetilde{a},\widetilde{b},z)\right|^{*},\nonumber
\end{eqnarray}
where the last limit exists due to Birkhoff Ergodic theorem.

Applying Birkhoff Ergodic theorem again, we get
$$\tau^*(\Phi(z))=\tau^*(z)\quad \mathrm{and}\quad
L^*(\widetilde{F};\widetilde{a},\widetilde{b},\Phi(z))=L^*(\widetilde{F};\widetilde{a},\widetilde{b},z),$$
where $\tau^*(z)$ is the limit of the sequence
$\{\tau_n(z)/n\}_{n\geq1}$, and $\Phi$ is the first return map on
$U_i$. For $\mu$-a.e. $z\in U_i$, we have
\begin{equation}\label{eq: i=l divide tau}
   i(\widetilde{F};\widetilde{a},\widetilde{b},z)
=\lim_{m\rightarrow+\infty}\frac{L_m(\widetilde{F};\widetilde{a},\widetilde{b},z)}{\tau_{m}(z)}
=\lim_{m\rightarrow+\infty}\frac{L_m(\widetilde{F};\widetilde{a},\widetilde{b},z)/m}{\tau_{m}(z)/m}
=\frac{L^*(\widetilde{F};\widetilde{a},\widetilde{b},z)}{\tau^*(z)}.
\end{equation}

Therefore,
$i(\widetilde{F};\widetilde{a},\widetilde{b},\Phi(z))=i(\widetilde{F};\widetilde{a},\widetilde{b},z)$.
Moreover, observing that
$\tau(z)|i(\widetilde{F};\widetilde{a},\widetilde{b},z)|$ $\in
L^1(U_i,\mathbb{R},\mu)$, we obtain
\begin{eqnarray*}
&&\lim_{m\rightarrow+\infty}\frac{1}{m}\sum_{j=0}^{m-1}\left(\tau(\Phi^j(z))
\left|i(\widetilde{F};\widetilde{a},\widetilde{b},\Phi^j(z))\right|\right)\\
 &=&\lim_{m\rightarrow+\infty}
\left(\frac{1}{m}\sum_{j=0}^{m-1}\tau(\Phi^j(z))\right)\cdot\left|i(\widetilde{F};\widetilde{a},\widetilde{b},z)\right|\\
 &=&\tau^*(z)\left|i(\widetilde{F};\widetilde{a},\widetilde{b},z)\right|
\end{eqnarray*} for $\mu$-a.e. $z\in U_i$.
This implies that
\begin{equation}\label{eq:an ergodic result of pruduction}
    \int_{U_i}\tau(z)\left|i(\widetilde{F};\widetilde{a},\widetilde{b},z)
\right|\,\mathrm{d}\mu=\int_{U_i}\tau^*(z)\left|i(\widetilde{F};\widetilde{a},\widetilde{b},z)\right|\,\mathrm{d}\mu.
\end{equation}

From the equalities \ref{eq: FkUi}, \ref{eq: i=l divide tau},
\ref{eq:an ergodic result of pruduction} and the inequality
\ref{ineq:two birkhoff ineqs} above, we obtain
\begin{eqnarray*}
\int_{\bigcup\limits_{i=1}^{N'}U_i}
  \left|i(\widetilde{F};\widetilde{a}_n,\widetilde{a},z)\right|\,\mathrm{d}\mu&\leq&
\sum_{i=1}^{N'}\int_{\bigcup_{k\geq0}F^k(U_i)}\left|i(\widetilde{F};\widetilde{a}_n,\widetilde{a},z)\right|\,\mathrm{d}\mu \\
  &=& \sum_{i=1}^{N'}\int_{U_i}\tau(z)\left|i(\widetilde{F};\widetilde{a}_n,\widetilde{a},z)\right|\,\mathrm{d}\mu \\
   &=& \sum_{i=1}^{N'}\int_{U_i}\tau^{*}(z)\left|i(\widetilde{F};\widetilde{a}_n,\widetilde{a},z)\right|\,\mathrm{d}\mu \\
   &=& \sum_{i=1}^{N'}\int_{U_i}\left|L^{*}(\widetilde{F};\widetilde{a}_n,\widetilde{a},z)\right|\,\mathrm{d}\mu  \\
   &\leq& \sum_{i=1}^{N'}\int_{U_i}\left|L_1(\widetilde{F};\widetilde{a}_n,\widetilde{a},z)\right|^{*}\,\mathrm{d}\mu \\
   &=&
   \sum_{i=1}^{N'}\int_{U_i}\left|L_1(\widetilde{F};\widetilde{a}_n,\widetilde{a},z)\right|\,\mathrm{d}\mu.
\end{eqnarray*}

As $N'$ is finite, according to Lebesgue's dominating convergence
theorem (with the dominated function $N\tau(z)$) and Lemma
\ref{lem:L1anaz}, we have
$$\lim_{n\rightarrow+\infty}\sum_{i=1}^{N'}\int_{U_i}
\left|L_1(\widetilde{F};\widetilde{a}_n,\widetilde{a},z)\right|\,\mathrm{d}\mu=0.$$
Therefore, there exists a positive number $N''$ such that when
$n\geq N''$, $$\int_{\bigcup\limits_{i=1}^{\,N'}U_i}
  \left|i(\widetilde{F};\widetilde{a}_n,\widetilde{a},z)\right|\,\mathrm{d}\mu\\
< \frac{\epsilon}{2}.$$ Finally, when $n\geq N''$, we obtain
\begin{eqnarray*}
  &&\int_{M\setminus\mathrm{Fix}(F)}
 \left|(i(\widetilde{F};\widetilde{a}_n,\widetilde{a},z)\right|\,\mathrm{d}\mu\\
   &=& \int_{\bigcup\limits_{i=1}^{N'}U_i}\left|(i(\widetilde{F};\widetilde{a}_n,\widetilde{a},z)\right|\,\mathrm{d}\mu+\int_{\bigcup\limits_{N'+1}^{+\infty}U_i}
\left|(i(\widetilde{F};\widetilde{a}_n,\widetilde{a},z)\right|\,\mathrm{d}\mu\\
   &<& \frac{\epsilon}{2}+\frac{\epsilon}{2N}\cdot
   N\\
   &=&\epsilon.
\end{eqnarray*}
Hence, the first statement holds.\smallskip

If $\rho_{M,I}(\mu)=0$, to see $L_\mu$ is continuous on
$\pi(\widetilde{X})$, let $a\in\pi(\widetilde{X})$ and let
$\{a_n\}_{n\geq1}\subset \pi(\widetilde{X})\setminus\{a\}$ converge
to $a$. By Proposition \ref{clm:L is well defined}, we only need to
consider a lift $\widetilde{a}\in\widetilde{X}$ of $a$ and a lifted
sequence $\{\widetilde{a}_n\}_{n\geq1}\subset \widetilde{X}$ of
$\{a_n\}_{n\geq1}$ that converges to $\widetilde{a}$. Then it
follows from the statement above.
\end{proof}

\section{The proof of Proposition \ref{prop:action is constant
on the path connected components}}\label{sec:action is constant on
the path connected components}

Suppose that $X\subseteq\mathrm{Fix}_{\mathrm{Cont},I}(F)$ is
connected and not reduced to a singleton. By Lemma \ref{lem:fps is
connected and B property}, $I$ satisfies the B-property on
$\pi^{-1}(X)$. If $I$ satisfies the hypotheses of Theorem
\ref{thm:PW}, according to Proposition \ref{prop:the continuity of
lmu}, the action function $L_{\mu}$ is continuous on $X$. In fact,
we have the following stronger result: \bigskip

\newenvironment{prop03}{\noindent\textbf{Proposition \ref{prop:action is constant on the path connected
components}}~\itshape}{\par}

\begin{prop03} Under the hypotheses of Theorem \ref{thm:PW}, for
every two distinct contractible fixed points $a$ and $b$ of $F$
which belong to a same connected component of
$\mathrm{Fix}_{\mathrm{Cont},I}(F)$, we have
$I_{\mu}(\widetilde{F};a,b)=0$. As a conclusion, the action function
$L_{\mu}$ is a constant on each connected component of
$\mathrm{Fix}_{\mathrm{Cont},I}(F)$.
\end{prop03}\smallskip

Given $Y\subset M$ and $\epsilon>0$, let $Y_\epsilon=\{z\in M\mid
d(z,y)<\epsilon,y\in Y\}$ be the $\epsilon$-neighborhood of $Y$. If
$N$ is a submanifold of $M$, the inclusion $i: N\hookrightarrow M$
naturally induces a homomorphism: $i_*: \pi_1(N,p)\rightarrow
\pi_1(M,p)$, where $p\in N$. To prove Proposition \ref{prop:action
is constant on the path connected components}, we need the following
 topological lemma that will be proved in Appendix. The
reader may find a similar version of Alexander-Spanier (co-)homology
of this lemma (see \cite{Sp81}). The author thanks Le Calvez for the
proof.

\begin{lem}\label{lem:topology lemma of connected set}If $Z$ is a connected compact subset of
$M$ and $z\in Z$, then there is $\epsilon_0>0$ such that
\begin{equation*}
i_*(\pi_1(Z_\epsilon,z))=i_*(\pi_1(Z_{\epsilon_0},z))\quad \text{for
all } 0<\epsilon< \epsilon_0.
\end{equation*}
\end{lem}

\begin{proof}[Proof of Proposition \ref{prop:action is constant on the path connected
components}] Let $X$ be a connected component of
$\mathrm{Fix}_{\mathrm{Cont},I}(F)$ that is not a singleton. We will
first consider the linking number
$i(\widetilde{F};\widetilde{a},\widetilde{b},z)$ in the case where
$z\in \mathrm{Rec}^+(F)\setminus \mathrm{Fix}_{\mathrm{Cont},I}(F)$
and $(\widetilde{a},\widetilde{b})\in(\pi^{-1}(X)
\times\pi^{-1}(X))\setminus\widetilde{\Delta}$.

Recall the following functions  defined in Section \ref{subsec:the
definition of a new linking number}:
$$L_{k}: ((\mathrm{Fix}(\widetilde{F})\times
\mathrm{Fix}(\widetilde{F}))\setminus\widetilde{\Delta})\times
(\mathrm{Rec}^+(F)\cap U)\rightarrow \mathbb{Z},$$
\begin{equation*}
    L_{k}(\widetilde{F};\widetilde{c}_1,\widetilde{c}_2,z)=\widetilde{\gamma}\wedge
\widetilde{\Gamma}^k_{\widetilde{I}_1,z}=\sum\limits_{j=0}^{k-1}
L_1(\widetilde{F};\widetilde{c}_1,\widetilde{c}_2,\Phi^j(z)),
\end{equation*}
where $z\in \mathrm{Rec}^+(F)$ , $U\subset
M\setminus\pi(\{\widetilde{c}_1,\widetilde{c}_2\})$ is an open disk
containing $z$, and $\widetilde{I}_1$ is an isotopy from
$\mathrm{Id}_{\widetilde{M}}$ to $\widetilde{F}$ that fixes
$\widetilde{c}_1$ and $\widetilde{c}_2$.

We claim that, for every $z\in \mathrm{Rec}^+(F)\setminus
\mathrm{Fix}_{\mathrm{Cont},I}(F)$ and $k\geq1$, there exists
$\epsilon>0$ which merely depends on $z$ and $k$ such that
$L_{k}(\widetilde{F};\widetilde{a},\widetilde{b},z)=0$ when
$\widetilde{d}(\widetilde{a},\widetilde{b})<\epsilon$.

Indeed, since $X$ is compact, $z\not\in X$, and
$\widetilde{F}_t\circ T=T\circ\widetilde{F}_t$ for any $T\in G$, the
value
\begin{equation*}\label{formula:c'}
    c'=\min_{t\in[0,1],\ \widetilde{z}\in
C_z^k,\ \widetilde{z}'\in
\pi^{-1}(X)}|\widetilde{F}_t(\widetilde{z})-\widetilde{F}_t(\widetilde{z}')|
\end{equation*}
is positive and only depends on $z$ and $k$, where
$C_z^k=\pi^{-1}(\{z,F(z),\cdots,F^{\tau_k(z)-1}(z)\})$.

Recall that the isotopy
\begin{equation*}\label{formula:isotopies}
   \widetilde{I}'(\widetilde{z})(t)=\frac{\widetilde{b}-\widetilde{a}}{\widetilde{F}_t
   (\widetilde{b})-\widetilde{F}_t(\widetilde{a})}\cdot(\widetilde{F}_t(\widetilde{z})-
   \widetilde{F}_t(\widetilde{a}))
   +\widetilde{a}
\end{equation*}
fixes $\widetilde{a}$, $\widetilde{b}$ and $\infty$. Let
$\epsilon>0$ be small enough such that
$\max_{t\in[0,1]}|\widetilde{F}_t(\widetilde{a})-\widetilde{F}_t(\widetilde{b})|<c'/3$
when $\widetilde{d}(\widetilde{a},\widetilde{b})<\epsilon$, and let
$\widetilde{V}'$ be a disk whose center is $\widetilde{a}$ and
radius is $2|\widetilde{b}-\widetilde{a}|$. The claim follows from
the proof of Lemma \ref{lem:L1anaz} if one replaces
$\widetilde{I}_n$ in Formula \ref{formula:isotopy} by
$\widetilde{I}'$, $\widetilde{V}_n$ by $\widetilde{V}'$, and $c$ in
Formula \ref{formula:c} by $c'$.\smallskip

Fix $x\in X$ and a lift $\widetilde{x} \in \widetilde{M}$ of $x$. By
Lemma \ref{lem:topology lemma of connected set}, there is
$\epsilon_0>0$ such that for all $0<\epsilon< \epsilon_0$,
$$i_*(\pi_1(X_\epsilon,x))=i_*(\pi_1(X_{\epsilon_0},x)).$$ Suppose that $\widetilde{X}_\epsilon$ is
the connected component of $\pi^{-1}(X_\epsilon)$ that contains
$\widetilde{x}$. Let $G_{\widetilde{X}_{\epsilon}}$ be the
stabilizer of $\widetilde{X}_\epsilon$ in the group $G$, i.e.,
$G_{\widetilde{X}_{\epsilon}}=\{T\in G\mid
T(\widetilde{X}_\epsilon)=\widetilde{X}_\epsilon\}$. It is clear
that $i_*(\pi_1(X_\epsilon,x))\simeq G_{\widetilde{X}_{\epsilon}}$.
Hence
$G_{\widetilde{X}_{\epsilon_1}}=G_{\widetilde{X}_{\epsilon_2}}$ for
all $0<\epsilon_2<\epsilon_1\leq\epsilon_0$. Let
$\widetilde{Y}_\epsilon=\widetilde{X}_\epsilon\cap\pi^{-1}(X)$.
Recall that $X$ is connected. We have
$\pi(\widetilde{Y}_\epsilon)=X$ for all $0<\epsilon\leq\epsilon_0$
since $X_\epsilon$ is path connected. Note that
$\widetilde{Y}_\epsilon$ is $4\epsilon$-\emph{chain connected},
i.e., for any
$\widetilde{y},\widetilde{y}\,'\in\widetilde{Y}_\epsilon$ there
exists a sequence $\{\widetilde{y}_i\}_{i=1}^n\subset
\widetilde{Y}_\epsilon$ such that $\widetilde{y}_1=\widetilde{y}$,
$\widetilde{y}_n=\widetilde{y}\,'$, and
$\widetilde{d}(\widetilde{y}_i,\widetilde{y}_{i+1})<4\epsilon$.
Indeed, we can find a path $\gamma$ in $X_\epsilon$ from
$\pi(\widetilde{y})$ to $\pi(\widetilde{y}\,')$ and a lift
$\widetilde{\gamma}$ of $\gamma$ in $\widetilde{X}_\epsilon$ from
$\widetilde{y}$ to $\widetilde{y}\,'$. On the path
$\widetilde{\gamma}$, we choose a sequence
$\{\widetilde{x}_i\}_{i=1}^n\subset \widetilde{\gamma}$ such that
$\widetilde{x}_1=\widetilde{y}$, $\widetilde{x}_n=\widetilde{y}\,'$,
and the disks $\{\widetilde{D}(\widetilde{x}_i,\epsilon)\}_{i=1}^n$
cover $\widetilde{\gamma}$ with
$\widetilde{D}(\widetilde{x}_i,\epsilon)\cap\widetilde{D}(\widetilde{x}_{i+1},\epsilon)\neq
\emptyset$ for all $i=1,\ldots,n-1$, where
$\widetilde{D}(\widetilde{x}_i,\epsilon)$ is a disk on
$\widetilde{M}$ whose center is $\widetilde{x}_i$ and radius is
$\epsilon$. Choose a sequence $\{\widetilde{y}_i\}_{i=1}^{n}\subset
\widetilde{Y}_\epsilon$ such that $\widetilde{y}_1=\widetilde{y}$,
$\widetilde{y}_n=\widetilde{y}\,'$, and $\widetilde{y}_i\in
\widetilde{D}(\widetilde{x}_i,\epsilon)\cap \widetilde{Y}_\epsilon$
for $2\leq i\leq n-1$. Obviously, $\{\widetilde{y}_i\}_{i=1}^{n}$ is
a $4\epsilon$-chain in $\widetilde{Y}_\epsilon$ from $\widetilde{y}$
to $\widetilde{y}\,'$ by the triangle inequality.

For any $y\in X$, we claim that
$\widetilde{y}\in\widetilde{Y}_\epsilon$ for all $\widetilde{y}\in
\pi^{-1}(y)\cap\widetilde{Y}_{\epsilon_0}$ and all
$0<\epsilon\leq\epsilon_0$. Otherwise, there is
$0<\epsilon_1<\epsilon_0$ and
$\widetilde{y}\in\widetilde{Y}_{\epsilon_0}\subset
\widetilde{X}_{\epsilon_0}$ such that
$\widetilde{y}\not\in\widetilde{Y}_{\epsilon_1}$, and hence
$\widetilde{y}\not\in\widetilde{X}_{\epsilon_1}$. However, there is
a lift $\widetilde{y}\,'$ of $y$ such that
$\widetilde{y}\,'\in\widetilde{Y}_{\epsilon_1}\subset
\widetilde{X}_{\epsilon_1}\subset \widetilde{X}_{\epsilon_0}$. On
the one hand, $T\in G_{\widetilde{X}_{\epsilon_0}}$ since
$\widetilde{y},\widetilde{y}\,'\in \widetilde{X}_{\epsilon_0}$,
where $\widetilde{y}=T(\widetilde{y}\,')$. On the other hand,
$T\not\in G_{\widetilde{X}_{\epsilon_1}}$ since
$\widetilde{y}\not\in\widetilde{X}_{\epsilon_1}$. This is impossible
because
$G_{\widetilde{X}_{\epsilon_1}}=G_{\widetilde{X}_{\epsilon_0}}$ and
hence the claim holds. This implies that
$\widetilde{Y}_{\epsilon}=\widetilde{Y}_{\epsilon_0}$ for all
$0<\epsilon<\epsilon_0$, and thereby $\widetilde{Y}_{\epsilon_0}$ is
$\epsilon$-chain connected for all $0<\epsilon\leq
\epsilon_0/4$.\smallskip

Recall the equality in Proposition \ref{lem:i is 3coboundary for
point} for any distinct points $\widetilde{c}_1,\widetilde{c}_2$ and
$\widetilde{c}_3$ of $\mathrm{Fix}(\widetilde{F})$:
\begin{equation}\label{Eq:formula 4.3}
    L_{k}(\widetilde{F};\widetilde{c}_1,\widetilde{c}_2,z)+L_{k}(\widetilde{F};\widetilde{c}_2,
\widetilde{c}_3,z)+L_{k}(\widetilde{F};\widetilde{c}_3,\widetilde{c}_1,z)=0.
\end{equation}

Applying Equality \ref{Eq:formula 4.3}, we get that, for all
distinct points $\widetilde{a},\widetilde{b}\in
\widetilde{Y}_{\epsilon_0}$,
$L_k(\widetilde{F};\widetilde{a},\widetilde{b},z)=0$ for all $k$ and
$z\in \mathrm{Rec}^+(F)\setminus \mathrm{Fix}_{\mathrm{Cont},I}(F)$.
This implies that $i(\widetilde{F};\widetilde{a},\widetilde{b},z)=0$
for all $(\widetilde{a},\widetilde{b})\in
(\widetilde{Y}_{\epsilon_0}\times\widetilde{Y}_{\epsilon_0})\setminus\widetilde{\Delta}$
and $z\in \mathrm{Rec}^+(F)\setminus
\mathrm{Fix}_{\mathrm{Cont},I}(F)$.\smallskip

Let us now consider the case  $z\in
\mathrm{Fix}_{\mathrm{Cont},I}(F)$ to finish our proof, which is in
turn divided into two cases:
\begin{enumerate}
\item there is a set $\widetilde{X}$ on $\widetilde{M}$ which is a connected component of
$\pi^{-1}(X)$ and satisfies that the covering map
$\pi:\widetilde{X}\rightarrow X$ is surjective (this case contains
the case where $X$ is path connected);
\item there is no such set satisfying Item 1.
\end{enumerate}\smallskip

Recall the linking number of $z$ for $\widetilde{a}$ and
$\widetilde{b}$ (see Formula \ref{linking number of fixed points}):
$$i(\widetilde{F};\widetilde{a},\widetilde{b},z)=\sum_{\pi(\widetilde{z})=z}
\left(i(\widetilde{F};\widetilde{a},\widetilde{z})-i(\widetilde{F};\widetilde{b},\widetilde{z})\right),
$$
where
$(\widetilde{a},\widetilde{b})\in(\mathrm{Fix}(\widetilde{F})\times\mathrm{Fix}(\widetilde{F}))\setminus\widetilde{\Delta}$
and
$i(\widetilde{F};\widetilde{c},\widetilde{z})=i_{\widetilde{I}}(\widetilde{c},\widetilde{z})$
(see Formula \ref{eq:linking number for two fixed points}).

In the first case, for any $\widetilde{z}\in\pi^{-1}(z)$, by Lemma
\ref{rem:linking number on connected set},
$i(\widetilde{F};\widetilde{z}\,',\widetilde{z})$ is a constant
(which depends on $\widetilde{z}$) for all
$\widetilde{z}\,'\in\widetilde{X}\setminus\{\widetilde{z}\}$. We get
that $i(\widetilde{F};\widetilde{a},\widetilde{b},z)=0$ for any
$(\widetilde{a},\widetilde{b})\in(\widetilde{X}\times\widetilde{X})\setminus\widetilde{\Delta}$
and $z\in \mathrm{Fix}_{\mathrm{Cont},I}(F)\setminus
\pi(\{\widetilde{a},\widetilde{b}\})$.

Note that $\widetilde{Y}_{\epsilon_0}=\widetilde{X}$ in this case.
Therefore, by the definition of the action function, we get that
$i_{\mu}(\widetilde{F};\widetilde{a},\widetilde{b})=0$ for all
$(\widetilde{a},\widetilde{b})\in(\widetilde{X}\times\widetilde{X})\setminus\widetilde{\Delta}$.
The conclusion follows from the fact that $\pi(\widetilde{X})=X$ and
the hypothesis that $\rho_{M,I}(\mu)=0$ in this case.
\smallskip

In the second case, write $\pi^{-1}(X)$ as
$\bigsqcup_{\alpha\in\widetilde{\Lambda}}\widetilde{X}_{\alpha}$
where $\widetilde{X}_{\alpha}$ is a connected component of
$\pi^{-1}(X)$ on $\widetilde{M}$. Note that
$2\leq\sharp\widetilde{\Lambda}\leq+\infty$. 
It is easy to see that every such $\widetilde{X}_{\alpha}$ is
unbounded on $\widetilde{M}$ by the hypotheses and the connectedness
of $X$.

Similar to the proof of the first case, for every
$\alpha\in\widetilde{\Lambda}$ and
$\widetilde{c}\in\widetilde{X}_\beta$ with $\alpha\neq\beta$, the
following property holds: when $z\in
\mathrm{Fix}_{\mathrm{Cont},I}(F)$, the linking number
$i(\widetilde{F};\cdot,\widetilde{c},z)\in\mathbb{Z}$ is a constant
on $\widetilde{X}_\alpha$, and hence
$i(\widetilde{F};\widetilde{a},\widetilde{b},z)=0$ for all
$(\widetilde{a},\widetilde{b})\in(\widetilde{X}_\alpha\times\widetilde{X}_\alpha)\setminus\widetilde{\Delta}$.
Observing that every $\widetilde{X}_\alpha$ is unbounded, we have
that the constant is zero by Formula \ref{linking number of fixed
points} and Lemma \ref{lem:the linking number on a unbounded set}.
Therefore, $i(\widetilde{F};\widetilde{a},\widetilde{b},z)=0$ for
all
$(\widetilde{a},\widetilde{b})\in(\pi^{-1}(X)\times\pi^{-1}(X))\setminus\widetilde{\Delta}$.

Finally, by the definition of the action function, we get that
$i_{\mu}(\widetilde{F};\widetilde{a},\widetilde{b})=0$ for all
$(\widetilde{a},\widetilde{b})\in(\widetilde{Y}_{\epsilon_0}\times\widetilde{Y}_{\epsilon_0})\setminus\widetilde{\Delta}$.
The conclusion then follows from the facts that
$\pi(\widetilde{Y}_{\epsilon_0})=X$ and that $\rho_{M,I}(\mu)=0$ in
the second case.
\end{proof}\bigskip

\section{The proof of Theorem \ref{prop:F is not constant if the
contractible fixed points is finite}}\label{sec:F is not constant if
the contractible fixed points is finite}

To prove Theorem \ref{prop:F is not constant if the contractible
fixed points is finite}, we need the following theorem:
\begin{thm}[\cite{P1,Ma}]\label{thm:hamiltonian map has at least 3 contricible fixed
points} Let $M$ be a closed oriented surface with genus $g\geq1$. If
$F$ is the time-one map of a $\mu$-Hamiltonian isotopy $I$ on $M$,
then there exist at least three contractible fixed points of $F$.
\end{thm}

Remark that Theorem \ref{thm:hamiltonian map has at least 3
contricible fixed points} is not valid when the measure has no full
support (see Example \ref{exem:T2 and supp(u)notM} and Example
\ref{exem:M and supp(u)notM} below).\smallskip

\newenvironment{prop02}{\noindent\textbf{Theorem \ref{prop:F is not
constant if the contractible fixed points is
finite}}~\itshape}{\par}
\begin{prop02}Let $F$ be the
time-one map of a $\mu$-Hamiltonian isotopy $I$ on a closed oriented
surface $M$ with $g\geq 1$. If $I$ satisfies the WB-property and $F$
is not $\mathrm{Id}_{M}$, the action function $L_\mu$ is not
constant.
\end{prop02}

Theorem \ref{prop:F is not constant if the contractible fixed points
is finite} is proved in two cases: the set
$\mathrm{Fix}_{\mathrm{Cont},I}(F)$ is finite or infinite.

\begin{proof}[Proof of Theorem \ref{prop:F is not
constant if the contractible fixed points is finite} for the case
$\sharp\mathrm{Fix}_{\mathrm{Cont},I}(F)<+\infty$]\qquad\par

We say that $X\subseteq \mathrm{Fix}_{\mathrm{Cont},I}(F)$ is
\emph{unlinked}
 if there exists an isotopy $I'=(F'_t)_{t\in [0,1]}$ homotopic
to $I$ which fixes every point of $X$. Moreover, we say that $X$ is
a \emph{maximal unlinked set} if any set $X'\subseteq
\mathrm{Fix}_{\mathrm{Cont},I}(F)$ that strictly contains $X$ is not
unlinked.\smallskip

In the proof of Theorem \ref{thm:hamiltonian map has at least 3
contricible fixed points} (\cite[Theorem 10.1]{P1}), Le Calvez has
proved that there exists a maximal unlinked set $X\subseteq
\mathrm{Fix}_{\mathrm{Cont},I}(F)$ with $\sharp X\geq 3$ if
$\sharp\mathrm{Fix}_{\mathrm{Cont},I}(F)<+\infty$.\smallskip

There exists an oriented topological foliation $\mathcal{F}$ on
$M\setminus X$ (or, equivalently, a singular oriented foliation
$\mathcal{F}$ on $M$ with $X$ equal to the singular set) such that,
for all $z\in M\setminus X$, the trajectory $I(z)$ is homotopic to
an arc $\gamma$ joining $z$ and $F(z)$ in $M\setminus X$, which is
positively transverse to $\mathcal{F}$. It means that for every
$t_0\in[0,1]$ there exists an open neighborhood $V\subset M\setminus
X$ of $\gamma(t_0)$ and an orientation preserving homeomorphism
$h:V\rightarrow(-1,1)^2$ which sends the foliation $\mathcal{F}$ on
the horizontal foliation (oriented with $x_1$ increasing) such that
the map $t\mapsto p_2(h(\gamma(t)))$ defined in a neighborhood of
$t_0$ is strictly increasing, where $p_2(x_1,x_2)=x_2$.\smallskip

We can choose a point $z\in
\mathrm{Rec}^+(F)\setminus\mathrm{Fix}(F)$ and a leaf $\lambda$
containing $z$. Proposition 10.4 in \cite{P1} states that the
$\omega$-limit set $\omega(\lambda)\in X$, the $\alpha$-limit set
$\alpha(\lambda)\in X$, and $\omega(\lambda)\neq\alpha(\lambda)$.
Fix an isotopy $I'$ homotopic to $I$ that fixes $\omega(\lambda)$
and $\alpha(\lambda)$ and a lift $\widetilde{\lambda}$ of $\lambda$
that joins $\widetilde{\omega(\lambda)}$ and
$\widetilde{\alpha(\lambda)}$. Let us now study the linking number
$i(\widetilde{F};\widetilde{\omega(\lambda)},\widetilde{\alpha(\lambda)},z')$
for $z'\in \mathrm{Rec}^+(F)\setminus X$ when it exists. Observing
that for all $z'\in M\setminus X$, the trajectory $I'(z')$ is still
homotopic to an arc that is positively transverse to $\mathcal{F}$.
Hence, for all $z'\in \mathrm{Rec}^+(F)\setminus X$, without loss of
generality, we can choose an open disk $U$ containing $z'$ such that
$U\cap \lambda=\emptyset$ by shrinking $U$ and perturbing $\lambda$
if necessary. Then we get
\begin{equation*}\label{eq: Ln is not negtive}
    L_n(\widetilde{F};\widetilde{\omega(\lambda)},\widetilde{\alpha(\lambda)},z')
=\widetilde{\lambda}\wedge\widetilde{\Gamma}^n_{\widetilde{I}\,',z'}=
\lambda\wedge\Gamma^n_{I\,',z'}\geq0
\end{equation*}
for every $n\geq1$, where $\widetilde{I}\,'$ is the lift of $I'$ to
$\widetilde{M}$ and
$\Gamma^n_{I\,',z'}=\pi(\widetilde{\Gamma}^n_{\widetilde{I}\,',z'})$.

According to Definition \ref{def:Intersection number density}, we
have
$$i(\widetilde{F};\widetilde{\omega(\lambda)},\widetilde{\alpha(\lambda)},z')\geq0$$
for $\mu$-a.e.
$z'\in\mathrm{Rec}^+(F)\setminus\{\omega(\lambda),\alpha(\lambda)\}$.\smallskip

By the continuity of $I'$ and the hypothesis on $\mu$, there exists
an open free disk $U$ containing $z$ such that $\mu(U)>0$ and
$L_1(\widetilde{F};\widetilde{\omega(\lambda)},\widetilde{\alpha(\lambda)},z')>0$
when $z'\in U\cap\mathrm{Rec}^+(F)$.

Similarly to the proof of Proposition \ref{prop:the continuity of
lmu}, we obtain
\begin{eqnarray*}
I_\mu(\widetilde{F};\omega(\lambda),\alpha(\lambda))&\geq&
\int_{\bigcup_{k\geq0}F^k(U)}i(\widetilde{F};\widetilde{\omega(\lambda)},\widetilde{\alpha(\lambda)},z)\,\mathrm{d}\mu \\
  &=& \int_{U}\tau(z)i(\widetilde{F};\widetilde{\omega(\lambda)},\widetilde{\alpha(\lambda)},z)\,\mathrm{d}\mu \\
   &=&\int_{U}\tau^{*}(z)i(\widetilde{F};\widetilde{\omega(\lambda)},\widetilde{\alpha(\lambda)},z)\,\mathrm{d}\mu \\
   &=&\int_{U}L^{*}(\widetilde{F};\widetilde{\omega(\lambda)},\widetilde{\alpha(\lambda)},z)\,\mathrm{d}\mu  \\
   &=&\int_{U}L_1(\widetilde{F};\widetilde{\omega(\lambda)},\widetilde{\alpha(\lambda)},z)\,\mathrm{d}\mu \\
   &>& 0.
\end{eqnarray*}
\end{proof}

Before proving the case where the set
$\mathrm{Fix}_{\mathrm{Cont},I}(F)$ is infinite, let us recall two
results:

\begin{prop}[Franks' Lemma \cite{F}]\label{prop:Franks' Lemma}
Let $F: \mathbb{R}^2\rightarrow \mathbb{R}^2$ be an orientation
preserving homeomorphism. If $F$  possesses a  periodic free disk
chain, a family $(U_{r})_{r\in\mathbf{Z}/n\mathbf{Z}}$ of pairwise
disjoint free topological open disks, such that for every
$r\in\mathbf{Z}/n\mathbf{Z}$, one of the positive iterates of
$U_{r}$ meets $U_{r+1}$, then $F$ has at least one fixed point.
\end{prop}

\begin{thm}[\cite{J}]\label{thm:O.Jaulent}
Let $M$ be an oriented surface and $F$ be the time-one map of an
identity isotopy $I$ on $M$. There exists a closed subset $X\subset
\mathrm{Fix}(F)$ and an isotopy $I'$ in $\mathrm{Homeo(M\setminus
X)}$ joining $\mathrm{Id}_{M\setminus X}$ to $F|_{M\setminus X}$
such that
\begin{enumerate}
  \item For all $z\in X$, the loop $I(z)$ is homotopic to zero in
  $M$.
  \item For all $z\in\mathrm{Fix}(F)\setminus X$, the loop $I'(z)$
  is not homotopic to zero in $M\setminus X$.
  \item For all $z\in M\setminus X$, the trajectories $I(z)$ and
  $I'(z)$ are homotopic with fixed endpoints in $M$.
  \item There exists an oriented topological foliation $\mathcal{F}$
  on $M\setminus X$ such that, for all $z\in M\setminus X$, the trajectory
  $I'(z)$ is homotopic to an arc $\gamma$ joining $z$ and $F(z)$ in
  $M\setminus X$
  which is positively to $\mathcal{F}$.
\end{enumerate}
Moreover, the isotopy $I'$ satisfies the following property:

\begin{enumerate}
\item[(5)] For all finite $Y\subset X$, there exists an isotopy $I_Y'$
joining $\mathrm{Id}_M$ and $F$ in $\mathrm{Homeo}(M)$ which fixes
$Y$ such that, if $z\in M\setminus X$, the arc $I'(z)$ and $I_Y'(z)$
are homotopic in $M\setminus Y$. And if $z\in X\setminus Y$, the
loop $I_Y'(z)$ is contractible in $M\setminus Y$.
\end{enumerate}
\end{thm}

\begin{proof}[Proof of Theorem \ref{prop:F is not
constant if the contractible fixed points is finite} for the case
$\sharp\mathrm{Fix}_{\mathrm{Cont},I}(F)=+\infty$]\qquad\par\smallskip

Suppose that $X$, $I'$ and $\mathcal {F}$ are respectively the
closed contractible
 fixed points set, the isotopy, and  the foliation, as stated in Theorem
 \ref{thm:O.Jaulent}. Obviously, $X\neq\emptyset$ (see Remark \ref{rem: contractible fixed point and isotopy}) and $\mu(M\setminus X)>0$.
 Assume that $X'$ is the union of the connected
 components of $X$ that separate $M$. Write $M\setminus X'=\sqcup_i S_i$ where  $S_i$ is an open $F$-invariant
 connected subsurface of $M$ (see \cite{BK}). For every $i$, we denote by $I'_i$ the restriction of $I'$ on
$S_i\setminus (S_i\cap X)$. We claim that
 $$\sharp \{\text{the connected components of}\,\,
\partial S_i \cup (S_i\cap X)\} \geq 2\quad \text{for every i}.$$

To prove this claim, we can suppose that $\sharp \{\text{the
connected components of}\,\,
\partial S_i \cup (X\cap S_i)\}$ is finite. If $S_i$ is a
subsurface of the sphere, we only need to consider the case of disk.
In this case, by Proposition \ref{prop:Franks' Lemma} and Item 2 of
Theorem~\ref{thm:O.Jaulent}, $X\cap S_i\neq \emptyset$ and thereby
the claim follows. When $S_i$ is not a subsurface of the sphere, we
can get a closed surface $S'_i$ through compactifying $S_i$. More
precisely, we add one point on each connected component of $\partial
S_i\cup (S_i\cap X)$. Note that $S_i\setminus (S_i\cap X)$ is
embedded in $S_i'$ and we can extend $I'_i$ on $S_i\setminus
(S_i\cap X)$ to an identity isotopy on $S_i'$ which is still denoted
by $I'_i$. To be more precise, let $X_i'=S'_i\setminus (S_i\setminus
(S_i\cap X))$ be the added points set. Observing that $X_i'$ is
totally connected (in fact finite), the isotopy $I_i'$ can be
extended to an identity isotopy on $S_i'$ that fixes every point in
$X_i'$ (see \cite[Remark 1.18]{J}). By the definitions of $S_i'$ and
 $I'_i$, and the items 1 and 3 of Theorem~\ref{thm:O.Jaulent}, we
get $\rho_{S'_i,I'_i}(\mu)=0\in H_1(S'_i,\mathbb{R})$. According to
Item 2 of Theorem~\ref{thm:O.Jaulent}, $X'_i$ is a maximal unlink
set of $I'_i$. Thanks to the proof of Theorem \ref{thm:hamiltonian
map has at least 3 contricible fixed points} (\cite[Theorem
10.1]{P1}), we have $\sharp X_i'\geq3$. Therefore, the claim
holds.\smallskip

Fix one connected set $S_i$. Similarly to the finite case, we choose
a point $z\in (\mathrm{Rec}^+(F)\setminus\mathrm{Fix}(F))\cap S_i$
and a leaf $\lambda\in\mathcal {F}$ containing $z$. In \cite{P3},
the proofs of Proposition 4.1 and 4.3 (page 150 and 152, for $S_i$
being a subset of the sphere) and of Proposition 6.1 (page 166, for
$S_i$ being not a subset of the sphere)  imply that the $\omega$-set
of $\lambda$, $\omega(\lambda)$ (resp. the $\alpha$-set of
$\lambda$, $\alpha(\lambda)$), is connected and is contained in a
connected component of $\partial S_i\cup(X\cap S_i)$. We write the
connected component as $X_+(\lambda)$ (resp. $X_-(\lambda)$).
Moreover, $X_+(\lambda)\neq X_-(\lambda)$. Choose a lift
$\widetilde{\lambda}$ of $\lambda$. We need to consider the
following four cases: the set $\omega(\widetilde{\lambda})$ or
$\alpha(\widetilde{\lambda})$ contains $\infty$ or not.\smallskip

 Take two points $a\in\alpha(\lambda)$ and
$b\in\omega(\lambda)$. Let $Y=\{a,b\}$ and $I'_Y$ be the isotopy as
 in Theorem \ref{thm:O.Jaulent}. Suppose that
$\widetilde{I}\,'_Y$ is the identity lift of $I'_Y$ to
$\widetilde{M}$. Notice that
\begin{description}
  \item[(A1)] if $z\in M\setminus X$, then the arcs $I'(z)$
and $I_Y'(z)$ are homotopic in $M\setminus Y$ (Item 5, Theorem
\ref{thm:O.Jaulent}), and  $I_Y'(z)$ is homotopic to an arc $\gamma$
from $z$ to $F(z)$ in $M\setminus Y$ and positively transverse to
$\mathcal{F}$ (Item 4, Theorem \ref{thm:O.Jaulent});
  \item[(A2)] if $z\in X\setminus Y$, then $\gamma\wedge I_Y'(z)=0$
 where $\gamma$ is any path from $a$ to $b$ (Item 5, Theorem \ref{thm:O.Jaulent}).
\end{description}

We now suppose that neither $\alpha(\widetilde{\lambda})$ nor
$\omega(\widetilde{\lambda})$ contains $\infty$. Replacing
$\alpha(\lambda)$ by $a$, $\omega(\lambda)$ by $b$, and $I'$ by
$I'_Y$ in the proof of the finite case, we can get
$I_\mu(\widetilde{F};a,b)>0$.

We now consider the case that either $\alpha(\widetilde{\lambda})$
or $\omega(\widetilde{\lambda})$ contains $\infty$. Recall that
$\widetilde{d}$ is the distance on $\widetilde{M}$ induced by
 a distance $d$ on $M$ which is further induced by a Riemannian metric on $M$. Define
$\widetilde{d}(\widetilde{z},\widetilde{C})=\inf\limits_{\widetilde{c}\in\widetilde{C}}
       \widetilde{d}(\widetilde{z},\widetilde{c})$, if
       $\widetilde{z}\in \widetilde{M}$ and $\widetilde{C}\subset
       \widetilde{M}$. Take a sequence
$\{(\widetilde{a}_m,\widetilde{b}_m)\}_{m\geq1}$ such that
\begin{itemize}
       \item $\pi(\widetilde{a}_m)=a$ and $\pi(\widetilde{b}_m)=b$\,\,;
        \item if $\alpha(\widetilde{\lambda})$
        (resp. $\omega(\widetilde{\lambda})$) does not contain $\infty$, we set $\widetilde{a}_m=\widetilde{a}$
         (resp. $\widetilde{b}_m=\widetilde{b})$ for every $m$ where $\widetilde{a}\in\pi^{-1}(a)\cap\alpha(\widetilde{\lambda})$
         (resp. $\widetilde{b}\in\pi^{-1}(b)\cap\omega(\widetilde{\lambda})$);
       \item
       $\lim\limits_{m\rightarrow+\infty}\widetilde{d}(\widetilde{a}_m,\widetilde{\lambda})=0$
       and
       $\lim\limits_{m\rightarrow+\infty}\widetilde{d}(\widetilde{b}_m,\widetilde{\lambda})=0$.
     \end{itemize}

For every $m$, there exists $\widetilde{c}_m$ (resp.
$\widetilde{c}\,'_m$) on $\widetilde{\lambda}$ such that
$\widetilde{d}(\widetilde{a}_m,\widetilde{c}_m)=\widetilde{d}(\widetilde{a}_m,\widetilde{\lambda})$
(resp.
$\widetilde{d}(\widetilde{b}_m,\widetilde{c}\,'_m)=\widetilde{d}(\widetilde{b}_m,\widetilde{\lambda})$).
Note that $\widetilde{c}_m=\widetilde{a}_m=\widetilde{a}$ (resp.
$\widetilde{c}\,'_m=\widetilde{b}_m=\widetilde{b}$) and
$\widetilde{d}(\widetilde{a}_m,\widetilde{\lambda})=0$ (resp.
$\widetilde{d}(\widetilde{b}_m,\widetilde{\lambda})=0$) if
$\alpha(\widetilde{\lambda})$ (resp. $\omega(\widetilde{\lambda})$)
does not contain $\infty$. Choose a simple smooth path
$\widetilde{l}_m$ (resp. $\widetilde{l}\,'_m$) from
$\widetilde{a}_m$ (resp. $\widetilde{c}\,'_m$) to $\widetilde{c}_m$
(resp. $\widetilde{b}_m$) such that the length of $\widetilde{l}_m$
(resp. $\widetilde{l}\,'_m$) converges to zero as
$m\rightarrow+\infty$ and
$\pi(\widetilde{l}_{m+1})\subset\pi(\widetilde{l}_{m})$ (resp.
$\pi(\widetilde{l}\,'_{m+1})\subset\pi(\widetilde{l}\,'_{m})$).
Here, we assume that the simple path $\widetilde{l}_m$ (resp.
$\widetilde{l}\,'_m$) is empty if $\alpha(\widetilde{\lambda})$
(resp. $\omega(\widetilde{\lambda})$) does not contain $\infty$. Let
$\widetilde{\gamma}_m=\widetilde{l}_m\widetilde{\lambda}_m\widetilde{l}\,'_m$,
where $\widetilde{\lambda}_m$ is the sub-path of
$\widetilde{\lambda}$ from $\widetilde{c}_m$ to
$\widetilde{c}\,'_m$. Then $\widetilde{\gamma}_m$ is a path from
$\widetilde{a}_m$ to $\widetilde{b}_m$.\smallskip


We know that, for every $m\geq1$, the linking number
$i(\widetilde{F};\widetilde{a}_m,\widetilde{b}_m,z')$ exists for
$\mu$-a.e. $z'\in M\setminus\{a,b\}$. Hence, the linking number
$i(\widetilde{F};\widetilde{a}_m,\widetilde{b}_m,z')$ exists on a
full measure subset of $M\setminus\{a,b\}$ for all $m$.\smallskip

According to \textbf{A2} above, we have
$i(\widetilde{F};\widetilde{a}_m,\widetilde{b}_m,z')=0$ if $z'\in
X\setminus\{a,b\}$. To finish Theorem \ref{prop:F is not constant if
the contractible fixed points is finite}, we need the following
lemma whose proof will be provided afterwards.

\begin{lem}\label{lem:inf geq 0}
$\liminf\limits_{m\rightarrow+\infty}
i(\widetilde{F};\widetilde{a}_m,\widetilde{b}_m,z')\geq0\quad\mathrm{for}\quad\mu\text{-}\mathrm{a.
e.}\quad z'\in \mathrm{Rec}^+(F)\setminus X.$\end{lem}

Armed with this lemma we obtain
\begin{equation}\label{liminf i(F,am,bm,z)geq 0}
    \liminf\limits_{m\rightarrow+\infty}i(\widetilde{F};\widetilde{a}_m,\widetilde{b}_m,z')\geq0
    \quad\mathrm{for}\quad
    \mu\text{-}\mathrm{a.e.}\quad z'\in\mathrm{Rec}^+(F)\setminus\{a,b\}.
\end{equation}

From the continuity of $I_Y'$ and the hypothesis on $\mu$, there
exists an open free disk $U$ containing $z$ such that $\mu(U)>0$ and
for $z'\in U\cap\mathrm{Rec}^+(F)$,
\begin{equation}\label{L1 >0}
    \lim\limits_{m\rightarrow+\infty}L_1(\widetilde{F};\widetilde{a}_m,\widetilde{b}_m,z')>0.
\end{equation}

As $\rho_{M,I}(\mu)=0$, by Proposition \ref{clm:L is well defined},
the inequalities \ref{liminf i(F,am,bm,z)geq 0} and \ref{L1 >0}, and
Fatou's lemma, we have
\begin{eqnarray*}
I_\mu(\widetilde{F};a,b)&=&\lim_{m\rightarrow+\infty}i_\mu(\widetilde{F};\widetilde{a}_m,\widetilde{b}_m)\\
&=&\lim_{m\rightarrow+\infty}\int_{M\setminus\{a,b\}}i(\widetilde{F};\widetilde{a}_m,\widetilde{b}_m,z)\,\mathrm{d}\mu\\&\geq&
\int_{M\setminus\{a,b\}}\liminf_{m\rightarrow+\infty}i(\widetilde{F};\widetilde{a}_m,\widetilde{b}_m,z)\,\mathrm{d}\mu\\&\geq&
\int_{\bigcup_{k\geq0}F^k(U)}\liminf_{m\rightarrow+\infty}i(\widetilde{F};\widetilde{a}_m,\widetilde{b}_m,z)\,\mathrm{d}\mu \\
  &=& \int_{U}\liminf_{m\rightarrow+\infty}\tau(z)i(\widetilde{F};\widetilde{a}_m,\widetilde{b}_m,z)\,\mathrm{d}\mu \\
   &=&\int_{U}\liminf_{m\rightarrow+\infty}\tau^{*}(z)i(\widetilde{F};\widetilde{a}_m,\widetilde{b}_m,z)\,\mathrm{d}\mu \\
   &=&\int_{U}\liminf_{m\rightarrow+\infty}L^{*}(\widetilde{F};\widetilde{a}_m,\widetilde{b}_m,z)\,\mathrm{d}\mu  \\
   &=&\int_{U}\liminf_{m\rightarrow+\infty}L_1(\widetilde{F};\widetilde{a}_m,\widetilde{b}_m,z)\,\mathrm{d}\mu \\
   &>& 0.
\end{eqnarray*}

Therefore, we only need to prove Lemma \ref{lem:inf geq 0}.
\begin{proof}[Proof of Lemma \ref{lem:inf geq 0}]Fix one point $z'\in \mathrm{Rec}^+(F)\setminus X$ and choose
an open disk $U$ containing $z'$ (here again, we suppose that $U\cap
\lambda=\emptyset$). By \textbf{A1} and the construction of
$\widetilde{\gamma}_m$, for every $n\geq1$, there exists
$m(z',n)\in\mathbb{N}$ such that when $m\geq m(z',n)$, the value
\begin{equation}\label{Ln geq 0}L_n(\widetilde{F};\widetilde{a}_m,\widetilde{b}_m,z')=
\widetilde{\gamma}_m\wedge\widetilde{\Gamma}^n_{\widetilde{I}\,'_Y,z'}
=\pi(\widetilde{\gamma}_m)\wedge\Gamma^n_{I\,'_Y,z'}\geq0\end{equation}
is constant with regard to $m$. \smallskip

We prove this lemma by contradiction. Suppose that $$\mu\{z'\in
\mathrm{Rec}^+(F)\setminus X\mid
\liminf\limits_{m\rightarrow+\infty}
i(\widetilde{F};\widetilde{a}_m,\widetilde{b}_m,z')<0\}>0.$$

There exists a small number $c>0$ such that
\begin{equation}\label{mu(E)>c}
    \mu\{z'\in
\mathrm{Rec}^+(F)\setminus X\mid
\liminf\limits_{m\rightarrow+\infty}
i(\widetilde{F};\widetilde{a}_m,\widetilde{b}_m,z')<-c\}>c.
\end{equation}

Write $E=\{z'\in \mathrm{Rec}^+(F)\setminus X\mid
\liminf\limits_{m\rightarrow+\infty}
i(\widetilde{F};\widetilde{a}_m,\widetilde{b}_m,z')<-c\}$. Fix a
point $z'\in E$ and an open disk $U$ containing $z'$ as before. By
taking a subsequence if necessary, we can suppose that
$$-\infty\leq\lim\limits_{m\rightarrow+\infty}
i(\widetilde{F};\widetilde{a}_m,\widetilde{b}_m,z')<-c.$$ Then there
exists $N(z')$ such that for $m\geq N(z')$,
$$
i(\widetilde{F};\widetilde{a}_m,\widetilde{b}_m,z')=\lim_{n\rightarrow+\infty}
\frac{L_n(\widetilde{F};\widetilde{a}_m,\widetilde{b}_m,z')}{\tau_n(z')}<-c.$$
Fix $m_0\geq N(z')$. There exists $n(z',m_0)\in\mathbb{N}$ such that
when $n\geq n(z',m_0)$,
$$\frac{L_n(\widetilde{F};\widetilde{a}_{m_0},\widetilde{b}_{m_0},z')}{\tau_n(z')}<-c.$$
Then we can choose  $n_0\geq n(z',m_0)$ such that
$$L_{n_0}(\widetilde{F};\widetilde{a}_{m_0},\widetilde{b}_{m_0},z')<-c\tau_{n_0}(z').$$
Based on Inequality \ref{Ln geq 0}, there exists $m(z',n_0)>m_0$
such that when $m\geq m(z',n_0)$,
$$L_{n_0}(\widetilde{F};\widetilde{a}_m,\widetilde{b}_m,z')\geq0.$$
By fixing $m_1\geq m(z',n_0)$, there exists $n(z',m_1)>n_0$ such
that when $n\geq n(z',m_1)$,
$$\frac{L_n(\widetilde{F};\widetilde{a}_{m_1},\widetilde{b}_{m_1},z')}{\tau_n(z')}<-c.$$
Then  we can choose $n_1\geq n(z',m_1)$ such that
$$L_{n_1}(\widetilde{F};\widetilde{a}_{m_1},\widetilde{b}_{m_1},z')<-c\tau_{n_1}(z').$$

By induction, we can construct a sequence
$\{(m_i,n_i)\}_{i\geq0}\subset \mathbb{N}\times\mathbb{N}$
satisfying that
\begin{description}
  \item[(B1)]\, $\{m_i\}_{i\geq0}$ and $\{n_i\}_{i\geq0}$ are strictly increasing
sequences;
  \item[(B2)]\, for all $i\geq0$, we have
  $$L_{n_i}(\widetilde{F};\widetilde{a}_{m_i},\widetilde{b}_{m_i},z')<-c\tau_{n_i}(z')\quad
\mathrm{and} \quad
L_{n_i}(\widetilde{F};\widetilde{a}_{m_{i+1}},\widetilde{b}_{m_{i+1}},z')\geq0.$$
\end{description}

According to the positively transverse property of $\mathcal {F}$,
it is clear that the negative part of the value
$L_{n_i}(\widetilde{F};\widetilde{a}_{m_i},\widetilde{b}_{m_i},z')$
comes from at least one of the algebraic intersection numbers of the
curve $\widetilde{\Gamma}^n_{\widetilde{I}\,'_Y,z'}$ with
$\widetilde{l}_{m_i}$ and $\widetilde{l}\,'_{m_i}$.\smallskip

We deal with the case where both $\alpha(\widetilde{\lambda})$ and
$\omega(\widetilde{\lambda})$ contain $\infty$, and other cases
follow similarly. In this case, the both sets
$\alpha(\lambda),\omega(\lambda)\subset X$ are not contractible.
According to Item 5 of Theorem \ref{thm:O.Jaulent}, for any $z''\in
M\setminus Y$, the loop $I_Y'^{-1}I'(z'')$ is contractible in
$M\setminus Y$ (see Section \ref{sec:Identity isotopies} for the
definition of $I_Y'^{-1}$). It implies that
$\widetilde{\gamma}_m\wedge I_Y'^{-1}I'(z'')=0$ for all $m$ and
$z''\in\mathrm{Rec}^+(F)\setminus X$. Note that $I_Y'$ fixes $a$ and
$b$, the loop $I_Y'(z'')$ is contractible in $M\setminus Y$ for any
$z''\in X\setminus Y$, but
 $\alpha(\lambda)$ and $\omega(\lambda)$ are not contractible. By
the continuity of $I_Y'$, we get $|\pi(\widetilde{l}_{m_i})\wedge
I_Y'(x)|\leq1$ (resp. $|\pi(\widetilde{l}\,'_{m_i})\wedge
I_Y'(x)|\leq1$) if the algebraic intersection number is
defined\footnote{\,The algebraic intersection number of
$\pi(\widetilde{l}_{m_i})\wedge I_Y'(x)$ (resp.
$\pi(\widetilde{l}\,'_{m_i})\wedge I_Y'(x)$) is well defined for
$\mu$-a.e. $x\in \mathrm{Rec}^+(F)$ if
$\mu(\pi(\widetilde{l}_{m_i}))=0$ (resp.
$\mu(\pi(\widetilde{l}\,'_{m_i}))=0$), which can be easily done by
slightly perturbing $\lambda$ and $\widetilde{l}_{m_i}$ (resp.
$\widetilde{l}\,'_{m_i}$) if necessary.} and $x$ is close to $a$
(resp. $b$). Based on the construction of $\widetilde{\lambda}_m$,
\textbf{B1} and \textbf{B2}, there must be a sequence of open disks
$\{U^a_i\}_{i\geq0}$ containing the set
$(I_Y')^{-1}(\pi(\widetilde{l}_{m_i}))=\cup_{y\in\pi(\widetilde{l}_{m_i})}(I_Y')^{-1}(y)$
(resp. $\{U^b_i\}_{i\geq0}$ containing the set
$(I_Y')^{-1}(\pi(\widetilde{l}\,'_{m_i}))=\cup_{y\in\pi(\widetilde{l}\,'_{m_i})}(I_Y')^{-1}(y)$)
that satisfies
\begin{description}
  \item[(C1)] $U^a_{i+1}\subset U^a_{i}$ (resp. $U^b_{i+1}\subset U^b_{i}$)
 and $\mu(U^a_i)\rightarrow0$ (resp. $\mu(U^b_i)\rightarrow0$) as
 $i\rightarrow+\infty$ (since the measure $\mu$ has no
atoms on $\mathrm{Fix}_{\mathrm{Cont},I}(F)$);
  \item[(C2)] for every $i\geq0$,
 $$\frac{1}{\tau_{n_i}(z')}\sum_{j=0}^{\tau_{n_i}(z')-1}\chi_{U^a_i}\circ
 F^j(z')>\frac{c}{2}\quad \mathrm{or} \quad\frac{1}{\tau_{n_i}(z')}\sum_{j=0}^{\tau_{n_i}(z')-1}\chi_{U^b_i}\circ
 F^j(z')>\frac{c}{2},$$ where $\chi_U$ is the indicator function
 of $U\subset M$.
\end{description}

Denote by $\chi_{U}^*(x)$ the limit of
$\frac{1}{n}\sum\limits_{j=0}^{n-1}\chi_{U}\circ
 F^j(x)$ as $n\rightarrow+\infty$ for $\mu$-a.e. $x\in M$ (due to Birkhoff Ergodic theorem). By
 \textbf{C2}
and Inequality~\ref{mu(E)>c}, for each $i$, we have
 $$\mu(\{x\in \mathrm{Rec}^+(F)\setminus X\mid\chi_{U^a_i}^*(x)\geq\frac{c}{2}\quad\mathrm{or}
 \quad\chi_{U^b_i}^*(x)\geq\frac{c}{2}\})>c.$$ This implies that
 $\int_M(\chi_{U^a_i}^*(x)+\chi_{U^b_i}^*(x))\mathrm{d}\mu\geq\frac{c^2}{2}>0$ for every $i$. On the
 other hand, thanks to Birkhoff Ergodic theorem and \textbf{C1}, we have $$\int_M(\chi_{U^a_i}^*(x)+\chi_{U^b_i}^*(x))\mathrm{d}\mu
 =\int_M(\chi_{U^a_i}(x)+\chi_{U^b_i}(x))\mathrm{d}\mu=\mu(U^a_i)+\mu(U^b_i)\rightarrow0$$
 as $i\rightarrow+\infty$, which gives a contradiction.
We have finished the proof of Lemma \ref{lem:inf geq 0}.\end{proof}
We have completed the proof of Theorem \ref{prop:F is not constant
if the contractible fixed points is finite}.
\end{proof}

\section{The proof of Theorem \ref{cor:the symplectic action when M with genus bigger
1}}

\newenvironment{cor04}{\noindent\textbf{Theorem \ref{cor:the symplectic action when M with genus bigger
1}}~\itshape}{\par}
\begin{cor04}\label{cor:g is great than 1}Let $F$ be the time-one map of an identity isotopy $I$ on a closed oriented
surface $M$ with $g>1$. If $I$ satisfies the WB-property, $F\in
\mathrm{Homeo}(M)\setminus\{\,\mathrm{Id}_{M}\}$, and
$\mu\in\mathcal {M}(F)$ has full support, then
 there exist two distinct fixed points $\widetilde{a}$ and
$\widetilde{b}$ of $\widetilde{F}$ such that
$i_\mu(\widetilde{F};\widetilde{a},\widetilde{b})\neq0$.
\end{cor04}
\begin{proof}
If $\rho_{M,I}(\mu)=0$, by Theorem \ref{prop:F is not constant if
the contractible fixed points is finite}, there exist two distinct
contractible fixed points $a$ and $b$ of $F$ such that
$I_{\mu}(\widetilde{F};a,b)\neq0$, thus for any their lifts
$\widetilde{a}$ and $\widetilde{b}$ we have
$i_{\mu}(\widetilde{F};\widetilde{a},\widetilde{b})=I_{\mu}(\widetilde{F};a,b)\neq
0$.\smallskip

If $\rho_{M,I}(\mu)\neq0$, there exists $\alpha\in G$ such that
$\varphi(\alpha)\wedge\rho_{M,I}(\mu) \neq0$, where $\varphi$ is the
Hurewitz homomorphism from $G$ to $H_1(M,\mathbb{Z})$. By
Lefschetz-Nielsen's formula, we know that
$\mathrm{Fix}_{\mathrm{Cont},I}(F)\neq \emptyset$. Choose $a\in
\mathrm{Fix}_{\mathrm{Cont},I}(F)$ and a lift $\widetilde{a}$ of
$a$. There exists an isotopy $I'$ homotopic to $I$ that fixes $a$
(see Remark \ref{rem: contractible fixed point and isotopy}). It is
lifted to an isotopy $\widetilde{I}\,'$ that fixes $\widetilde{a}$
and $\alpha(\widetilde{a})$. Let $z\in \mathrm{Rec}^+(F)$ and $U$ be
an open disk that contains $z$. Observe that if $\widetilde{\gamma}$
is an oriented path from $\widetilde{a}$ to $\alpha(\widetilde{a})$,
then the intersection number
$\widetilde{\gamma}\wedge\widetilde{\Gamma}^n_{\widetilde{I}\,',z}$
is equal to the intersection of the loop $\pi(\widetilde{\gamma})$
with the loop $\Gamma^n_{I',z}=I'\,^{\tau_n(z)}(z)\gamma_{\Phi^n(z),
z}$ (see Section \ref{subsec:the definition of a new linking
number}). Recall the fact that
$$\lim_{n\rightarrow+\infty}\frac{[\Gamma^n_{I',z'}]}{\tau_n(z')}=\rho_{M,I}(z')$$
for $\mu$-a.e. $z'\in U$ (see \cite[pages 54-56]{W2}). We get that
\begin{eqnarray*}i_\mu(\widetilde{F};\widetilde{a},\alpha(\widetilde{a}))&=&
\int_{M\setminus\{\pi(\widetilde{a})\}}i(\widetilde{F};\widetilde{a},\alpha(\widetilde{a}),z)\,\mathrm{d}\mu\\
&=&
\int_{M\setminus\{\pi(\widetilde{a})\}}\lim_{n\rightarrow+\infty}\frac{L_n(\widetilde{F};
\widetilde{a},\alpha(\widetilde{a}),z)}{\tau_n(z)}\,\mathrm{d}\mu\\
&=&\int_{M\setminus\{\pi(\widetilde{a})\}}\lim_{n\rightarrow+\infty}\frac{\widetilde{\gamma}
\wedge\widetilde{\Gamma}^n_{\widetilde{I}\,',z}}{\tau_n(z)}\,\mathrm{d}\mu\\
&=&\pi(\widetilde{\gamma})\wedge\rho_{M,I'}(\mu)\\
&=&\varphi(\alpha)\wedge\rho_{M,I}(\mu)\\
&\neq&0.
\end{eqnarray*}
We have finished the proof.\end{proof}

\section{The absence of distortion in
$\mathrm{Ham}^1(\mathbb{T}^2,\mu)$
 and $\mathrm{Diff}^1_*(\Sigma_g,\mu)$ with $g>1$.}\label{sec:distortion of
 group}

In 2002, Polterovich \cite{P} showed us a Hamiltonian version of the
Zimmer program (see \cite{Zim}) dealing with actions of lattices. It
is achieved by using the classical action defined in symplectic
geometry, the symplectic filling function (see Section 1.2 in
\cite{P}), and Schwarz's theorem which we have mentioned in the
beginning of this article. In 2003, Franks and Handel \cite{F2}
developed the Thurston theory of normal forms for surface
homeomorphisms with finite fixed sets. In 2006, they \cite{F4} used
the generalized normal form to give a more general version (the map
is a $C^1$-diffeomorphism and the measure is a Borel finite measure)
 of the Zimmer program on the closed oriented
surfaces. We recommend the reader a survey by Fisher \cite{Fish} and
an article by Brown, Fisher and Hurtado \cite{BFH} for the recent
progress of Zimmer program. We will give an alternative proof of the
$C^1$-version of the Zimmer's conjecture on surfaces when the
measure is a Borel finite measure with full support.\smallskip

Suppose that $F$ is a $C^1$-diffeomorphism of $\Sigma_g$ ($g\geq 1$)
which is the time-one map of an identity isotopy
$I=(F_t)_{t\in[0,1]}$ on $\Sigma_g$ and $\widetilde{F}$ is the
time-one map of the lifted identity isotopy
$\widetilde{I}=(\widetilde{F}_t)_{t\in[0,1]}$ on the universal cover
$\widetilde{M}$ of $\Sigma_g$. Recall that, if
$\mathscr{G}\subset\mathrm{Diff}^1_*(\Sigma_g,\mu)$ is a finitely
generated subgroup containing $F$, $\|F\|_{\mathscr{G}}$ is the word
length of $F$ in $\mathscr{G}$. We have the following proposition
whose proof will be provided in Appendix.

\begin{prop}\label{lem:zimmer program}If there exist two distinct fixed points
$\widetilde{a}$ and $\widetilde{b}$ of $\widetilde{F}$, and a point
$z_*\in\mathrm{Rec}^+(F)\setminus\pi(\{\widetilde{a},\widetilde{b}\})$
such that $i(\widetilde{F};\widetilde{a},\widetilde{b},z_*)$ exists
and is not zero, then for any finitely generated subgroup
$F\in\mathscr{G}\subset\mathrm{Diff}^1_*(\Sigma_g,\mu)$ ($g\geq 1$),
$$\|F^n\|_{\mathscr{G}}\succeq \sqrt{n}.$$
\end{prop}

If $F\neq\mathrm{Id}_{\Sigma_g}$ and $\mu$ has full support, by Item
\textbf{A2} in the proof of Theorem \ref{prop:F is not constant if
the contractible fixed points is finite}, we can choose $z_*\in
\mathrm{Rec}^+(F)\setminus X$, such that $\rho_{\Sigma_g,I}(z_*)$
and $i(\widetilde{F};\widetilde{a},\widetilde{b},z_*)$ exist, and
$i(\widetilde{F};\widetilde{a},\widetilde{b},z_*)$ is not zero. By
Proposition \ref{lem:zimmer program}, we can get the following
result which is a generalization of Theorem 1.6 B in \cite{P} on the
closed surfaces.

\begin{cor}\label{cor:sqrt n of growth}
If
$F\in\mathrm{Diff}^1_*(\Sigma_g,\mu)\setminus\{\mathrm{Id}_{\Sigma_g}\}$
($g>1$) or
$F\in\mathrm{Ham}^1(\mathbb{T}^2,\mu)\setminus\{\mathrm{Id}_{\mathbb{T}^2}\}$,
then for any finitely generated subgroup
$F\in\mathscr{G}\subset\mathrm{Diff}^1_*(\Sigma_g,\mu)$ ($g\geq 1$),
$$\|F^n\|_{\mathscr{G}}\succeq \sqrt{n}.$$
\end{cor}

Moreover, we can improve Corollary \ref{cor:sqrt n of growth}. The
following result is our main theorem in this section. \smallskip

\newenvironment{thm02}{\noindent\textbf{Theorem \ref{thm:Fn thicksim
n}}~\itshape}{\par}\smallskip

\begin{thm02}
Let
$F\in\mathrm{Diff}^1_*(\Sigma_g,\mu)\setminus\{\mathrm{Id}_{\Sigma_g}\}$
($g>1$) (resp.
$F\in\mathrm{Ham}^1(\mathbb{T}^2,\mu)\setminus\{\mathrm{Id}_{\mathbb{T}^2}\}$),
and $\mathscr{G}\subset\mathrm{Diff}^1_*(\Sigma_g,\mu)$ ($g>1$)
(resp. $\mathscr{G}\subset\mathrm{Ham}^1(\mathbb{T}^2,\mu)$) be a
finitely generated subgroup containing $F$, then
$$\|F^n\|_{\mathscr{G}}\thicksim n.$$
As a consequence, the groups $\mathrm{Diff}^1_*(\Sigma_g,\mu)$
($g>1$) and $\mathrm{Ham}^1(\mathbb{T}^2,\mu)$ have no
distortions.\end{thm02}\smallskip

Theorem \ref{thm:Fn thicksim n} can be obtained immediately from the
following two lemmas which will be proved in Appendix.

\begin{lem}\label{lem:zimmer1}
If
$F\in\mathrm{Homeo}_*(\Sigma_g,\mu)\setminus\mathrm{Hameo}(\Sigma_g,\mu)$
($g>1$), for any finitely generated subgroup
$F\in\mathscr{G}\subset\mathrm{Homeo}_*(\Sigma_g,\mu)$, we have
$\|F^n\|_{\mathscr{G}}\thicksim n.$
\end{lem}

\begin{lem}\label{lem:zimmer2}
If
$F\in\mathrm{Ham}^1(\Sigma_g,\mu)\setminus\{\mathrm{Id}_{\Sigma_g}\}$
($g\geq1$), for any finitely generated subgroup
$F\in\mathscr{G}\subset\mathrm{Diff}^1_*(\Sigma_g,\mu)$, we have
$\|F^n\|_{\mathscr{G}}\thicksim n.$
\end{lem}

As a consequence of Theorem \ref{thm:Fn thicksim n}, we have the
following theorem:

\begin{thm}\label{thm:zimmer1}
Let $\mathscr{G}$ be a finitely generated group with generators
$\{g_1,\ldots,g_s\}$ and $f\in\mathscr{G}$ be an element which is
distorted with respect to the word norm on $\mathscr{G}$. Then
$\phi(f)=\mathrm{Id}_{\mathbb{T}^2}$ (resp.
$\phi(f)=\mathrm{Id}_{\Sigma_g}$ where $g>1$) for any homomorphism
$\phi: \mathscr{G}\rightarrow \mathrm{Ham}^1(\mathbb{T}^2,\mu)$
(resp. $\phi: \mathscr{G}\rightarrow
\mathrm{Diff}^1_*(\Sigma_g,\mu)$ with $g>1$). In particular, if
$\mathscr{G}$ is a finitely generated subgroup of
$\mathrm{Ham}^1(\mathbb{T}^2,\mu)$ (resp.
$\mathrm{Diff}^1_*(\Sigma_g,\mu)$ with $g>1$), every element of
$\mathscr{G}\setminus\{\mathrm{Id}_{\Sigma_g}\}$ ($g\geq1$) is
undistorted with respect to the word norm on $\mathscr{G}$.
\end{thm}

\begin{proof}We only prove the case where $\phi: \mathscr{G}\rightarrow
\mathrm{Ham}^1(\mathbb{T}^2,\mu)$ since other cases follow
similarly. Let $\mathscr{G}'$ be the finitely generated group
generated by $\{\phi(g_1),\ldots,\phi(g_s)\}$. As $f$ is a
distortion element of $\mathscr{G}$, there exists a subsequence
$\{n_i\}_{i\geq1}\subset\mathbb{N}$ such that
$$\lim_{i\rightarrow+\infty}\frac{\|\phi^{n_i}(f)\|_{\mathscr{G}'}}{n_i}=
\lim_{i\rightarrow+\infty}\frac{\|\phi(f^{n_i})\|_{\mathscr{G}'}}{n_i}=0
.$$ By Theorem \ref{thm:Fn thicksim n}, we have
$\phi(f)=\mathrm{Id}_{\mathbb{T}^2}$.
\end{proof}\smallskip

Let us recall some results about the irreducible lattice
$\mathrm{SL}(n,\mathbb{Z})$ with $n\geq3$. The lattice
$\mathrm{SL}(n,\mathbb{Z})$ and its any normal subgroup of finite
order have the following properties:
\begin{itemize}
 \item It contains a subgroup isomorphic to the group of upper
  triangular integer valued matrices of order $3$ with $1$ on the
  diagonal (the integer Heisenberg group),
  which tells us the existence of distortion element of every infinite normal subgroup
   of $\mathrm{SL}(n,\mathbb{Z})$
   (see \cite[Proposition 1.7]{P2});
   \item It is almost simple (every normal subgroup is finite or has a finite
  index) which is due to Margulis (Margulis finiteness theorem, see
  \cite{Mar}).
\end{itemize}

Applying these results above and Theorem \ref{thm:zimmer1}, we get
the following result:

\smallskip

\newenvironment{thm04}{\noindent\textbf{Theorem \ref{thm:zimmer2}}~\itshape}{\par}\smallskip

\begin{thm04}Every homomorphism from
$\mathrm{SL}(n,\mathbb{Z})$ ($n\geq3$) to
$\mathrm{Ham}^1(\mathbb{T}^2,\mu)$ or
$\mathrm{Diff}^1_*(\Sigma_g,\mu)$ ($g>1$) is trivial. As a
consequence, every  homomorphism from $\mathrm{SL}(n,\mathbb{Z})$
($n\geq3$) to $\mathrm{Diff}^1(\Sigma_g,\mu)$ ($g>1$) has only
finite images.
\end{thm04}

\begin{proof}Again, we only prove the case where $\phi: \mathscr{G}\rightarrow
\mathrm{Ham}^1(\mathbb{T}^2,\mu)$ since other cases follow
similarly. The following argument is due to Polterovich \cite[Proof
of Theorem 1.6]{P2}. By the first item of properties of
$\mathrm{SL}(n,\mathbb{Z})$, there is a distortion element $f$ in
$\mathrm{SL}(n,\mathbb{Z})$. Apply Theorem \ref{thm:zimmer1} to the
distortion element $f$ of infinite order of
$\mathrm{SL}(n,\mathbb{Z})$. We have that $f$ lies in the kernel of
$\phi$. Note that $\mathrm{Ker}(\phi)$ is an infinite normal
subgroup of $\mathrm{SL}(n,\mathbb{Z})$. By the second item of
properties of $\mathrm{SL}(n,\mathbb{Z})$, $\mathrm{Ker}(\phi)$ has
finite index in $\mathrm{SL}(n,\mathbb{Z})$. Hence the quotient
$\mathrm{SL}(n,\mathbb{Z})/\mathrm{Ker}(\phi)$ is finite. Therefore,
$\phi$ has finite images. Applying Corollary \ref{cor:no torsion},
we get $\phi$ is trivial.

Finally, let us recall a classical result about the mapping class
group $\mathrm{Mod}(S)$, where $S$ is a compact, orientable,
connected surface, possibly with boundary, and
$\mathrm{Mod}(S)=\mathrm{Homeo}^+(S)/\mathrm{Homeo}_*(S)$ is the
isotopy classes of orientation preserving homeomorphisms of $S$ (see
\cite{FM}): any homomorphism $\phi: \Gamma\rightarrow
\mathrm{Mod}(S)$ has finite images where $\Gamma$ is an irreducible
lattice in a semisimple lie group of $\mathbb{R}$-rank at least two.

Applying the results above, we get the last statement: every
homomorphism from $\mathrm{SL}(n,\mathbb{Z})$ ($n\geq3$) to
$\mathrm{Diff}^1(\Sigma_g,\mu)$ ($g>1$) has only finite images,
which is a general conjecture of Zimmer in the special case of
surfaces.
\end{proof}\smallskip

\section{Appendix}\label{Appendix}

\subsection{Proofs of Lemma \ref{rem:linking number on connected set},
Lemma \ref{lem:the linking number on a unbounded set} and Lemma \ref{lem:fps is connected and B property}}\quad\\

Endow the surface $M$ with a Riemannian metric and denote by $d$ the
distance induced by the metric. Lift the Riemannian metric to
$\widetilde{M}$ and write $\widetilde{d}$ for the distance induced
by the metric. Let us recall some properties of
$i(\widetilde{F};\widetilde{z},\widetilde{z}\,')$ defined in Formula
\ref{eq:linking number for two fixed points}, whose proofs can be
found in \cite[page 56]{W2}):
\begin{description}
  \item[(P1)] $i(\widetilde{F};\widetilde{z},\widetilde{z}\,')$ is locally constant
on $(\mathrm{Fix}(\widetilde{F})\times
\mathrm{Fix}(\widetilde{F}))\setminus\widetilde{\Delta}$;
  \item[(P2)] $i(\widetilde{F};\widetilde{z},\widetilde{z}\,')$ is invariant by
covering transformation, that is,
$$\qquad i(\widetilde{F};\alpha(\widetilde{z}),\alpha(\widetilde{z}\,'))=
i(\widetilde{F};\widetilde{z},\widetilde{z}\,')\quad\mathrm{for\,\,\,
every}\,\,\,\alpha\in G;$$
  \item[(P3)] $i(\widetilde{F};\widetilde{z},\widetilde{z}\,')=0$
if $\pi(\widetilde{z})=\pi(\widetilde{z}\,')$;
  \item[(P4)] there exists $K$ such that $i(\widetilde{F};\widetilde{z},\widetilde{z}\,')=0$
if $\widetilde{d}(\widetilde{z},\widetilde{z}\,')\geq K$.
\end{description}
\smallskip

Lemma \ref{lem:the linking number on a unbounded set} immediately
follows from Lemma \ref{rem:linking number on connected set} and P4.
Hence we only need to prove Lemma \ref{rem:linking number on
connected set} and Lemma \ref{lem:fps is connected and B property}.

\begin{proof}[Proof of Lemma \ref{rem:linking number on connected set}]
If $\widetilde{z}\not\in\widetilde{X}$, the conclusion holds
obviously by the continuity (see P1 above) and the connectedness of
$\widetilde{X}$. Suppose now that $\widetilde{z}\in\widetilde{X}$.
Fix a point $\widetilde{a}\in \widetilde{X}$. The linking number
$i(\widetilde{F};\widetilde{a},\cdot)$ will be a constant on each
connected component of $\widetilde{X}\setminus{\widetilde{a}}$. Let
$\widetilde{b}$ and $\widetilde{b}'$, $\widetilde{c}$ and
$\widetilde{c}'$ lie on different components, respectively. Then
$i(\widetilde{F};\widetilde{a},\widetilde{b})=i(\widetilde{F};\widetilde{a},\widetilde{b}')$
and
$i(\widetilde{F};\widetilde{a},\widetilde{c})=i(\widetilde{F};a,\widetilde{c}')$.
We have to prove that
$i(\widetilde{F};\widetilde{a},\widetilde{b})=i(\widetilde{F};\widetilde{a},\widetilde{c})$.
We now fix $\widetilde{b}$. Let $\widetilde{Y}$ be the connected
component of $\widetilde{X}\setminus\{\widetilde{a}\}$ that contains
$\widetilde{c}$. Then $\widetilde{a}$ belongs to the closure of
$\widetilde{Y}$ and hence $\widetilde{Y}\cup\widetilde{a}$ is
connected. Let $\widetilde{Z}$ be the connected component of
$\widetilde{X}\setminus\{\widetilde{b}\}$ that contains
$\widetilde{a}$. So $(\widetilde{Y}\cup\widetilde{a})\cap
\widetilde{Z}\neq \emptyset$. Hence
$\widetilde{c}\in\widetilde{Y}\subset\widetilde{Z}$. We get
$i(\widetilde{F};\widetilde{b},\widetilde{a})=i(\widetilde{F};\widetilde{b},\widetilde{c})$
since $\widetilde{a}$ and $\widetilde{c}$ lie on the same connected
component of $X\setminus{\widetilde{b}}$. Now fix $\widetilde{c}$.
Similarly, we have
$i(\widetilde{F};\widetilde{c},\widetilde{b})=i(\widetilde{F};\widetilde{c},\widetilde{a})$
since $\widetilde{b}$ and $\widetilde{a}$ lie on the same connected
component of $X\setminus{\widetilde{c}}$. Obviously,
$i(\widetilde{F};\widetilde{z},\widetilde{z}\,')$ is symmetrical on
$(\mathrm{Fix}(\widetilde{F})\times
\mathrm{Fix}(\widetilde{F}))\setminus\widetilde{\Delta}$ by the
definition of $i_{\widetilde{I}} (\widetilde{z},\widetilde{z}\,')$.
Therefore, we obtain
$$i(\widetilde{F};\widetilde{a},\widetilde{b})=i(\widetilde{F};\widetilde{b},\widetilde{a})=
i(\widetilde{F};\widetilde{b},\widetilde{c})=i(\widetilde{F};\widetilde{c},\widetilde{b})=
i(\widetilde{F};\widetilde{c},\widetilde{a})=i(\widetilde{F};\widetilde{a},\widetilde{c}).$$
\end{proof}

We now show that Lemma \ref{lem:fps is connected and B property}
holds. Prior to that, we establish the following lemma:

\begin{lem}\label{rem:wb property and connected set}
If $\widetilde{X}$ is a connected subset of
$\mathrm{Fix}(\widetilde{F})$ and $\widetilde{X}$ is not reduced to
a singleton, $I$
satisfies the B-property on $\widetilde{X}$. 
\end{lem}

\begin{proof}
Let $\widetilde{X}'$ be a connected component of
$\mathrm{Fix}(\widetilde{F})$ that contains $\widetilde{X}$. By
Lemma \ref{lem:the linking number on a unbounded set}, the
conclusion is obvious if $\widetilde{X}'$ is unbounded. Suppose now
that $\widetilde{X}'$ is bounded. Then $\widetilde{X}'$ is compact.
Let us consider the value
$i(\widetilde{F},\widetilde{z},\widetilde{z}')$ where
$\widetilde{z}\in \mathrm{Fix}(\widetilde{F})$ and
$\widetilde{z}'\in \widetilde{X}'$. By the second statement of Lemma
\ref{rem:linking number on connected set}, we only need to consider
the case where $\widetilde{z}\in
\mathrm{Fix}(\widetilde{F})\setminus\widetilde{X}'$. If there exists
a sequence $\{\widetilde{z}_n\}_{n=1}^{+\infty}\subset
\mathrm{Fix}(\widetilde{F})\setminus\widetilde{X}'$ such that
$|i(\widetilde{F},\widetilde{z}_n,\widetilde{z}')|\rightarrow+\infty$
as $n\rightarrow+\infty$, according to P4, the sequence
$\{\widetilde{z}_n\}$ must have a convergence subsequence. Without
loss of generality, we suppose that
$\lim_{n\rightarrow+\infty}\widetilde{z}_n=\widetilde{z}_0$.
Obviously, $\widetilde{z}_0\not\in \widetilde{X}'$ by the second
statement of Lemma \ref{rem:linking number on connected set}, which
is also impossible in this case since
$\widetilde{d}(\widetilde{z}_0,\widetilde{X}')>0$ and due to the
first statement of Lemma \ref{rem:linking number on connected set}.
\end{proof}

\begin{proof}[Proof of Lemma \ref{lem:fps is connected and B property}]
If $X$ is not contractible, the first property of Lemma \ref{lem:fps
is connected and B property} follows from Lemma \ref{lem:the linking
number on a unbounded set}. Otherwise, it follows from Lemma
\ref{rem:wb property and connected set} and the properties P2--P4 of
$i(\widetilde{F};\widetilde{z},\widetilde{z}\,')$.

Furthermore, we assume that the number of connected components of
$\mathrm{Fix}_{\mathrm{Cont},I}(F)$ is finite. If some connected
components of $\mathrm{Fix}_{\mathrm{Cont},I}(F)$ are singletons,
the second property follows from the properties P2--P4 and the first
property of this lemma has been proved above.
\end{proof}

\subsection{Proof of Lemma \ref{lem:topology lemma of connected set}}\quad\\

To prove Lemma \ref{lem:topology lemma of connected set}, we need
the following lemma:
\begin{lem}\label{lem:homology and homotopy}
Let $S$ and $S'$ be two open connected subsurfaces of an orientable
closed surface $M$, with $S'\subset \mathrm{Int}(S)$ and $z\in S'$.
If $S$ and $S'$ satisfy the following properties:
\begin{itemize}
     \item $i_*(H_1(S,\mathbb{Z}))$ and $i_*(H_1(S',\mathbb{Z}))$ have
the same image in $H_1(M,\mathbb{Z})$;
     \item the number of connected components  with positive genus of $\mathrm{Cl}(M\setminus S)$
     equals to that (i.e., the number of connected components with positive genus) of $\mathrm{Cl}(M\setminus S')$,
     \end{itemize}
then $i_*(\pi_1(S, z))$ and $i_*(\pi_1(S',z))$ have the same image
in $\pi_1(M,z)$.\end{lem}

\begin{proof} Let $C$ be a connected component of $\partial S'$ which belongs to
$\partial(\mathrm{Cl}(S\setminus S'))$, more precisely, to the
boundary of a connected component $S''$ of $\mathrm{Cl}(S\setminus
S')$. The genus of $S''$ is zero because $i_*(H_1(S,\mathbb{Z}))$
and $i_*(H_1(S',\mathbb{Z}))$ have the same image in
$H_1(M,\mathbb{Z})$. We claim that $C$ is the unique common
component of $\partial S'$ and $\partial S''$. Otherwise, one can
find a cycle $r$ in $S$ that has a nonzero algebraic intersection
number with $C$, which contradicts with the fact that
$i_*(H_1(S,\mathbb{Z}))$ and $i_*(H_1(S',\mathbb{Z}))$ have the same
image in $H_1(M,\mathbb{Z})$ (see the left, Figure
\ref{fig:topologic lemma}). Secondly, we note that
$i_*(H_1(\mathrm{Cl}(M\setminus S),\mathbb{Z}))$ and
$i_*(H_1(\mathrm{Cl}(M\setminus S'),\mathbb{Z}))$ have the same
image in $H_1(M,\mathbb{Z})$ by the hypotheses. It implies that
$S''$ (in fact, every connected component of $\mathrm{Cl}(S\setminus
S')$) is a subsurface without genus (see the right part of Figure
\ref{fig:topologic lemma} for a counter-example).

\begin{figure}[ht]
\begin{center}\scalebox{0.18}[0.18]{\includegraphics{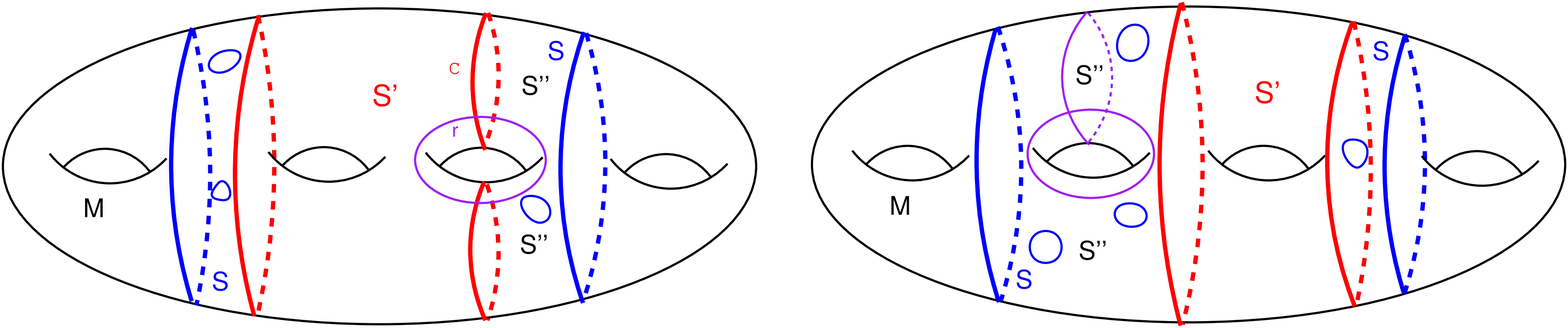}}\end{center}
\caption{
\label{fig:topologic lemma}}
\end{figure}

Furthermore, by the hypotheses, we can deduce that $S''$ (and hence
every connected component of $\mathrm{Cl}(S\setminus S')$) satisfies
the following property: the number of the boundary circles of $S''$
that bound a nonzero genus subsurface of $M\setminus S$ is at most
$1$ (see Figure \ref{fig:topologic lemma 2}, the left situation can
happen but the right one can not). All the other boundary circles of
$S''$ not bounding a nonzero genus subsurface of $M\setminus S$ will
bound disks, say $D_i\, (1\leq i\leq m)$. And each $D_i$ is
contained in $M\setminus S$. It implies that every path in $S''$
whose endpoints are on $C$ is homotopic in $S''\cup (\cup_{i=1}^m
D_i)$ to a path on $C$. Thus, the conclusion follows.

\begin{figure}[ht]
\begin{center}\scalebox{0.3}[0.3]{\includegraphics{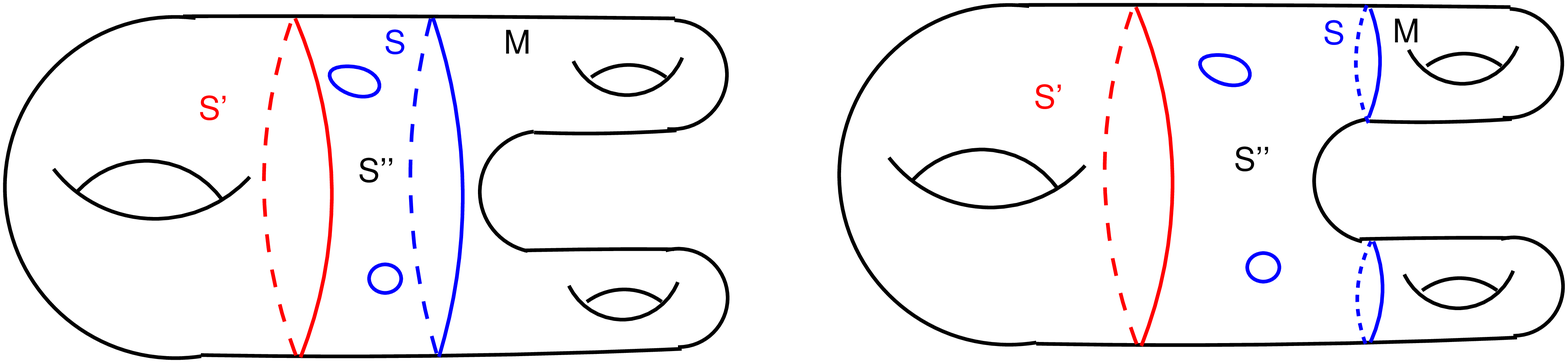}}\end{center}
\caption{
\label{fig:topologic lemma 2}}
\end{figure}
\end{proof}

\begin{proof}[Proof of Lemma \ref{lem:topology lemma of connected set}]
For a small $\epsilon>0$, the inclusion $i:
Z_{\epsilon}\hookrightarrow M$ naturally induces a homomorphism:
$i_*: H_1(Z_\epsilon,\mathbb{Z})\rightarrow H_1(M,\mathbb{Z})$.
Observing that the sequences of subspaces
$\{i_*(H_1(Z_{\epsilon},\mathbb{Z}))\}_{\epsilon>0}$
 and $\{i_*(H_1(M\setminus
Z_{\epsilon},\mathbb{Z}))\}_{\epsilon>0}$ are respectively
non-decreasing and non-increasing in a finite dimensional space
(since $H_1(M,\mathbb{Z})$ is finitely generated), we get that they
must stabilize. Therefore, we can choose a $\epsilon_0>0$ small
enough such that
$$i_*(H_1(Z_{\epsilon},\mathbb{Z}))=
i_*(H_1(Z_{\epsilon_0},\mathbb{Z}))\quad \mathrm{and}\quad
i_*(H_1(\mathrm{Cl}(M\setminus Z_{\epsilon}),\mathbb{Z}))=
i_*(H_1(\mathrm{Cl}(M\setminus Z_{\epsilon_0}),\mathbb{Z}))$$ for
all $0<\epsilon<\epsilon_0$. Furthermore, we can make $\epsilon_0$
even smaller such that for all $0<\epsilon<\epsilon_0$, the
subsurfaces $Z_{\epsilon_0}$ and $Z_{\epsilon}$ satisfy the
hypotheses of Lemma \ref{lem:homology and homotopy} since the genus
of $M$ is finite. The conclusion then follows from Lemma
\ref{lem:homology and homotopy}.
\end{proof}

\subsection{Proofs of Proposition \ref{lem:zimmer program}, Lemma \ref{lem:zimmer1} and Lemma \ref{lem:zimmer2}}

\begin{proof}[Proof of Proposition \ref{lem:zimmer program}]
If
$z_*\in\mathrm{Rec}^+(F)\setminus\pi(\{\widetilde{a},\widetilde{b}\})$
and $i(\widetilde{F};\widetilde{a},\widetilde{b},z_*)$ exists, we
have
$z_*\in\mathrm{Rec}^+(F^n)\setminus\pi(\{\widetilde{a},\widetilde{b}\})$
(see Lemma 19, \cite{W}), and by Proposition 4.5.3 in \cite{W2}, we
have that
$i(\widetilde{F}^n;\widetilde{a},\widetilde{b},z_*)=ni(\widetilde{F};\widetilde{a},\widetilde{b},z_*)$
for all $n\geq1$.\smallskip

Assume that $F\in\mathscr{G}=\langle
F_{1,1},\cdots,F_{s,1}\rangle\subset\mathrm{Diff}^1_*(\Sigma_g,\mu)$
and write $N(n)=\|F^n\|_{\mathscr{G}}$. Then there exist identity
isotopies $I_i=(F_{i,t})_{t\in[0,1]}\subset
\mathrm{Diff^1(\Sigma_g)}$ ($1\leq i\leq s$) such that, for every
$n\geq1$, the isotopy $I^n$ is homotopic to the isotopy
$I^{(n)}:=\left(F_t^{(n)}\right)_{0\leq t\leq
1}=\prod\limits_{j=1}^{N(n)}I^{~\epsilon_j}_{i_j}$, where
$i_j\in\{1,2,\cdots,s\}$, $\epsilon_j\in\{-1,1\}$
($j=1,2,\cdots,N(n)$) and
$$F_t^{(n)}(z)=F^{\epsilon_{i_k}}_{i_k,N(n)t-(k-1)}(F^{\epsilon_{i_{k-1}}}_{i_{k-1},1},\circ
\cdots\circ F^{\epsilon_{i_1}}_{i_1,1}(z)),\quad \mathrm{if} \quad
\frac{k-1}{N(n)}\leq t\leq \frac{k}{N(n)}.$$

Let $\widetilde{I}_i=(\widetilde{F}_{i,t})_{t\in[0,1]}$ ($1\leq
i\leq s$) and $\widetilde{I}^{(n)}=(\widetilde{F}_t^{(n)})_{0\leq
t\leq 1}$ be the lifts of $I_i$ ($1\leq i\leq s$) and
$(F_t^{(n)})_{0\leq t\leq 1}$ to $\widetilde{M}$ respectively.
Identify the sphere $\widetilde{M}\cup\{\infty\}$ as the Riemann
sphere $\mathbb{C}\cup\{\infty\}$. Again, for simplicity,
we can suppose that $\widetilde{a}=0$ and
$\widetilde{b}=1$.\smallskip

Fix $n\geq1$. Using the method of Lemma \ref{rem:identity isotopies
fix three points on sphere} (see \cite{W2}, page 57), we can get the
isotopy $\widetilde{I}'^{(n)}=(\widetilde{F}\,'^{(n)}_t)_{0\leq
t\leq 1}$ which fixes $0$, $1$ and is an isotopy on $\widetilde{M}$
from $\mathrm{Id}_{\widetilde{M}}$ to $\widetilde{F}^n$, where

\begin{equation}\label{eq:new isotopy}
\widetilde{F}\,'^{(n)}_t(\widetilde{z})=\frac{\widetilde{F}_t^{(n)}(\widetilde{z})
-\widetilde{F}_t^{(n)}(0)}{\widetilde{F}_t^{(n)}(1)-\widetilde{F}_t^{(n)}(0)}\,.
\end{equation}

Let $\widetilde{\gamma}=\{0\leq r\leq 1\}$ be the straight line from
$0$ to $1$. If
$\widetilde{I}'^{(n)}(\widetilde{z})\cap\widetilde{\gamma}\neq\emptyset$
for some point $\widetilde{z}\in \widetilde{M}\setminus\{0,1\}$,
then there exist $t_0\in[0,1]$ and $r_0\in]0,1[$ such that
$\widetilde{F}\,'^{(n)}_{t_0}(\widetilde{z})=r_0$, that is
\begin{equation}\label{eq:zimmer}
    \widetilde{F}_{t_0}^{(n)}(\widetilde{z}) -\widetilde{F}_{t_0}^{(n)}(0)=r_0
(\widetilde{F}_{t_0}^{(n)}(1) -\widetilde{F}_{t_0}^{(n)}(0)).
\end{equation}

Let
$$C=\max_{i\in\{1,\cdots,s\}}\sup_{t\in[0,1]\,,\,\widetilde{z}\in\widetilde{M}}
\widetilde{d}(\widetilde{F}_{i,t}(\widetilde{z}),\widetilde{z}).$$
We have
\begin{equation*}\label{ineq:Ak}
\left|\widetilde{F}_t^{(n)}(1)-\widetilde{F}_t^{(n)}(0)\right|\leq
2CN(n)+1
\end{equation*}
and
$$\left|\widetilde{F}_t^{(n)}(\widetilde{z})-\widetilde{F}_t^{(n)}(0)\right|\geq
|\widetilde{z}|-2CN(n)$$ for all $t\in[0,1]$. Hence when
$|\widetilde{z}|\geq 5CN(n)+1$, we get
$\left|\widetilde{F}\,'^{(n)}_t(\widetilde{z})\right|>1$, i.e.,
$\widetilde{I}'^{(n)}(\widetilde{z})\cap\widetilde{\gamma}=\emptyset$.
Recall that the open disks $\widetilde{V}$ and $\widetilde{W}$ that
contain $\infty$ in Section 4.6.1 of \cite{W2}. Here, we set
$\widetilde{V}=\{\widetilde{z}\in\widetilde{M}\mid|\widetilde{z}|>
5CN(n)+1\}$ and choose an open disk $\widetilde{W}$ containing
$\infty$ such that $\widetilde{\gamma}\cap \widetilde{W}=\emptyset$,
and for every $\widetilde{z}\in \widetilde{V}$, we have
$\widetilde{I}'^{(n)}(\widetilde{z})\subset \widetilde{W}$. Without
loss of generality, we can suppose that $z_*\notin
\pi(\widetilde{\gamma})$. Choose an open disk $U$ containing $z_*$
such that $U\cap \pi(\widetilde{\gamma})=\emptyset$. Write
respectively $\tau(n,z)$ and $\Phi_n(z)$ for the first return time
and the first return map of $F^n$ throughout this proof. For every
$m\geq1$, recall that
$\tau_m(n,z)=\sum\limits_{i=0}^{m-1}\tau(n,\Phi^i_n(z))$. Let us
consider the following value
$$L_m(\widetilde{F}^n;0,1,z_*)=\widetilde{\gamma}\wedge\widetilde{\Gamma}^m_{\widetilde{I}'^{(n)},z_*}.$$

By the same arguments with Lemma 4.6.4 and Lemma 4.6.6 in \cite{W2},
we can find multi-paths $\widetilde{\Gamma}^m_i(z_*)$ ($1\leq i \leq
P_m(z_*)$) from $\widetilde{V}$ to $\widetilde{V}$  such that
$$L_m(\widetilde{F}^n;0,1,z_*)=\widetilde{\gamma}\wedge\prod_{1\leq i\leq
P_m(z_*)}\widetilde{\Gamma}^m_i(z_*).$$
\smallskip

For every $j\in\{1,\cdots,s\}$ and
$(\widetilde{z},\widetilde{z}\,')\in\widetilde{M}\times\widetilde{M}\setminus\widetilde{\Delta}$,
there is a unique function $\theta_j: [0,1]\rightarrow \mathbb{R}$
such that $\theta_j(0)=0$ and
$$e^{2i\pi\theta_j(t)}=\frac{\widetilde{F}_{j,t}(\widetilde{z})-\widetilde{F}_{j,t}(\widetilde{z}\,')}
{\left|\widetilde{F}_{j,t}(\widetilde{z})-\widetilde{F}_{j,t}(\widetilde{z}\,')\right|}.$$
Let $\lambda_j(\widetilde{z},\widetilde{z}\,')=\theta_j(1)$. As
$\widetilde{I}_j\subset \mathrm{Diff}^1(\widetilde{M})$, there is a
natural compactification of
$\widetilde{M}\times\widetilde{M}\setminus\widetilde{\Delta}$
obtained by replacing the diagonal $\widetilde{\Delta}$ with the
unit tangent bundle such that the map $\lambda_j$ extends
continuously (see, for example, \cite[page 81]{CFGL}). Let

\begin{equation}\label{eq:C_1}
    C_1=\max_{i\in\{1,\cdots,s\}}\sup_{(\widetilde{z},\widetilde{z}\,')
\in\widetilde{M}\times\widetilde{M}\setminus\widetilde{\Delta}}
|\lambda_j(\widetilde{z},\widetilde{z}\,')|.
\end{equation}

Suppose that $\widetilde{M}_0\subset\widetilde{M}$ is a closed
fundamental domain with regard to the transformation group $G$.
Denote by $\widetilde{I}^{\pm}_i(\widetilde{M}_0)$ the set
$\{\widetilde{F}_{i,\pm t}(\widetilde{z})\mid (\pm
t,\widetilde{z})\in [0,1]\times\widetilde{M}_0\}$ where
$\widetilde{F}_{i,-t}=\widetilde{F}_{i,1-t}\circ\widetilde{F}^{-1}_{i,1}$.
As $\widetilde{F}_{i,\pm
t}\circ\alpha=\alpha\circ\widetilde{F}_{i,\pm t}$ for all $\alpha\in
G$ and $t\in[0,1]$, we can suppose that
$$C_2=\max_{i\in\{1,\cdots,s\},\,z\in M}\sharp\{\widetilde{z}\in
\pi^{-1}(z)\mid \widetilde{I}_i(\widetilde{z})\cap
\widetilde{I}^{+}_i(\widetilde{M}_0)\cup
\widetilde{I}^{-1}_i(\widetilde{z})\cap
\widetilde{I}^{-}_i(\widetilde{M}_0)\neq\emptyset\},$$ which is
independent of $n$. \bigskip

For every $0\leq j\leq m-1$, $\tau_j(n,z_*)\leq l<\tau_{j+1}(n,z_*)$
and $1\leq k\leq N(n)$, let
$F^{\epsilon_{i_0}}_{i_{0},1}=\mathrm{Id}_{\Sigma_g}$,
$\widetilde{F}^{~\epsilon_{i_{0}}}_{i_{0},1}=\mathrm{Id}_{\widetilde{M}}$,

$$z_{j,l,k}=F^{~\epsilon_{i_{k-1}}}_{i_{k-1},1}(F^{~\epsilon_{i_{k-2}}}_{i_{k-2},1}\circ\cdots\circ
F^{~\epsilon_{i_{1}}}_{i_1,1}(F^{n(l-\tau_j(n,z_*))}(\Phi_n^j(z_*))))$$
and
$$\widetilde{z}^{\,0}_{k}=\widetilde{F}^{~\epsilon_{i_{k-1}}}_{i_{k-1},1}(\widetilde{F}^{~\epsilon_{i_{k-2}}}_{i_{k-2},1}\circ\cdots\circ \widetilde{F}^{~\epsilon_{i_{1}}}_{i_1,1}(0))
,\quad\widetilde{z}^1_{k}=\widetilde{F}^{~\epsilon_{i_{k-1}}}_{i_{k-1},1}(\widetilde{F}^{~\epsilon_{i_{k-2}}}_{i_{k-2},1}\circ\cdots\circ
\widetilde{F}^{~\epsilon_{i_{1}}}_{i_1,1}(1)).$$

When $\frac{k-1}{N(n)}\leq t\leq \frac{k}{N(n)}$, recall that
\begin{equation*}\widetilde{F}\,'^{(n)}_t(\widetilde{z})=\frac{\widetilde{F}^
{~\epsilon_{i_{k}}}_{i_k,N(n)t-(k-1)}(\widetilde{z})-\widetilde{F}^{~\epsilon_
{i_{k}}}_{i_k,N(n)t-(k-1)}(\widetilde{z}^{\,0}_{k})}
{\widetilde{F}^{~\epsilon_{i_{k}}}_{i_k,N(n)t-(k-1)}(\widetilde{z}^{\,1}_{k})-
\widetilde{F}^{~\epsilon_{i_{k}}}_{i_k,N(n)t-(k-1)}(\widetilde{z}^{\,0}_{k})}.\end{equation*}

For every $\widetilde{z}\in\widetilde{M}$, denote by
$\widetilde{J}_k(\widetilde{z})$ the curve
$$\widetilde{J}_k(\widetilde{z})=\left(\widetilde{F}\,'^{(n)}_t(\widetilde{z})\right)_{\frac{k-1}{N(n)}\leq
t\leq \frac{k}{N(n)}}.$$

For every $k$, define the immersed square
\begin{eqnarray*}
 A_k: [0,1]^2 &\rightarrow& \widetilde{M}\nonumber\\
  (t,r)&\mapsto& \widetilde{F}^{~\epsilon_{i_{k}}}_{i_k,t}(~\widetilde{z}_k^{\,0})+r(\widetilde{F}^{~\epsilon_{i_{k}}}_{i_k,t}(\widetilde{z}_k^1)-\widetilde{F}^{~\epsilon_{i_{k}}}_{i_k,t}
(\widetilde{z}_k^{\,0})).
\end{eqnarray*}
\begin{figure}[ht]
\begin{center}\scalebox{0.5}[0.5]{\includegraphics{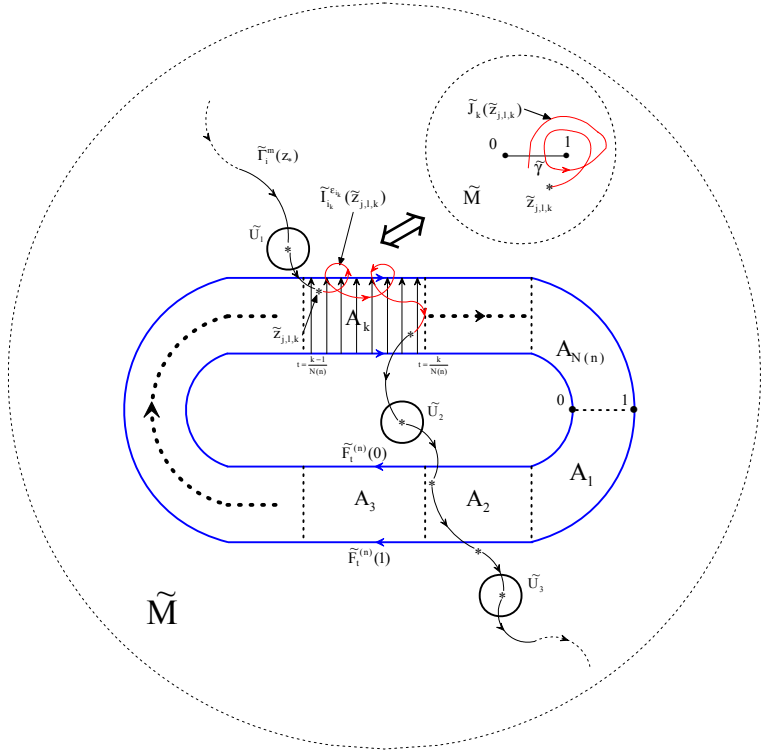}}\end{center}
\caption{The proof of Proposition \ref{lem:zimmer
program}\label{fig:Zimmer-program-1}}
\end{figure}

By Equality \ref{eq:zimmer}, we know that
$\widetilde{J}_k(\widetilde{z})\cap\widetilde{\gamma}\neq\emptyset$
implies $\widetilde{I}_{i_k}^{~\epsilon_{i_{k}}}(\widetilde{z})\cap
A_k\neq\emptyset$ (see Figure \ref{fig:Zimmer-program-1}
\footnote{In Figure \ref{fig:Zimmer-program-1}, there are two
universal covers that the curves in $\widetilde{M}$ (the big one) is
generated by the isotopy $\widetilde{I}^{(n)}$,
 and that the curve in $\widetilde{M}$ (the small one) is generated by $\widetilde{J}_k$ (and hence by the
isotopy $\widetilde{I}'^{(n)}$ defined by Formula \ref{eq:new
isotopy}).}). Let
$$C_{j,l,k}=\{\widetilde{z}_{j,l,k}\in\pi^{-1}(z_{j,l,k})\mid
\widetilde{I}_{i_k}^{~\epsilon_{i_{k}}}(\widetilde{z}_{j,l,k}) \cap
A_k\neq\emptyset\}.$$\smallskip


For each $k$, since
$|\widetilde{z}^{\,0}_{k}-\widetilde{z}^{\,1}_{k}|\leq 2C(k-1)+1$,
there exists $C_3>0$ (depending only on $\widetilde{a}$ and
$\widetilde{b}$) such that
\begin{equation}\label{ineq:Cjlk1}
\sharp\{\alpha\in G\mid A_k\cap
\alpha(\widetilde{M}_0)\neq\emptyset\}\leq C_3N(n).
\end{equation}
Therefore,
\begin{equation}\label{ineq:Cjlk2}
\sum_{j,l,k}\sharp C_{j,l,k}\leq C_2C_3N^2(n)\tau_m(n,z_*).
\end{equation}


We obtain
\begin{equation}\label{ineq:Cjlk3}
L_m(\widetilde{F}^n;0,1,z_*)=\widetilde{\gamma}\wedge\prod_{1\leq
i\leq
P_m(z_*)}\widetilde{\Gamma}^m_i(z_*)=\widetilde{\gamma}\wedge\prod_{j,l,k}\left(\prod_{\widetilde{z}
\in C_{j,l,k}}\widetilde{J}_{k}(\widetilde{z})\right).\end{equation}
We get
\begin{equation}\label{eq: key eq}
\left|L_m(\widetilde{F}^n;0,1,z_*)\right|\leq
c_0N^2(n)\tau_m(n,z_*),
\end{equation}
where $c_0=C_1C_2C_3$. Therefore,
\begin{equation*}
\left|ni(\widetilde{F};0,1,z_*)\right|=\left|i(\widetilde{F}^n;0,1,z_*)\right|\leq
c_0N^2(n).
\end{equation*}
This implies that, for every $n\geq1$,
\begin{equation*}\label{ineq:subline}
0<\left|i(\widetilde{F};0,1,z_*)\right|\leq
c_0\frac{N^2(n)}{n}.\end{equation*} That is
$\|F^n\|_{\mathscr{G}}\succeq \sqrt{n}.$
\end{proof}

\begin{proof}[Proof of Lemma \ref{lem:zimmer1}]
By the definition of $\mathrm{Hameo}(\Sigma_g,\mu)$, we know that
$\rho_{\Sigma_g,I}(\mu)\neq0$. Assume that $F\in\mathscr{G}=\langle
F_{1,1},\cdots,F_{s,1}\rangle\subset\mathrm{Homeo}_*(\Sigma_g,\mu)$
and $I_i$ ($1\leq i\leq s$) are the identity isotopies corresponding
to $F_{i,1}$. Recall that $\|\cdot\|_{H_1(\Sigma_g,\mathbb{R})}$ is
the norm on the space $H_1(\Sigma_g,\mathbb{R})$. Write
$$\kappa=\max_{i\in\{1,\cdots,s\}}
\left\{\left\|\rho_{\Sigma_g,I_i}(\mu)\right\|_{H_1(\Sigma_g,\mathbb{R})}\right\}.$$
As $\rho_{\Sigma_g,I}(\mu)\neq0$ and $F\in\mathscr{G}$, we have
$\kappa>0$.

For every $n\in\mathbb{N}$, if $I^n$ is homotopic to
$\prod\limits_{s=1}^{N(n)}I^{\;\epsilon_{i_{s}}}_{i_s}$, then we
have
$$n\cdot\|\rho_{\Sigma_g,I}(\mu)\|_{H_1(\Sigma_g,\mathbb{R})}=
\|\rho_{\Sigma_g,I^n}(\mu)\|_{H_1(\Sigma_g,\mathbb{R})}=
\left\|\sum_{s=1}^{N(n)}\rho_{\Sigma_g,I^{\;\epsilon_{i_{s}}}_{i_s}}(\mu)\right\|_{H_1(\Sigma_g,\mathbb{R})}\leq
\kappa\cdot N(n).$$ Hence $\|F^n\|_{\mathscr{G}}\succeq n$. On the
other hand, we have $\|F^n\|_{\mathscr{G}}\preceq n$, which
completes the proof.
\end{proof}\smallskip

\begin{proof}[Proof of Lemma \ref{lem:zimmer2}]
It is sufficient to prove that $\|F^n\|_{\mathscr{G}}\succeq n$. We
use the notations in the proofs of Theorem \ref{prop:F is not
constant if the contractible fixed points is finite} and Lemma
\ref{lem:zimmer program}.

If $\sharp\mathrm{Fix}_{\mathrm{Cont},I}(F)=+\infty$, assume that
$X\subset\mathrm{Fix}_{\mathrm{Cont},I}(F)$, $I'$, $Y=\{a,b\}\subset
X$,$I_Y'$ are the notations defined in the proof for the case
$\sharp\mathrm{Fix}_{\mathrm{Cont},I}(F)=+\infty$ of Theorem
\ref{prop:F is not constant if the contractible fixed points is
finite}. If $\sharp\mathrm{Fix}_{\mathrm{Cont},I}(F)<+\infty$, for
convenience, we require $a=\alpha(\lambda)$, $b=\omega(\lambda)$,
and $I_Y'=I'$ where $\alpha(\lambda)$, $\omega(\lambda)$ and $I'$
are the notions defined in the proof for the case
$\sharp\mathrm{Fix}_{\mathrm{Cont},I}(F)<+\infty$ of Theorem
\ref{prop:F is not constant if the contractible fixed points is
finite}. Suppose that $\widetilde{I}'$, $\widetilde{I}_Y'$ are
respectively the lifts of $I'$ and $I_Y'$ to $\widetilde{M}$. Choose
a lift $\widetilde{a}$ of $a$ and a lift $\widetilde{b}$ of $b$. We
know that $I_{\mu}(\widetilde{F};a,b)\neq0$. As
$F\neq\mathrm{Id}_{\Sigma_g}$ and $\mu$ has full support, by Item
\textbf{A2} in the proof of Theorem \ref{prop:F is not constant if
the contractible fixed points is finite}, we can choose $z_*\in
\mathrm{Rec}^+(F)\setminus X$, such that $\rho_{\Sigma_g,I}(z_*)$
and $i(\widetilde{F};\widetilde{a},\widetilde{b},z_*)$ exist, and
$i(\widetilde{F};\widetilde{a},\widetilde{b},z_*)$ is not zero.

Suppose now that $z\in \Sigma_g\setminus X$. By the items 3 and 5 of
Theorem \ref{thm:O.Jaulent}, we know that $I(z)$ and $I_Y'(z)$ are
homotopic in $\Sigma_g$. Hence, for every $n\in\mathbb{N}$,
$I^{(n)}(z)=\prod_{j=1}^{N(n)}I^{~\epsilon_{i_j}}_{i_j}(z)$ is
homotopic to $(I_{Y}')^{n}(z)$ in $\Sigma_g$.

Recall that $\widetilde{I}^{(n)}$ is the lift of $I^{(n)}$ to
$\widetilde{M}$ and $\widetilde{M}_0$ is a closed fundamental domain
with regard to $G$. Obviously, there exists $C'>0$ independent of
$z$ such that
\begin{equation*}
\sharp\{\alpha\in G\mid \widetilde{I}^{(n)}(z)\cap
\alpha(\widetilde{M}_0)\neq\emptyset\}\leq C'N(n).
\end{equation*}

Observe that $N(k)-N(1)\leq N(k+1)\leq N(k)+N(1)$ for every $1\leq
k< n$.
This implies that there exists $C'_2>0$, independent of $z$, such
that
\begin{equation}\label{homotopy finiteness}
\sharp\{\alpha\in G\mid (\widetilde{I}_{Y}')^{n}(z)\cap
\alpha(\widetilde{M}_0)\neq\emptyset\}\leq C'_2N(n).
\end{equation}

Assume that $C_{j,l,k}$, $\widetilde{\gamma}$,
$\widetilde{I}'^{(n)}$, $C_1$ and $C_2$ are the notations in the
proof of Lemma~\ref{lem:zimmer program}.

Observing that $(\widetilde{I}_{Y}')^{n}$ and $\widetilde{I}'^{(n)}$
are two isotopies from $\mathrm{Id}_{\widetilde{M}}$ to
$\widetilde{F}^n$ which fix $\widetilde{a}$ and $\widetilde{b}$, by
Remark \ref{rem:some result of of sphere delete three points},
$(\widetilde{I}_{Y}')^{n}$ is homotopic to $\widetilde{I}'^{(n)}$ in
$\widetilde{M}\setminus\{\widetilde{a},\widetilde{b}\}$.

Recall that if $z\in X\setminus \{a,b\}$, then $\gamma\wedge
I_Y'(z)=0$ (see the proof of Theorem \ref{prop:F is not constant if
the contractible fixed points is finite}). For every $z\in
\Sigma_g\setminus X$, if there is at least one of the ends of
$\widetilde{I}'^{(n)}(\widetilde{z})$ on $\widetilde{\gamma}$, we
construct a new path by extending
$\widetilde{I}'^{(n)}(\widetilde{z})$ a little bit such that the
ends of the new path are not on $\widetilde{\gamma}$ and denote the
new one still by $\widetilde{I}'^{(n)}(\widetilde{z})$. Consider the
value
$V(\widetilde{z})=|\widetilde{\gamma}\wedge\widetilde{I}'^{(n)}(\widetilde{z})|+2$.
For every $z\in \Sigma_g\setminus\{a,b\}$, by Inequality
\ref{homotopy finiteness}, we have

\begin{equation}\label{ineq:N(n)}
\sharp\{\widetilde{z}\in\pi^{-1}(z)\mid V(\widetilde{z})\neq0\}\leq
C'_2N(n).
\end{equation}\smallskip

Write
\begin{eqnarray*}
C\,'_{j,l}&=&\{\widetilde{z}\in\pi^{-1}(z_{j,l,N(n)})\mid
V(\widetilde{z})>2\},\\
C'_{j,l,k}&=&\{\widetilde{z}_{j,l,k}\in\pi^{-1}(z_{j,l,k})\mid\widetilde{I}_{i_k}^{\,\epsilon_{i_k}}(\widetilde{z}_{j,l,k})
\cap(A_{k}(\{r=0\}\cup \{r=1\})\neq\emptyset\}.\end{eqnarray*}
Obviously,
$$\sharp C'_{j,l,k}\leq 2 C_2,\quad \sum_{j,l,k}\sharp
C'_{j,l,k}\leq 2C_2N(n)\tau_m(n,z_*).$$

Under the hypotheses of Lemma \ref{lem:zimmer2}, we want to improve
the value $N^2(n)$ to $N(n)$ in Inequality \ref{eq: key eq}.

Based on the analyses above, we get
$$L_m(\widetilde{F}^n;\widetilde{a},\widetilde{b},z_*)=\widetilde{\gamma}\wedge
\widetilde{\Gamma}^m_{\widetilde{I}'^{(n)},z_*}=
\widetilde{\gamma}\wedge\widetilde{\Gamma}^m_{\left(\widetilde{I}\,'_Y\right)^{\,n},z_*}
=\pi(\widetilde{\gamma})\wedge\Gamma^m_{\left(I\,'_Y\right)^{\,n},z_*}.$$

To estimate the value
$L_m(\widetilde{F}^n;\widetilde{a},\widetilde{b},z_*)$, we need to
consider the isotopy $\widetilde{I}'^{(n)}$. 
If the immersed squares $A_k$ $(k=1,\cdots, N(n))$ are uniformly
bounded in $n$, then by the inequalities \ref{ineq:Cjlk1}$-$\ref{eq:
key eq}, we are done. We explain now the case where the squares
$A_k$ are not bounded in $n$, that is, $\max_{1\leq k\leq
N(n)}|\widetilde{z}^{\,0}_{k}-\widetilde{z}^{\,1}_{k}|\rightarrow+\infty$
as $n\rightarrow+\infty$, is also true. Inequality \ref{ineq:N(n)}
shows that, for every fixed $j$ and $l$, it is sufficient to
consider at most $C_2'N(n)$ elements of $C'_{j,l}$. We can write
$\widetilde{I}'^{(n)}(\widetilde{z})$ as the concatenation of $N(n)$
sub-paths $\widetilde{J}_k(\widetilde{z})\, (k=1,\cdots,N(n))$.
Obviously, for every $\widetilde{z}\in C'_{j,l}$, we have
\begin{eqnarray}\label{eq:L_m F^n}
  \qquad\left|L_m(\widetilde{F}^n;\widetilde{a},\widetilde{b},z_*)
  \right|&\leq&\left|\widetilde{\gamma}\wedge\prod_{j,l,k}\left(\prod_{\widetilde{z}
\in
C'_{j,l,k}}\widetilde{J}_k(\widetilde{z})\right)\right|+\left|\widetilde{\gamma}\wedge
\prod_{j,l}\left(\prod_{\widetilde{z} \in
C'_{j,l}}\widetilde{I}'^{(n)}(\widetilde{z})\right)\right|.
\end{eqnarray}

We know that the value of the first part of the right hand side of
Inequality \ref{eq:L_m F^n} is less than $2C_1C_2N(n)\tau_m(n,z_*)$.
Hence, to explore the relation of the bound of
$\left|L_m(\widetilde{F}^n;\widetilde{a},\widetilde{b},z_*)
  \right|$ and the power of $N(n)$, we can suppose that the path
  $\widetilde{J}_k(\widetilde{z})$ never meets
  $A_k(\{r=0\}\cup\{r=1\})$ for every $k$ and $\widetilde{z}\in
  C'_{j,l}$. As the isotopies $\widetilde{I}_i$~($1\leq
i\leq s$) commutes with the covering transformations, 
it is easy to prove that there is a positive number $C'_1$ such that
$$\sum_{z\in C'_{j,l}}V(\widetilde{z})\leq C'_1C'_2 N(n).$$ Hence,
\begin{eqnarray*}
  \left|\widetilde{\gamma}\wedge
\left(\prod_{j,l}\prod_{\widetilde{z} \in
C'_{j,l}}\widetilde{I}'^{(n)}(\widetilde{z})\right)\right|
  &=& \left|\widetilde{\gamma}\wedge
\left(\prod_{j,l}\prod_{\widetilde{z} \in
C'_{j,l}}\prod_{k=1}^{N(n)}\widetilde{J}_k(\widetilde{z})\right)\right|\\
  &&\\
  &\leq&C_1'C_2'N(n)\tau_m(n,z_*).
\end{eqnarray*}

Therefore, there exists $C''_1>\max\{C_1,C'_1\}$, where $C_1$
defined in Formula \ref{eq:C_1}, such that
$$\left|L_m(\widetilde{F}^n;\widetilde{a},\widetilde{b},z_*)
  \right|\leq C''_1(2C_2+C_2')N(n)\tau_m(n,z_*).$$

It implies that
\begin{equation*}
0<\left|i(\widetilde{F};\widetilde{a},\widetilde{b},z_*)\right|\leq
c'_0\frac{N(n)}{n}\quad \text{for every}\quad n\geq1,
\end{equation*} where $c_0'=C''_1(2C_2+C_2')$.
We deduce that $\|F^n\|_{\mathscr{G}}\succeq n.$ Therefore,
$\|F^n\|_{\mathscr{G}}\thicksim n$, which completes the proof.
\end{proof}

\subsection{Counter-examples}

We construct two examples for $\mathrm{Supp}(\mu)\neq M$.

\begin{exem}\label{exem:T2 and supp(u)notM}
Consider the following smooth identity isotopy on $\mathbb{R}^2$:
$\widetilde{I}=(\widetilde{F}_t)_{t\in[0,1]}:
(x,y)\mapsto(x+\frac{t}{10}\cos(2\pi y),y+\frac{t}{10}\sin(2\pi
y))$. It induces an identity smooth isotopy $I=(F_t)_{t\in[0,1]}$ on
$\mathbb{T}^2$. Let $\mu$ have a constant density on
$\{(x,y)\in\mathbb{T}^2\mid y=0\,\,\mathrm{or}\,\, y=\frac{1}{2}\}$
and vanish on elsewhere. Obviously, $\rho_{\mathbb{T}^2,I}(\mu)=0$
but $\mathrm{Fix}_{\mathrm{Cont},I}(F_1)=\emptyset$.
\end{exem}

In the case where $g=1$ and $\mathrm{Supp}(\mu)\neq M$, Example
\ref{exem:T2 and supp(u)notM} tells us that there is even no sense
to talk about the action function in some special cases. The
following example, belonging to Le Calvez, implies that Theorem
\ref{prop:F is not constant if the contractible fixed points is
finite} is no longer true in the case
where $g>1$ and $\mathrm{Supp}(\mu)\neq M$. 

\begin{exem}[\cite{P1}, page 73]\label{exem:M and supp(u)notM}
Let $M$ be a closed orientated surface with $g=2$. Le Calvez
constructed a smooth identity isotopy $I=(F_t)_{t\in[0,1]}$, two
points $a$ and $b$ ($z_3$ and $z_4$ respectively in his example)
which are the only two contractible fixed points of $F_1$, and a
point $c$ ($z_{1,3}'$ in his example) which is a periodic point of
$F_1$ with period 20 and an arc $\prod_{0\leq i\leq 19}I(F_1^i(c))$
which is homologic to zero in $M\setminus\{a,b\}$.

We now define the measure
$\mu=\frac{1}{20}\sum_{i=0}^{19}\delta_{F_1^i(c)}$, where $\delta_z$
is the Dirac measure. The fact that the arc $\prod_{0\leq i\leq
19}I(F_1^i(c))$ is homologic to zero in $M\setminus\{a,b\}$ means
that $\rho_{M,I}(\mu)=0$ and that $I_\mu(\widetilde{F};a,b)=0$,
i.e., the action function is constant.
\end{exem}
\end{document}